\documentclass[a4paper, 11pt, oneside]{article}

\usepackage[utf8]{inputenc}
\usepackage{cmap}

\usepackage[english]{babel}
\usepackage[warn]{mathtext}
\usepackage[T2A]{fontenc}
\usepackage{microtype}

\usepackage[a4paper, margin=2cm]{geometry}

\usepackage[unicode=true, colorlinks=true, linkcolor=blue, citecolor=Green]{hyperref}
\usepackage[usenames,dvipsnames]{color} 

\usepackage{amssymb,amsfonts,amsmath,amsthm,mathtools}

\usepackage{mathrsfs}

\usepackage[inline]{enumitem}
\usepackage{indentfirst}
\usepackage{titlesec}

\DeclareMathOperator{\Dom}{Dom}
\DeclareMathOperator{\meas}{meas}
\DeclareMathOperator{\rank}{rank}
\DeclareMathOperator{\Ran}{Ran}
\DeclareMathOperator{\Ker}{Ker}
\DeclareMathOperator{\clos}{clos}
\DeclareMathOperator{\esssup}{ess-sup}
\DeclareMathOperator{\const}{const}
\renewcommand{\Re}{\mathop{\mathrm{Re}}\nolimits}
\renewcommand{\Im}{\mathop{\mathrm{Im}}\nolimits}

\newcommand*{\hm}[1]{#1\nobreak\discretionary{}	{\hbox{\mathsurround=0pt \ensuremath{#1}}}{}}

\titleformat{\part}[hang]{ \large\bfseries\scshape}{Chapter~\thepart.}{0.5ex}{\centering}[]
\titleformat{\section}[hang]{ \scshape\bfseries}{\S\thesection.}{0.5ex}{\centering}[]
\titleformat{\subsection}[runin]{\bfseries}{\thesubsection.}{0.4ex}{}[.]

\numberwithin{equation}{section}

\theoremstyle{plain}
\newtheorem{thrm}{\scshape Theorem}[section]
\newtheorem{proposition}[thrm]{\scshape Proposition}
\newtheorem{lemma}[thrm]{\scshape Lemma}
\newtheorem{corollary}[thrm]{\scshape Corollary}
\newtheorem{remark}[thrm]{\scshape Remark}
\newtheorem{condition}[thrm]{\scshape Condition}

\theoremstyle{definition}
\newtheorem{example}[thrm]{\scshape Example}

\newtheoremstyle{break}
{}{}
{\itshape}{}
{\bfseries}{}
{\newline}{}
\theoremstyle{break}

\theoremstyle{break}

\providecommand{\keywords}[1]{\textbf{{Keywords:}} #1.}

\title{Operator error estimates for homogenization of the nonstationary Schr\"{o}dinger-type equations: sharpness of the results\footnote{Supported by Young Russian Mathematics award and Ministry of Science and Higher Education of the Russian Federation, agreement \textnumero\,075-15-2019-1619.}}
\date{}
\author{Mark Dorodnyi\footnote{Leonhard Euler International Mathematical Institute, St.~Petersburg State University, 14th Line V.O., 29B, St.~Petersburg, 199178, Russia; e-mail: \texttt{mdorodni@yandex.ru}.}}

\begin{document}

\setlength{\abovedisplayskip}{4pt}
\setlength{\abovedisplayshortskip}{4pt}
\setlength{\belowdisplayskip}{4pt}
\setlength{\belowdisplayshortskip}{4pt}
	
\clubpenalty = 10000
\widowpenalty = 10000
   
\maketitle
\begin{abstract}
	\noindent
	In $L_2 (\mathbb{R}^d; \mathbb{C}^n)$, we consider a selfadjoint matrix strongly elliptic second order differential operator $\mathcal{A}_\varepsilon$ with periodic coefficients depending on $\mathbf{x}/\varepsilon$.  We find approximations of the exponential $e^{-i \tau \mathcal{A}_\varepsilon}$, $\tau \in \mathbb{R}$, for small $\varepsilon$ in the ($H^s \to L_2$)-operator norm with suitable $s$. The sharpness of the error estimates with respect to $\tau$ is discussed. The results are applied to study the behavior of the solution $\mathbf{u}_\varepsilon$ of the Cauchy problem for the Schr\"{o}dinger-type equation $i\partial_{\tau} \mathbf{u}_\varepsilon = \mathcal{A}_\varepsilon \mathbf{u}_\varepsilon + \mathbf{F}$. 
\end{abstract}


\part*{Introduction}
The paper concerns homogenization for periodic differential operators (DOs). A broad literature is devoted to homogenization problems in the small period limit; first of all, we mention the books~\cite{BeLP,BaPa,ZhKO}. For homogenization problems in $\mathbb{R}^d$, one of the methods is the spectral approach based on the Floquet--Bloch theory; see, e.~g.,~\cite[Chapter~4]{BeLP}, \cite[Chapter~2]{ZhKO}, \cite{Se1981}, \cite{COrVa}, \cite{APi}.

\subsection{The class of operators}
Let $\Gamma$ be a lattice in $\mathbb{R}^d$, and let $\Omega$ be the elementary cell of the lattice~$\Gamma$. For $\Gamma$-periodic functions in $\mathbb{R}^d$, we denote $\varphi^{\varepsilon}(\mathbf{x}) \coloneqq \varphi(\varepsilon^{-1}\mathbf{x})$, $\varepsilon>0$. In $L_2 (\mathbb{R}^d; \mathbb{C}^n)$, we consider self-adjoint elliptic matrix DOs of the following form
\begin{equation}
\label{intro_A_eps}
\mathcal{A}_{\varepsilon} = (f^{\varepsilon}(\mathbf{x}))^* b(\mathbf{D})^* g^{\varepsilon}(\mathbf{x}) b(\mathbf{D}) f^{\varepsilon}(\mathbf{x}).
\end{equation}
Here $b(\mathbf{D})$ is a homogeneous first order matrix DO with constant coefficients. We assume that the symbol $b(\boldsymbol{\xi})$ is an ($m \times n$)-matrix of rank $n$ ($m \ge n$). Next, $g(\mathbf{x})$ is a $\Gamma$-periodic bounded and uniformly positive definite ($m \times m$)-matrix-valued function and $f(\mathbf{x})$ is a $\Gamma$-periodic bounded together with its inverse ($n \times n$)-matrix-valued function.

It is convenient to start with a simpler class of operators
\begin{equation}
\label{intro_hatA_eps}
\widehat{\mathcal{A}}_{\varepsilon} = b(\mathbf{D})^* g^{\varepsilon}(\mathbf{x}) b(\mathbf{D})
\end{equation}
corresponding to the case where $f = \mathbf{1}_n$. Many operators of mathematical physics can be represented in the form~(\ref{intro_A_eps}) or~(\ref{intro_hatA_eps}); see, e.~g.,~\cite[Chapter~4]{BSu2005-2}. The simplest example is the acoustics operator $\widehat{\mathcal{A}}_{\varepsilon} = \mathbf{D}^* g^{\varepsilon}(\mathbf{x}) \mathbf{D} = -\operatorname{div} g^{\varepsilon}(\mathbf{x}) \nabla$.

\subsection{Survey} 
In 2001 M.~Birman and T.~Suslina (see~\cite{BSu2001}) suggested an operator-theoretic (spectral) approach to homogenization problems in $\mathbb{R}^d$, based on the scaling transformation, the Floquet--Bloch theory, and the analytic perturbation theory. 
With the help of this method, the so-called~\emph{operator error estimates} for homogenization problems were obtained.

In the case of elliptic and parabolic problems this approach was developed in detail: see \cite{BSu2003, BSu2005, BSu2005-2, BSu2006, Su2004, Su2007, Su, V, VSu2011, VSu2012}.

A different approach to operator error estimates for elliptic and parabolic problems (the \textquotedblleft shift method\textquotedblright) was suggested by V.~Zhikov and S.~Pastukhova: see~\cite{Zh2006, ZhPas2005, ZhPas2006} and survey~\cite{ZhPas2016}.

The operator error estimates for nonstationary Schr\"{o}dinger-type and hyperbolic equations have been studied to a lesser extent. The papers~\cite{BSu2008, Su2016, Su2017, DSu2018, M2017, M2019a} were devoted to such problems; see also~\cite{D2019} and \cite{M2019b}, where a wider class of operators with the lower order terms was considered. In operator terms, the behavior of the operator-valued functions $e^{-i \tau \widehat{\mathcal{A}}_{\varepsilon}}$, $\cos(\tau \widehat{\mathcal{A}}_{\varepsilon}^{1/2})$, and $\widehat{\mathcal{A}}_{\varepsilon}^{-1/2} \sin(\tau \widehat{\mathcal{A}}_{\varepsilon}^{1/2})$ (where $\tau \in \mathbb{R}$) for small $\varepsilon$ was studied. Let us dwell on the results for the nonstationary Schr\"{o}dinger-type equations.
In~\cite{BSu2008}, the following estimate was obtained:
\begin{equation}
\label{intro_exp_est_1}
\| e^{-i \tau \widehat{\mathcal{A}}_{\varepsilon}} - e^{-i \tau \widehat{\mathcal{A}}^0} \|_{H^3 (\mathbb{R}^d) \to L_2 (\mathbb{R}^d)} \le C (1 + |\tau|)\varepsilon.
\end{equation}
Here $\widehat{\mathcal{A}}^0 = b(\mathbf{D})^* g^0 b(\mathbf{D})$ is the \emph{effective operator} with the constant \emph{effective matrix} $g^0$. Next, in~\cite{Su2017} (see also~\cite{Su2016}) it was shown that, in the general case, this estimate is sharp with respect to the type of the operator norm. On the other hand, under some additional assumptions (formulated in the spectral terms near the lower edge of the spectrum) this result was improved:
\begin{equation}
\label{intro_exp_est_2}
\| e^{-i \tau \widehat{\mathcal{A}}_{\varepsilon}} - e^{-i \tau \widehat{\mathcal{A}}^0} \|_{H^{2} (\mathbb{R}^d) \to L_2 (\mathbb{R}^d)} \le {C}(1+|\tau|) \varepsilon.
\end{equation}

\subsection{Main results of the paper}
The present paper is devoted to error estimates for the operator exponential; a special attention is paid to the dependence of the estimates on time. We show that, in the general case, the factor 
$(1+ |\tau|)$ in~(\ref{intro_exp_est_1}) cannot be replaced by $(1 + |\tau|^\alpha)$ with $\alpha <1$. On the other hand, we prove that estimate~(\ref{intro_exp_est_2}) (which holds under some additional assumptions) can be improved:
\begin{equation*}
\| e^{-i \tau \widehat{\mathcal{A}}_{\varepsilon}} - e^{-i \tau \widehat{\mathcal{A}}^0} \|_{H^{2} (\mathbb{R}^d) \to L_2 (\mathbb{R}^d)} \le {C}(1+|\tau|^{1/2}) \varepsilon.
\end{equation*}
This result allows us to obtain qualified estimates for large time $\tau = O(\varepsilon^{-\alpha})$ with $\alpha < 2$.
Analogs of these results are obtained also for the more general operator~(\ref{intro_A_eps}). It turns out that it is convenient to study the operator $f^\varepsilon e^{-i \tau \mathcal{A}_{\varepsilon}} (f^\varepsilon)^{-1}$ (the operator exponential sandwiched between rapidly oscillating factors). 

The results given in the operator terms are applied to study the behavior of the solution $\mathbf{u}_\varepsilon (\mathbf{x}, \tau)$, $\mathbf{x} \in \mathbb{R}^d$, $\tau \in \mathbb{R}$, of the problem
\begin{equation}
\label{intro_Ahat_Cauchy_problem}
\left\{
\begin{aligned}
&i \frac{\partial \mathbf{u}_\varepsilon (\mathbf{x}, \tau)}{\partial \tau} =  (\widehat{\mathcal{A}}_{\varepsilon} \mathbf{u}_\varepsilon) (\mathbf{x}, \tau) + \mathbf{F} (\mathbf{x}, \tau), \\
& \mathbf{u}_\varepsilon (\mathbf{x}, 0) = \boldsymbol{\phi} (\mathbf{x}), 
\end{aligned}
\right.
\end{equation}
and also a more general problem with the operator $\mathcal{A}_{\varepsilon}$.

\subsection{Method}
The results are obtained with the help of the operator-theoretic approach. The scaling transformation reduces investigation of the difference of exponentials under the norm sign in~(\ref{intro_exp_est_1}) to studying the difference $e^{-i \tau \varepsilon^{-2} \widehat{\mathcal{A}}} - e^{-i \tau \varepsilon^{-2} \widehat{\mathcal{A}}^0}$, where $\widehat{\mathcal{A}} = b(\mathbf{D})^* g(\mathbf{x}) b(\mathbf{D})$. Next, with the help of the unitary Gelfand transformation, the operator~$\widehat{\mathcal{A}}$ expands into the direct integral of the operators $\widehat{\mathcal{A}}(\mathbf{k})$ depending on the quasimomentum $\mathbf{k}$ and acting in the space $L_2 (\Omega; \mathbb{C}^n)$. According to~\cite{BSu2003}, we distinguish the one-dimensional parameter $t = |\mathbf{k}|$ and consider the family $\widehat{\mathcal{A}}(\mathbf{k})$ as a quadratic operator pencil with respect to the parameter $t$. Here, a good deal of constructions can be done in the framework of an abstract operator-theoretic setting. In the abstract scheme, the operator family $A(t)$ acting in some Hilbert space $\mathfrak{H}$ and admitting a factorization of the form $A(t) = X(t)^* X(t)$, where $X(t) = X_0 + t X_1$, is considered.

\subsection{The plan of the paper}
The paper consists of three chapters. Chapter~\ref{abstr_part} (\S\S1--3) contains necessary abstract operator-theoretic material. In Chapter~\ref{L2_operators_part} (\S\S4--9) periodic DOs acting in $L_2(\mathbb{R}^d;\mathbb{C}^n)$ are studied. Chapter~\ref{main_results_part} (\S\S10--12) is devoted to homogenization problems for nonstationary Schr\"{o}dinger-type equations. In~\S\ref{main_results_exp_section} the main results of the paper in operator terms are obtained. Next, in~\S\ref{main_results_Cauchy_section} these results are applied to homogenization of the Cauchy problem~(\ref{intro_Ahat_Cauchy_problem}) and a more general problem with the operator $\mathcal{A}_{\varepsilon}$. \S\ref{appl_section} is devoted to applications of the general results to particular equations.

\subsection{Notation}
Let $\mathfrak{H}$ and $\mathfrak{H}_{*}$ be complex separable Hilbert spaces. The symbols $(\cdot, \cdot)_{\mathfrak{H}}$ and $ \| \cdot \|_{\mathfrak{H}}$ denote the inner product and the norm in $\mathfrak{H}$. The symbol $\| \cdot \|_{\mathfrak{H} \to \mathfrak{H}_*}$ stands for the norm of a bounded linear operator from $\mathfrak{H}$ to $\mathfrak{H}_{*}$. Sometimes we omit the indices if this does not lead to confusion. By $I = I_{\mathfrak{H}}$ we denote the identity operator in $\mathfrak{H}$. If $A \colon \mathfrak{H} \to \mathfrak{H}_*$ is a linear operator, then $\Dom A$ and $\Ker A$ stand for its domain and kernel, respectively. If $\mathfrak{N}$ is a subspace in $\mathfrak{H}$, then $\mathfrak{N}^{\perp} \coloneqq \mathfrak{H} \ominus \mathfrak{N}$. If $P$ is the orthogonal projection of $\mathfrak{H}$ onto $\mathfrak{N}$, then $P^{\perp}$ is the orthogonal projection of $\mathfrak{H}$ onto $\mathfrak{N}^{\perp}$. 

The symbols $\left< \cdot, \cdot \right>$ and $| \cdot |$ stand for the standard inner product and the norm in $\mathbb{C}^n$; $\boldsymbol{1}_n$ is the unit ($n \times n$)-matrix. If $a$ is an $(m \times n)$-matrix, then $a^*$ stands for the adjoint matrix.
Next, $\mathbf{x} = (x_1, \ldots , x_d) \in \mathbb{R}^d$, $i D_j = \frac{\partial}{\partial x_j}$, $j = 1,\ldots, d$, $\mathbf{D} = -i \nabla = (D_1, \ldots, D_d)$. 

The $L_p$-classes of $\mathbb{C}^n$-valued functions in a domain $\mathcal{O} \subset \mathbb{R}^d$ are denoted by $L_p (\mathcal{O}; \mathbb{C}^n)$, $1 \le p \le \infty$. The Sobolev classes of $\mathbb{C}^n$-valued functions in a domain $\mathcal{O} \subset \mathbb{R}^d$ of order $s$ and integrability index $p$ are denoted by $W^s_p (\mathcal{O}; \mathbb{C}^n)$. For $p=2$ we use the notation $H^s (\mathcal{O}; \mathbb{C}^n)$, $s \in \mathbb{R}$. If $n=1$, we write simply $L_p (\mathcal{O})$, $W^s_p (\mathcal{O})$, $H^s (\mathcal{O})$, etc., but sometimes we use such abbreviated notation also for spaces of vector-valued or matrix-valued functions.

Various constants in estimates are denoted by $C$, $c$, $\mathcal{C}$, $\mathfrak{C}$ (probably, with indices and marks).

\subsection{Acknowledgement}
The author is grateful to T.~A.~Suslina for helpful discussions and advices.

\part{Abstract operator-theoretic scheme}
\label{abstr_part}

\section{Quadratic operator pencils}

\subsection{The operators $X(t)$ and $A(t)$}  
\label{abstr_X_A_section}

Let $\mathfrak{H}$ and $\mathfrak{H}_{*}$ be complex separable Hilbert spaces. Suppose that $X_{0}: \mathfrak{H} \to \mathfrak{H}_{*}$ is a densely defined and closed operator, and $X_{1} \colon \mathfrak{H} \to \mathfrak{H}_{*}$ is a bounded operator. On the domain $\Dom X_0$, we introduce the operator $X(t) \coloneqq X_0 + t X_1$, $t \in \mathbb{R}$. Consider
the family of selfadjoint (and nonnegative) operators $A(t) \coloneqq X(t)^*  X(t)$ in $\mathfrak{H}$. The operator $A(t)$ is
generated by the closed quadratic form $\| X(t) u \|^{2}_{\mathfrak{H}_*}$, $u \in \Dom X_0$. Denote $A_0 \coloneqq A(0)$, 
$\mathfrak{N} \coloneqq \Ker  A_0 = \Ker X_0$, $\mathfrak{N}_{*} \coloneqq \Ker X^*_0$. We impose the following condition.

\begin{condition}
	The point $\lambda_0 = 0$ is an isolated point in the spectrum of $A_0$, and $0 < n \coloneqq \operatorname{dim} \mathfrak{N} < \infty$, $n \le n_* \coloneqq 
	\operatorname{dim}  \mathfrak{N}_* \le \infty$.
\end{condition}

Denote by $d^0$ \emph{the distance from the point $\lambda_0 = 0$ to the rest of the spectrum of $A_0$}. Let $P$ and $P_*$ be the orthogonal projections of $\mathfrak{H}$ onto $\mathfrak{N}$ and of $\mathfrak{H}_*$ onto  $\mathfrak{N}_*$, respectively. Denote by $F(t;[a, b])$ the spectral projection of $A(t)$ for the interval $[a,b]$, and put $\mathfrak{F} (t;[a,b]) \coloneqq F(t;[a, b]) \mathfrak{H}$. \emph{We fix a number $\delta > 0$ such that $8 \delta < d^0$}. We write $F(t)$ in
place of $F(t; [0,\delta])$ and $\mathfrak{F} (t)$ in
place of $\mathfrak{F}(t; [0, \delta])$. Next, we choose a number $t^0 > 0$ such that
\begin{equation}
\label{abstr_t0}
t^0 \le \delta^{1/2} \|X_1\|^{-1}.
\end{equation}
According to~\cite[Chapter~1, Proposition~1.2]{BSu2003}, $F(t; [0,\delta]) = F(t;[0, 3 \delta])$ and $\rank F(t; [0,\delta]) = n$ for $|t| \le t^0$.

\subsection{The operators $Z$, $R$, and $S$}
Now we introduce some operators appearing in the analytic perturbation theory considerations; see~\cite[Chapter~1, \S1]{BSu2003} and~\cite[\S1]{BSu2005}.

Let $\omega \in \mathfrak{N}$, and let $\psi = \psi(\omega) \in \Dom X_0 \cap \mathfrak{N}^{\perp}$ be a (weak) solution of the equation
\begin{equation*}
X^*_0 (X_0 \psi + X_1 \omega) = 0.
\end{equation*}
We define a bounded operator $Z \colon \mathfrak{H} \to \mathfrak{H}$ by the relation $Zu = \psi (P u)$, $u \in \mathfrak{H}$. Next, we define the operator $R\coloneqq X_0 Z + X_1 \colon \mathfrak{N}  \to \mathfrak{N}_*$. Another representation for $R$ is given by $R= P_*X_1 |_{\mathfrak{N}}$. According to~\cite[Chapter~1, Section~1.3]{BSu2003}, the operator $S \coloneqq R^* R \colon \mathfrak{N} \to \mathfrak{N}$ is called \emph{the spectral germ} of the operator family $A(t)$ at $t=0$. The germ can be represented as $S = P X^*_1 P_* X_1 |_{\mathfrak{N}}$. The spectral germ is said to be \emph{non-degenerate} if $\Ker S = \{0\}$.

\subsection{The operators $Z_2$ and $R_2$}
\label{abstr_Z2_R2_section}
We need to introduce the operators $Z_2$ and $R_2$ defined in~\cite[\S1]{VSu2011}.

Let $\omega \in \mathfrak{N}$, and let $\phi = \phi(\omega) \in \Dom X_0 \cap \mathfrak{N}^{\perp}$ be a (weak) solution of the equation
\begin{equation*}
X^*_0 (X_0 \phi + X_1 Z \omega) = -P^{\perp} X_1^* R \omega.
\end{equation*}
The right-hand side of this equation belongs to $\mathfrak{N}^{\perp} = \Ran X_0^*$, so the solvability condition is fulfilled.
We define an operator $Z_2 \colon \mathfrak{H} \to \mathfrak{H}$ by the relation $Z_2 u = \phi (P u)$, $u \in \mathfrak{H}$. Finally, let $R_2 \coloneqq X_0 Z_2 + X_1 Z \colon \mathfrak{N}  \to \mathfrak{H}_*$.

\subsection{The analytic branches of eigenvalues and eigenvectors of the operator $A(t)$} According to the general analytic perturbation theory (see~\cite{K}), for $|t| \le t^0$ there exist real-analytic functions $\lambda_l (t)$ (the branches of eigenvalues) and real-analytic $\mathfrak{H}$-valued functions $\varphi_l (t)$ (the branches of eigenvectors) such that $A(t) \varphi_l(t) = \lambda_l (t) \varphi_l(t)$, $l = 1, \ldots, n$, and the set $\varphi_l (t)$, $l = 1, \ldots, n$, forms \emph{an orthonormal basis} in $\mathfrak{F}(t)$. Moreover, for $|t| \le t_*$, where $0 < t_* \le t^0$ is \emph{sufficiently small}, we have the following convergent power series expansions:
\begin{align}
\label{abstr_A(t)_eigenvalues_series}
\lambda_l(t) &= \gamma_l t^2 + \mu_l t^3 + \nu_l t^4 + \ldots, & &\gamma_l \ge 0, \; \mu_l, \nu_l \in \mathbb{R}, \qquad  l = 1, \ldots, n, \\ 
\label{abstr_A(t)_eigenvectors_series}
\varphi_l (t) &= \omega_l + t \psi_l^{(1)} + t^2 \psi_l^{(2)} + \ldots, & &l = 1, \ldots, n.
\end{align}
We agree to use the numeration such that 
$\gamma_1 \le \gamma_2 \le \dots \le \gamma_n$. 
The elements $\omega_l =  \varphi_l (0)$, $l = 1, \ldots, n$, form an orthonormal basis in $\mathfrak{N}$.
In~\cite[Chapter~1, \S1]{BSu2003} and~\cite[\S1]{BSu2005} it was checked that $\widetilde{\omega}_l \coloneqq \psi_l^{(1)} - Z \omega_l \in \mathfrak{N}$, 
\begin{alignat}{2}
\label{abstr_S_eigenvectors}
S \omega_l = \gamma_l \omega_l&, \qquad  &&l = 1, \ldots, n,\\
\label{abstr_tilde_omega_rel}
(\widetilde{\omega}_j, \omega_k ) + (\omega_j, \widetilde{\omega}_k) = 0&, \qquad &&j, k = 1,\ldots, n.
\end{alignat}
Thus, \emph{the numbers $\gamma_l$ and the elements $\omega_l$ defined by~\emph{(\ref{abstr_A(t)_eigenvalues_series})} and~\emph{(\ref{abstr_A(t)_eigenvectors_series})} are eigenvalues and eigenvectors of the germ $S$}. We have $P = \sum_{l=1}^{n} (\cdot, \omega_l) \omega_l$ and $SP = \sum_{l=1}^{n} \gamma_l (\cdot, \omega_l) \omega_l$.

\subsection{Threshold approximations}
The following statements were obtained in~\cite[Chapter~1, Theorems~4.1 and~4.3]{BSu2003} and~\cite[Theorem~4.1]{BSu2005}. In what follows, we agree to denote by $\beta_j$ various absolute constants (which can be controlled explicitly) assuming that $\beta_j \ge 1$.
\begin{thrm}[\cite{BSu2003}] 
	Under the assumptions of Subsection~\emph{\ref{abstr_X_A_section}}, for $|t| \le t^0$ we have
	\begin{align}
	\label{abstr_F(t)_threshold}
	\| F(t) - P \| &\le C_1 |t|, & C_1 &= \beta_1 \delta^{-1/2} \| X_1 \|,\\
	\notag
	\| A(t)F(t) - t^2 SP \| &\le C_2 |t|^3, & C_2 &= \beta_2 \delta^{-1/2}\| X_1 \|^3.
	\end{align}
\end{thrm}
\begin{thrm}[\cite{BSu2005}]
	\label{abstr_threshold_approx_thrm_2} 
	Under the assumptions of Subsection~\emph{\ref{abstr_X_A_section}}, for $|t| \le t^0$ we have
	\begin{equation*}
	A(t) F(t) = t^2 SP + t^3 K + \Xi (t), \qquad \| \Xi (t) \| \le C_3 t^4, \qquad C_3 = \beta_3 \delta^{-1}\| X_1 \|^4.
	\end{equation*}
	The operator $K$ is represented as $K = K_0 + N = K_0 + N_0 + N_*$, where $K_0$ takes $\mathfrak{N}$ to $\mathfrak{N}^{\perp}$ and $\mathfrak{N}^{\perp}$ to $\mathfrak{N}$, while $N = N_0 + N_*$ takes $\mathfrak{N}$ to itself and $\mathfrak{N}^{\perp}$ to $\{ 0 \}$. In terms of the power series coefficients, the operators $K_0$, $N_0$, $N_*$ are given by $K_0 = \sum_{l=1}^{n} \gamma_l \left( (\cdot, Z \omega_l) \omega_l + (\cdot, \omega_l) Z \omega_l \right)$, 
	\begin{equation}
	\label{abstr_N0_N*}
	N_0 = \sum_{l=1}^{n} \mu_l (\cdot, \omega_l) \omega_l, \qquad N_* = \sum_{l=1}^{n} \gamma_l \left( (\cdot, \widetilde{\omega}_l) \omega_l + (\cdot, \omega_l) \widetilde{\omega}_l\right).
	\end{equation}
	In the invariant terms, we have $K_0 = Z S P + S P Z^*$ and $N = Z^*X_1^* R P + (RP)^* X_1 Z$.
\end{thrm}

\begin{remark}
	\label{abstr_N_remark}	
	\begin{enumerate*}[label=\emph{\arabic*$^{\circ}.$}, ref=\arabic*$^{\circ}$]
	\item If $Z = 0$, then $K_0 = 0$, $N = 0$, and $K = 0$. 
	\item In the basis $\{\omega_l\}_{l=1}^n$ the operators $N$, $N_0$, $N_*$ \emph{(}restricted to the subspace $\mathfrak{N}$\emph{)} are represented by matrices of size $n \times n$. The operator $N_0$ is diagonal: $(N_0 \omega_j, \omega_k ) = \mu_j \delta_{jk}$, $j, k = 1, \ldots ,n$. The matrix entries of $N_*$ are given by $(N_* \omega_j, \omega_k) = \gamma_k (\omega_j, \widetilde{\omega}_k) + \gamma_j (\widetilde{\omega}_j, \omega_k ) = ( \gamma_j - \gamma_k)(\widetilde{\omega}_j, \omega_k )$, $j, k = 1,\ldots, n$. So, the diagonal elements of $N_*$ are equal to zero. Moreover, $(N_* \omega_j, \omega_k) = 0$ if $\gamma_j = \gamma_k$.
	\item If $n = 1$, then $N_* = 0$ and $N = N_0$.
	\end{enumerate*}
\end{remark}

\subsection{The nondegeneracy condition}
Below we impose the following additional condition.
\begin{condition}
	\label{abstr_nondegeneracy_cond}
	There exists a constant $c_* > 0$ such that $A(t) \ge c_* t^2 I$ for $|t| \le t^0$.
\end{condition}
From Condition~\ref{abstr_nondegeneracy_cond} it follows that $\lambda_l (t) \ge c_* t^2$, $l = 1, \ldots, n$, for $|t| \le t^0$. By~(\ref{abstr_A(t)_eigenvalues_series}), this implies $\gamma_l \ge c_* > 0$, $l= 1, \ldots, n$, i.e., the spectral germ is nondegenerate:
\begin{equation}
\label{abstr_S_nondegeneracy}
S \ge c_* I_{\mathfrak{N}}.
\end{equation}

\subsection{The clusters of eigenvalues of $A(t)$}
\label{abstr_cluster_section}

The content of this subsection is borrowed from~\cite[Section~2]{Su2017} and concerns the case where $n \ge 2$.

Suppose that Condition~\ref{abstr_nondegeneracy_cond} is satisfied. Now, it is convenient to change the notation tracing the multiplicities of the eigenvalues of the operator $S$. Let $p$ be the number of  different eigenvalues of the germ $S$. We enumerate these eigenvalues in the increasing order and denote them by $\gamma^{\circ}_j$, $j = 1, \ldots, p$. Let $k_1, \ldots, k_p$ be their multiplicities (obviously, $k_1 + \ldots + k_p = n$). Denote $\mathfrak{N}_j \coloneqq \Ker(S - \gamma^{\circ}_j I_\mathfrak{N})$, $j = 1 ,\ldots, p$. Then $\mathfrak{N} = \sum_{j=1}^{p} \oplus \mathfrak{N}_j$.
Let $P_j$ be the orthogonal projection of $\mathfrak{H}$ onto $\mathfrak{N}_j$. Then $P = \sum_{j=1}^{p} P_j$, and $P_j P_l = 0$ for $j \ne l$.

\begin{remark}
	By Remark \emph{\ref{abstr_N_remark}}, we have $P_j N_* P_j =0$ and $P_l N_0 P_j=0$ for $l\ne j$. Hence, the operators $N_{0}$ and $N_{*}$ admit the invariant representions: 
	\begin{equation}
	\label{abstr_N0_N*_invar_repr}
	N_{0} = \sum_{j=1}^{p} P_j N P_j, \qquad N_{*} = \sum_{\substack{1 \le l,j \le p \\ j \ne l}} P_l N P_j.
	\end{equation}
\end{remark}

We divide the first $n$ eigenvalues of the operator $A(t)$ in $p$ clusters for $|t| \le t^0$; the $j$-th cluster consists of the eigenvalues $\lambda_l (t)$,
$l=i, \dots, i+ k_j -1$, where $i=i(j) = k_1 + \dots + k_{j-1}+1$. 

For each pair of indices $(j, l)$, $1 \le j,l \le p$, $j \ne l$, we denote $c^{\circ}_{jl} \coloneqq \min \{c_*, n^{-1} |\gamma^{\circ}_l - \gamma^{\circ}_j|\}$. Clearly, there exists a number $i_0 = i_0 (j,l)$, where $j \le i_0 \le l-1$ if $j < l$ and $l \le i_0 \le j-1$ if $l < j$, such that $\gamma^{\circ}_{i_0 + 1} -  \gamma^{\circ}_{i_0} \ge c^{\circ}_{jl}$. It means that on the interval between $\gamma^{\circ}_j$ and $\gamma^{\circ}_l$ there is a gap in the spectrum of $S$ of length at least $c^{\circ}_{jl}$. If such $i_0$ is not unique, we agree to take the minimal possible $i_0$ (for definiteness). Next, we choose a number $t^{00}_{jl} \le t^0$ such that
\begin{equation*}
t^{00}_{jl} \le (4C_2)^{-1} c^{\circ}_{jl} = (4 \beta_2)^{-1} \delta^{1/2} \|X_1\|^{-3 } c^{\circ}_{jl}.
\end{equation*}
Let $\Delta^{(1)}_{jl} \coloneqq [\gamma^{\circ}_1 - c^{\circ}_{jl}/4, \gamma^{\circ}_{i_0} + c^{\circ}_{jl}/4]$ and $\Delta^{(2)}_{jl} \coloneqq [\gamma^{\circ}_{i_0+1} - c^{\circ}_{jl}/4, \gamma^{\circ}_p + c^{\circ}_{jl}/4]$. The distance between the segments $\Delta^{(1)}_{jl}$ and $\Delta^{(2)}_{jl}$ is at least $c^{\circ}_{jl}/2$. As was shown in~\cite[Section~2]{Su2017}, for $|t| \le  t^{00}_{jl}$ the operator $A(t)$ has exactly $k_1 + \ldots + k_{i_0}$ eigenvalues (counted with multiplicities) in the segment $t^2 \Delta^{(1)}_{jl}$ and exactly $k_{i_0+1} + \ldots + k_p$ eigenvalues in the segment $t^2 \Delta^{(2)}_{jl}$.

\subsection{The coefficients $\nu_l$, $l=1, \ldots,n$}
\label{abstr_nu_section}
We need to establish a relationship between the coefficients $\nu_l$, $l=1, \ldots,n$, and some eigenvalue problem. 

In~\cite[(1.34), (1.37)]{VSu2011}, it was checked that
\begin{align}
\notag
\psi_l^{(2)} - Z \widetilde{\omega}_l - Z_2 \omega_l \eqqcolon \widetilde{\omega}^{(2)}_l \in \mathfrak{N}&, & &l = 1, \ldots, n,\\
\label{abstr_tilde_omega^(2)_rel}
(\widetilde{\omega}^{(2)}_l, \omega_k) + (Z \omega_l, Z \omega_k) + (\widetilde{\omega}_l, \widetilde{\omega}_k) + (\omega_l, \widetilde{\omega}^{(2)}_k) = 0&, & &l, k = 1, \ldots, n.
\end{align}  
Next, by~\cite[(2.47), the formula below~(2.46)]{VSu2011}, we have
\begin{multline}
\label{abstr_nu_rel_1}
(N_1 \omega_l, \omega_k) - \mu_l (\widetilde{\omega}_l, \omega_k) - \mu_k (\omega_l, \widetilde{\omega}_k) - \gamma_l (\widetilde{\omega}^{(2)}_l, \omega_k) - \gamma_k (\omega_l, \widetilde{\omega}^{(2)}_k) - (S \widetilde{\omega}_l, \widetilde{\omega}_k) = \nu_l \delta_{lk},\\
l, k = 1, \ldots, n,
\end{multline}
where $N_1 \coloneqq N_1^0 - Z^*Z S P - S P Z^*Z$, $N_1^0 \coloneqq Z_2^* X_1^* R P + (RP)^* X_1 Z_2 + R_2^* R_2 P$.

Let $\gamma_q^{\circ}$ be the $q$-th eigenvalue of problem~(\ref{abstr_S_eigenvectors}) of multiplicity $k_q$ (i.e. $\gamma_q^{\circ} = \gamma_i = \ldots = \gamma_{i+k_q-1}$ for $i = i(q) = k_1+\ldots+k_{q-1}+1$). Consider the eigenvalue problem (see Remark~\ref{abstr_N_remark}) 
\begin{equation}
\label{abstr_N_eigenvalues}
P_q N \omega_l = \mu_l \omega_l, \qquad l = i, \ldots, i+k_q-1.
\end{equation}
Assume that $\mu_l$, $l = i, \ldots, i+k_q-1$, are enumerated in the increasing order. Let $p'(q)$ be the number of different eigenvalues of problem \eqref{abstr_N_eigenvalues}
and denote by $k_{1,q}, \ldots, k_{p'(q),q}$ their multiplicities (of course, $k_{1,q} + \ldots + k_{p'(q),q} = k_q$). We also change the notation and denote by $\mu^{\circ}_{j,q}$, $j = 1, \ldots, p'(q)$, the different eigenvalues of problem \eqref{abstr_N_eigenvalues}, enumerating them in the increasing order. Denote $\mathfrak{N}_{j,q} \coloneqq \Ker(P_qN|_{\mathfrak{N}_q} - \mu^{\circ}_{j,q} I_{\mathfrak{N}_q})$, $j = 1 ,\ldots, p'(q)$. Then $\mathfrak{N}_q = \sum_{j=1}^{p'(q)} \oplus \mathfrak{N}_{j,q}$. Let $P_{j,q}$ be the orthogonal projection of $\mathfrak{H}$ onto $\mathfrak{N}_{j,q}$. Then $P_q = \sum_{j=1}^{p'(q)} P_{j,q}$ and $P_{j,q} P_{r,q} = 0$ for $j \ne r$.  

Let $\mu^{\circ}_{q',q}$ be the $q'$-th eigenvalue of problem~(\ref{abstr_N_eigenvalues}) of multiplicity $k_{q',q}$, i.e., $\mu^{\circ}_{q',q} = \mu_{i'} = \ldots = \mu_{i'+k_{q',q}-1}$, where $i' = i'(q',q) =i(q)+k_{1,q}+\ldots+k_{q'-1,q}$. Using~(\ref{abstr_tilde_omega_rel}), (\ref{abstr_tilde_omega^(2)_rel}) and taking into account that $\gamma_l = \gamma_k = \gamma^{\circ}_q$, $\mu_l = \mu_k = \mu^{\circ}_{q',q}$, $l,k = i', \ldots, i'+k_{q',q}-1$, from~(\ref{abstr_nu_rel_1}) we deduce
\begin{equation}
\label{abstr_nu_rel_2}
(N_1 \omega_l, \omega_k) +  \gamma_l (Z \omega_l, Z \omega_k) + \gamma_l (\widetilde{\omega}_l, \widetilde{\omega}_k) - (S \widetilde{\omega}_l, \widetilde{\omega}_k) = \nu_l \delta_{lk}, \qquad l,k = i', \ldots, i'+k_{q',q}-1.
\end{equation}
Next, by virtue of Remark~\ref{abstr_N_remark}, we have
\begin{multline*}
\gamma_l (\widetilde{\omega}_l, \widetilde{\omega}_k) - (S \widetilde{\omega}_l, \widetilde{\omega}_k) =  \sum_{l'=1}^{n} (\gamma_l - \gamma_{l'}) (\widetilde{\omega}_l,\omega_{l'}) (\omega_{l'}, \widetilde{\omega}_k) \\= \sum_{\substack{l'\in\{1,\ldots,n\}\\l' \ne i, \ldots, i+k_q-1}} \frac{(N \omega_{l}, \omega_{l'}) ( \omega_{l'}, N\omega_{k})}{\gamma^{\circ}_q - \gamma_{l'}}   =  \sum_{\substack{j\in\{1,\ldots,p\} \\	j \ne q}} \frac{( P_{j} N\omega_l, N\omega_k )}{\gamma^{\circ}_q - \gamma^{\circ}_{j}}  \eqqcolon \mathfrak{n}_0^{(q',q)}[\omega_l,\omega_k],\\
l,k = i', \ldots, i'+k_{q',q}-1.
\end{multline*}
Relations~(\ref{abstr_nu_rel_2}) can be treated as the eigenvalue problem for the operator $\mathcal{N}^{(q',q)}$: 
\begin{equation}
\label{abstr_srcN^q_eigenvalues}
\mathcal{N}^{(q',q)} \omega_l = \nu_l \omega_l, \qquad l = i', \ldots, i'+k_{q',q}-1,
\end{equation}
where
\begin{equation*}
\mathcal{N}^{(q',q)} \coloneqq P_{q',q} \left. \left( N_1^0 - \frac{1}{2} Z^*Z S P - \frac{1}{2} S P Z^*Z \right)\right|_{\mathfrak{N}_{q',q}} + \mathcal{N}^{(q',q)}_0
\end{equation*}
and $\mathcal{N}^{(q',q)}_0$ is the operator acting in $\mathfrak{N}_{q',q}$ and generated by the form $\mathfrak{n}_0^{(q',q)}[\cdot,\cdot]$.

\begin{remark}
	\label{rem1.7}
	Let $N_0 = 0$. By \eqref{abstr_N0_N*}, this condition is equivalent to the relations $\mu_l = 0$ for all $l=1, \ldots,n$. In this case, we have $\mathfrak{N}_{1,q} = \mathfrak{N}_{q}$, $q=1,\ldots,p$. Then we shall write $\mathcal{N}^{(q)}$ instead of $\mathcal{N}^{(1,q)}$. Suppose, in addition, that $\mathcal{N}^{(q)}\ne 0$
	for some $q \in \{1,\dots,p\}$. By \eqref{abstr_srcN^q_eigenvalues}, this assumption means that $\nu_j \ne 0$ for some $j \in \{1,\dots,n\}$. 
\end{remark}

\section{Approximation of the operator $e^{-i \tau \varepsilon^{-2} A(t)}P$}
\label{abstr_exp_section}

\subsection{Approximation of the operator $e^{-i \tau \varepsilon^{-2} A(t)}P$}
Let $\varepsilon > 0$. We study the behavior of the operator $e^{-i \tau \varepsilon^{-2} A ( t)}$ for small $\varepsilon$. We shall multiply this operator by the \textquotedblleft smoothing factor\textquotedblright \ $\varepsilon^s (t^2 + \varepsilon^2)^{-s/2}P$, where $s > 0$. (The term is explained by the fact that in applications to DOs this factor turns into the smoothing operator.) Our goal is to find an approximation of the smoothed operator exponential with an error of order $O (\varepsilon)$ for minimal possible $s$.

We rely on the following statements proved in~\cite[Theorem~2.1]{BSu2008} and~\cite[Corollaries~3.3, 3.5]{Su2017}.

\begin{thrm}[\cite{BSu2008}]
	\label{abstr_exp_general_thrm_wo_eps}
	For $\tau \in \mathbb{R}$ and $|t| \le t^0$ we have
	\begin{equation}
	\label{abstr_exp_general_est_wo_eps}
	\| e^{-i \tau A(t)} P - e^{-i \tau t^2 SP} P \| \le  2C_1 |t| + C_2 |\tau| |t|^3.
	\end{equation}
\end{thrm}
\begin{thrm}[\cite{Su2017}]
	\label{abstr_exp_enchcd_thrm_wo_eps_1}
	Suppose that $N=0$. Then for $\tau \in \mathbb{R}$ and $|t| \le t^0$ we have
	\begin{equation} 
	\label{abstr_exp_enchcd_est_wo_eps_1}     
	\| e^{-i \tau A(t)} P - e^{-i \tau t^2 SP} P \| \le 2C_1 |t| + C_{4} | \tau | t^4, 
	\end{equation}
	where $C_4 = \beta_4 \delta^{-1} \| X_1 \|^4$.
\end{thrm}
\begin{thrm}[\cite{Su2017}]
	\label{abstr_exp_enchcd_thrm_wo_eps_2}
	Suppose that $N_0 = 0$. Then for $\tau \in \mathbb{R}$ and $|t| \le t^{00}$ we have
	\begin{equation*}     
	\|  e^{-i \tau A(t)} P - e^{-i \tau t^2 SP} P \| \le C_{5} |t| + C_{6} | \tau | t^4. 
	\end{equation*}
	Here $t^{00}$ is subject to the restriction 
	\begin{equation}
	\label{abstr_t00}
	t^{00} \le  (4 \beta_2)^{-1} \delta^{1/2} \|X_1\|^{-3} c^{\circ},
	\end{equation}
	where
	\begin{equation}
	\label{abstr_c^circ}
	c^{\circ} \coloneqq \min_{(j,l) \in \mathcal{Z}} c^{\circ}_{jl}, \qquad \mathcal{Z} \coloneqq \{(j,l) \colon 1 \le j,l \le p, j \ne l, P_jNP_l \ne 0\}.
	\end{equation}
	The constants $C_5$, $C_6$ are given by 
	$$C_5 = \beta_5 \delta^{-1/2} (\|X_1\| + n^2 \| X_1\|^3 (c^{\circ})^{-1}),
	\quad  C_6 = \beta_6 \delta^{-1} (\|X_1\|^4 + n^2 \| X_1\|^8 (c^{\circ})^{-2}).
	$$ 
\end{thrm}

Now, we apply the formulated results and we start with Theorem~\ref{abstr_exp_general_thrm_wo_eps}. Let $|t| \le t^0$. By~(\ref{abstr_exp_general_est_wo_eps}) (with $\tau$ replaced by $\varepsilon^{-2} \tau$), 
\begin{multline*}
\bigl\| e^{-i \tau \varepsilon^{-2} A (t)} P - e^{-i \tau \varepsilon^{-2} t^2 SP} P \bigr\| \varepsilon^3 (t^2 + \varepsilon^2)^{-3/2} \\ \le (2C_{1}|t| + C_{2} \varepsilon^{-2}  |\tau| |t|^3) \varepsilon^3 (t^2 + \varepsilon^2)^{-3/2} \le (C_{1} + C_{2}|\tau|) \varepsilon.
\end{multline*}
We arrive at the following result which has been proved before in~\cite[Theorem~2.6]{BSu2008}.
\begin{thrm}[\cite{BSu2008}]
	\label{abstr_exp_general_thrm}
	For $\tau \in \mathbb{R}$ and $|t| \le t^0$ we have
	\begin{equation*}
	\bigl\| e^{-i \tau \varepsilon^{-2} A (t)} P - e^{-i \tau \varepsilon^{-2} t^2 SP} P \bigr\| \varepsilon^3 (t^2 + \varepsilon^2)^{-3/2} \le  (C_1 + C_2 |\tau|) \varepsilon.
	\end{equation*}
	The constants $C_1$, $C_2$ are majorated by polynomials of the variables $\delta^{-1/2}$, $\|X_1\|$.
\end{thrm}

Theorem~\ref{abstr_exp_enchcd_thrm_wo_eps_1} allows us to improve the result of Theorem~\ref{abstr_exp_general_thrm} in the case where $N = 0$.
\begin{thrm}
	\label{abstr_exp_enchcd_thrm_1}
	Suppose that $N = 0$. Then for $\tau \in \mathbb{R}$ and $|t| \le t^0$ we have
	\begin{equation} 
	\label{abstr_exp_enchcd_est_1}   
	\bigl\| e^{-i \tau \varepsilon^{-2} A (t)} P - e^{-i \tau \varepsilon^{-2} t^2 S P}P \bigr\| \varepsilon^{2} (t^2 + \varepsilon^2)^{-1} \le  (C_1 + C'_4 |\tau|^{1/2}) \varepsilon. 
	\end{equation}
	The constants $C_1$, $C'_4$ are majorated by polynomials of the variables $\delta^{-1/2}$, $\|X_1\|$.
\end{thrm}

\begin{proof}
	Note that for $|t| \ge \varepsilon^{1/2}/|\tau|^{1/4}$ we have
	\begin{equation*}
	\frac{\varepsilon^{2}}{t^2 + \varepsilon^2} \le
	\frac{\varepsilon^{2}}{\tfrac{\varepsilon}{|\tau|^{1/2}} + \varepsilon^2} = \frac{\varepsilon |\tau|^{1/2}}{1 + \varepsilon |\tau|^{1/2} } \le \varepsilon |\tau|^{1/2},
	\end{equation*}
	whence the left-hand side of~(\ref{abstr_exp_enchcd_est_1}) does not exceed $2 |\tau|^{1/2} \varepsilon$.
	
	Using~(\ref{abstr_exp_enchcd_est_wo_eps_1}) with $\tau$ replaced by $\varepsilon^{-2} \tau$, for $|t| < \varepsilon^{1/2}/|\tau|^{1/4}$ we obtain
	\begin{multline*}
	\bigl\| e^{-i \tau \varepsilon^{-2} A(t)} P - e^{-i \tau \varepsilon^{-2} t^2 S P} P\bigr\| \varepsilon^2 (t^2 + \varepsilon^2)^{-1} \le  \left( 2 C_1 |t| + C_4 \varepsilon^{-2} |\tau| t^4 \right) \varepsilon^{2} (t^2 + \varepsilon^2)^{-1} \\ \le C_1 \varepsilon + C_4 |\tau| t^2 \le C_1 \varepsilon + C_4 |\tau|^{1/2} \varepsilon.
	\end{multline*}
	The required estimate~(\ref{abstr_exp_enchcd_est_1}) follows with the constant $C'_4= \max\{2,C_4\}$.
\end{proof}

Similarly, using Theorem~\ref{abstr_exp_enchcd_thrm_wo_eps_2}, one can deduce the following result.
\begin{thrm}
	\label{abstr_exp_enchcd_thrm_2}
	Suppose that $N_0 = 0$. Then for $\tau \in \mathbb{R}$ and $|t| \le t^{00}$ we have
	\begin{equation*} 
	\bigl\| e^{-i \tau \varepsilon^{-2} A (t)} P - e^{-i \tau \varepsilon^{-2} t^2 S P}P \bigr\| \varepsilon^{2} (t^2 + \varepsilon^2)^{-1} \le  (C_5 + C'_6 |\tau|^{1/2}) \varepsilon. 
	\end{equation*}    
	Here $t^{00}$ is subject to~\emph{(\ref{abstr_t00})}, the constants $C_5$, $C'_6$ are majorated by polynomials of the variables  $\delta^{-1/2}$, $\|X_1\|$, $n$, $(c^{\circ})^{-1}$.
\end{thrm}

\begin{remark}
	Theorems~\emph{\ref{abstr_exp_enchcd_thrm_1}} and~\emph{\ref{abstr_exp_enchcd_thrm_2}} improve the results of Theorems~\emph{4.2} and~\emph{4.3} from~\emph{\cite{Su2017}} with respect to dependence of the estimates on $\tau$. 
\end{remark}

\subsection{Sharpness of the results with respect to the smoothing factor}
Now, we show that the obtained results are sharp with respect to the smoothing factor. The following theorem proved in~\cite[Theorem~4.4]{Su2017} confirms the sharpness of Theorem~\ref{abstr_exp_general_thrm}.

\begin{thrm}[\cite{Su2017}]
	\label{abstr_exp_smooth_shrp_thrm_1}
	Suppose that $N_0 \ne 0$. Let $\tau \ne 0$ and $0 \le s < 3$. Then there does not exist a constant $C(\tau) >0$ such that the estimate
	\begin{equation}
	\label{abstr_exp_shrp_est}
	\bigl\| e^{-i \tau \varepsilon^{-2} A (t)} P - e^{-i \tau \varepsilon^{-2} t^2 S P} P \bigr\| \varepsilon^s (t^2 + \varepsilon^2)^{-s/2} \le  C(\tau) \varepsilon
	\end{equation}
	holds for all sufficiently small $|t|$ and $\varepsilon$.
\end{thrm}

Next, we confirm the sharpness of Theorems~\ref{abstr_exp_enchcd_thrm_1}, \ref{abstr_exp_enchcd_thrm_2}.

\begin{thrm}
	\label{abstr_exp_smooth_shrp_thrm_2}
	Suppose that $N_0 = 0$ and $\mathcal{N}^{(q)} \ne 0$ for some $q \in \{1, \ldots, p\}$. Let $\tau \ne 0$ and $0 \le s < 2$. Then there does not exist a constant $C(\tau) >0$ such that estimate~\eqref{abstr_exp_shrp_est} holds for all sufficiently small $|t|$ and $\varepsilon$.
\end{thrm}

\begin{proof}
	We start with preliminary remarks. Since $F(t)^{\perp}P = (P - F(t))P$, from~(\ref{abstr_F(t)_threshold}) it follows that
	\begin{equation}
	\label{abstr_smooth_shrp_Fperp_est}
	\| e^{-i \tau \varepsilon^{-2} A (t)} F(t)^{\perp} P \| \varepsilon (t^2 + \varepsilon^2)^{-1/2} \le C_1 |t| \varepsilon (t^2 + \varepsilon^2)^{-1/2} \le C_1 \varepsilon, \qquad |t| \le t^0.
	\end{equation}
	Next, for $ |t| \le t^0$ we have
	\begin{equation}
	\label{abstr_smooth_shrp_f1}
	e^{-i \tau \varepsilon^{-2} A (t)} F(t) = \sum_{l=1}^{n} e^{-i \tau \varepsilon^{-2} \lambda_l(t)} (\cdot,\varphi_l(t)) \varphi_l(t).
	\end{equation}
	From the convergence of the power series expansions~(\ref{abstr_A(t)_eigenvectors_series}) it follows that
	\begin{equation}
	\label{abstr_smooth_shrp_f2}
	\| \varphi_l(t) - \omega_l \| \le c_1 |t|, \qquad |t| \le t_*, \quad l = 1,\ldots,n.
	\end{equation}
	
	It suffices to assume that $1 \le s <2$.	Let us fix $0 \ne \tau \in \mathbb{R}$. We prove by contradiction. Suppose that for some $1 \le s < 2$ there exists a constant $C(\tau) > 0$ such that~(\ref{abstr_exp_shrp_est}) is valid for all sufficiently small $|t|$ and $\varepsilon$. By~(\ref{abstr_smooth_shrp_Fperp_est})--(\ref{abstr_smooth_shrp_f2}), this assumption is equivalent to the existence of a positive constant $\widetilde{C} (\tau)$ such that the estimate
	\begin{equation}
	\label{abstr_smooth_shrp_f3}
	\left\| \sum_{l=1}^{n} \left( e^{-i \tau \varepsilon^{-2} \lambda_l(t)} - e^{-i \tau \varepsilon^{-2} t^2 \gamma_l} \right) (\cdot,\omega_l) \omega_l  \right\| \varepsilon^{s} (t^2 + \varepsilon^2)^{-s/2} \le \widetilde{C} (\tau) \varepsilon
	\end{equation}
	is valid for all sufficiently small $|t|$ and $\varepsilon$.
	
	By Remark~\ref{rem1.7}, the conditions $N_0 = 0$ and $\mathcal{N}^{(q)} \ne 0$ for some $q \in \{1, \ldots, p\}$ mean that in the expansions~(\ref{abstr_A(t)_eigenvalues_series}) $\mu_l = 0$ for all $l=1, \ldots,n$ and $\nu_j \ne 0$ at least for one $j$. Then, 
	by \eqref{abstr_A(t)_eigenvalues_series}, $\lambda_j(t) = \gamma_j t^2 + \nu_j t^4 + O(|t|^5)$.
	Assume that $t_*$ is sufficiently small so that 
	\begin{equation}
	\label{2.11}
	\frac{1}{2} |\nu_j| t^4 \le | \lambda_j(t) - \gamma_j t^2| \le \frac{3}{2} |\nu_j| t^4, \quad  |t| \le t_*.
	\end{equation}

	Apply the operator under the norm sign in~(\ref{abstr_smooth_shrp_f3}) to $\omega_j$. Then 
	\begin{equation}
	\label{abstr_smooth_shrp_f4}
	\left| e^{-i \tau \varepsilon^{-2} \lambda_j(t)} - e^{-i \tau \varepsilon^{-2} t^2 \gamma_j} \right| \varepsilon^{s} (t^2 + \varepsilon^2)^{-s/2} \le \widetilde{C} (\tau) \varepsilon
	\end{equation}
	for all sufficiently small $|t|$ and $\varepsilon$. The left-hand side of~(\ref{abstr_smooth_shrp_f4}) can be written as 
	$$2 \left| \sin \left( \frac{1}{2} \tau \varepsilon^{-2} (\lambda_j(t) - \gamma_j t^2) \right) \right| \varepsilon^{s} (t^2 + \varepsilon^2)^{-s/2}.$$

	Now, assuming that $\varepsilon$ is sufficiently small so that $\varepsilon \le \pi^{-1/2} |\nu_j\tau|^{1/2} t_*^2$, we put $t = t(\varepsilon) = \pi^{1/4} |\nu_j \tau|^{-1/4} \varepsilon^{1/2} = c \varepsilon^{1/2}$. Then $t(\varepsilon) \le t_*$ and, by \eqref{2.11},
	$$
	2 \left| \sin \left( \frac{1}{2} \tau \varepsilon^{-2} (\lambda_j(t(\varepsilon)) - \gamma_j  t(\varepsilon)^2) \right) \right|  \ge \sqrt{2},
	$$
	whence~(\ref{abstr_smooth_shrp_f4}) implies $\sqrt{2} \varepsilon^{s} (c^2 \varepsilon + \varepsilon^2)^{-s/2} \le \widetilde{C} (\tau) \varepsilon$.
	It follows that the function $\varepsilon^{s/2-1} (c^2 + \varepsilon)^{-s/2}$ is uniformly bounded for small $\varepsilon$. But this is not true if $s < 2$. This contradiction completes the proof.
\end{proof}

\subsection{Sharpness of the results with respect to time}
Now, we prove the following statement confirming the sharpness of Theorem~\ref{abstr_exp_general_thrm} with respect to dependence of the estimate on time.

\begin{thrm}
	\label{abstr_exp_time_shrp_thrm_1}
	Suppose that $N_0 \ne 0$. Let $s\ge 3$. Then there does not exist a positive function $C(\tau)$ such that $\lim_{\tau \to \infty} C(\tau)/ |\tau| = 0$ and estimate \eqref{abstr_exp_shrp_est} holds for all $\tau \in \mathbb{R}$ and all sufficiently small $|t|$ and $\varepsilon > 0$.
\end{thrm}

\begin{proof}
	We prove by contradiction. Suppose that there exists a positive function $C(\tau)$ such that $\lim_{\tau \to \infty} C(\tau)/ |\tau| = 0$ and~\eqref{abstr_exp_shrp_est} is valid for all sufficiently small $|t|$ and $\varepsilon$. By~(\ref{abstr_smooth_shrp_Fperp_est})--(\ref{abstr_smooth_shrp_f2}), this assumption is equivalent to the existence of a positive function $\widetilde{C} (\tau)$ such that $\lim_{\tau \to \infty} \widetilde{C}(\tau)/ |\tau| = 0$ and estimate \eqref{abstr_smooth_shrp_f3}
	holds for all sufficiently small $|t|$ and $\varepsilon$.
	
	The condition $N_0 \ne 0$ means that $\mu_j \ne 0$ at least for one $j$. Then $\lambda_j(t) = \gamma_j t^2 + \mu_j t^3 + O(t^4)$.
	Assume that $t_*$ is sufficiently small so that
	\begin{equation}
	\label{abstr_time_shrp_f4}
	\frac{1}{2} |\mu_j| |t|^3 \le | \lambda_j(t) - \gamma_j t^2 | \le \frac{3}{2} |\mu_j| |t|^3, \qquad |t| \le t_*.
	\end{equation}

	Applying the operator under the norm sign in~\eqref{abstr_smooth_shrp_f3} to $\omega_j$, we obtain 
	\begin{equation}
	\label{abstr_time_shrp_f2}
	2 \left| \sin \left( \frac{1}{2} \tau \varepsilon^{-2} (\lambda_j(t) - \gamma_j t^2) \right) \right| \varepsilon^{s} (t^2 + \varepsilon^2)^{-s/2} \le \widetilde{C} (\tau) \varepsilon
	\end{equation}
	for all sufficiently small $|t|$ and $\varepsilon$.

	Let $\tau \ne 0$, and let $\varepsilon \le \varepsilon_\flat |\tau|^{1/2}$, where $\varepsilon_\flat = (2\pi)^{-1/2} |\mu_j|^{1/2} t_*^{3/2}$. We put
	\begin{equation}
	\label{t_flat}
	t_\flat = t_\flat(\varepsilon,\tau) = c_\flat |\tau|^{-1/3} \varepsilon^{2/3}, \quad c_\flat = \left(\frac{\pi}{4}\right)^{1/3} |\mu_j|^{-1/3}.
	\end{equation}
	Then $t_\flat \le t_*/2$ and, by \eqref{abstr_time_shrp_f4}, $\left| \frac{\tau}{2 \varepsilon^2}  (\lambda_j(t_\flat) - \gamma_j t_\flat^2)  \right| \le \frac{3 \pi}{16} < \frac{\pi}{4}$.	Applying the estimate $|\sin y| \ge \frac{2}{\pi}|y|$ for $|y| \le \pi/2$ and using the lower estimate \eqref{abstr_time_shrp_f4},
	we obtain 
	$$
	\left| \sin \left( \frac{1}{2} \tau \varepsilon^{-2} (\lambda_j(t_\flat) - \gamma_j t_\flat^2) \right) \right|
	\ge \frac{|\tau|}{\pi \varepsilon^2} \left| \lambda_j(t_\flat) - \gamma_j t_\flat^2 \right|
	\ge \frac{|\tau| |\mu_j|}{2 \pi \varepsilon^2} t_\flat^3 = \frac{1}{8}.
	$$
	Together with \eqref{abstr_time_shrp_f2}, this yields $	\frac{1}{4} {\varepsilon^s}{(t_\flat^2 + \varepsilon^2)^{-s/2}} \le  \widetilde{C} (\tau) \varepsilon$ 	for all sufficiently small $\varepsilon$. By \eqref{t_flat}, this implies
	\begin{equation}
	\label{abstr_time_shrp_f5}
	\frac{1}{4}  \frac{( \varepsilon |\tau|)^{s/3 -1}}{( c_\flat^2 + (\varepsilon |\tau|)^{2/3})^{s/2} } \le \frac{\widetilde{C}(\tau)}{|\tau|}
	\end{equation}
	for all sufficiently small $\varepsilon > 0$. But estimate~(\ref{abstr_time_shrp_f5}) is not true for large $|\tau|$ and $\varepsilon = |\tau|^{-1}$ 
	since $\lim_{\tau \to \infty} \widetilde{C}(\tau)/ |\tau| = 0$. This contradiction completes the proof.
\end{proof}

The following statement confirms the sharpness of Theorems~\ref{abstr_exp_enchcd_thrm_1}, \ref{abstr_exp_enchcd_thrm_2}.

\begin{thrm}
	\label{abstr_exp_time_shrp_thrm_2}
	Suppose that $N_0 = 0$ and $\mathcal{N}^{(q)} \ne 0$ for some $q \in \{1, \ldots, p\}$. Let $s \ge 2$. Then there does not exist a positive function $C(\tau)$ such that $\lim_{\tau \to \infty} C(\tau)/ |\tau|^{1/2} = 0$ and estimate~\eqref{abstr_exp_shrp_est}
	holds for all $\tau \in \mathbb{R}$ and all sufficiently small $|t|$ and $\varepsilon > 0$.
\end{thrm}

\begin{proof}
	By Remark~\ref{rem1.7}, the conditions $N_0 = 0$ and $\mathcal{N}^{(q)} \ne 0$ for some $q \in \{1, \ldots, p\}$ mean that  $\mu_l = 0$ for all $l=1, \ldots,n$ and $\nu_j \ne 0$ at least for one $j$. Then 
	for sufficiently small $t_*$ relations \eqref{2.11} hold. 
	
	We prove by contradiction. Suppose the opposite. Then, similarly to the proof of Theorem \ref{abstr_exp_time_shrp_thrm_1}, we conclude that 
	there exists  a positive function $\widetilde{C} (\tau)$ such that $\lim_{\tau \to \infty} \widetilde{C}(\tau)/ |\tau|^{1/2} = 0$ and 
	\eqref{abstr_time_shrp_f2} holds for all sufficiently small $|t|$ and $\varepsilon$.

	Let $\tau \ne 0$, and let $\varepsilon \le \varepsilon_\dag |\tau|^{1/2}$, where $\varepsilon_\dag = \frac{1}{2} \pi^{-1/2} |\nu_j|^{1/2} t_*^{2}$. We put
	\begin{equation}
	\label{t_dag}
	t_\dag = t_\dag(\varepsilon,\tau) = c_\dag |\tau|^{-1/4} \varepsilon^{1/2}, \quad c_\dag = \left(\frac{\pi}{4}\right)^{1/4} |\nu_j|^{-1/4}.
	\end{equation}
	Then $t_\dag \le t_*/2$ and, by \eqref{2.11},
	$\left| \frac{\tau}{2 \varepsilon^2}  (\lambda_j(t_\dag) - \gamma_j t_\dag^2)  \right| \le \frac{3 \pi}{16} < \frac{\pi}{4}$.
	Applying the estimate $|\sin y| \ge \frac{2}{\pi}|y|$ for $|y| \le \pi/2$ and using the lower estimate \eqref{2.11},
	we obtain 
	$$
	\left| \sin \left( \frac{1}{2} \tau \varepsilon^{-2} (\lambda_j(t_\dag) - \gamma_j t_\dag^2) \right) \right|
	\ge \frac{|\tau|}{\pi \varepsilon^2} \left| \lambda_j(t_\dag) - \gamma_j t_\dag^2 \right|
	\ge \frac{|\tau| |\nu_j|}{2 \pi \varepsilon^2} t_\dag^4 = \frac{1}{8}.
	$$
	Combining this with  \eqref{abstr_time_shrp_f2}, we have $\frac{1}{4} {\varepsilon^s}{(t_\dag^2 + \varepsilon^2)^{-s/2}} \le  \widetilde{C} (\tau) \varepsilon$ for all sufficiently small $\varepsilon$. By \eqref{t_dag}, this is equivalent to 
	\begin{equation}
	\label{2.18}
	\frac{1}{4}  \frac{( \varepsilon |\tau|^{1/2})^{s/2 -1}}{( c_\dag^2 + \varepsilon |\tau|^{1/2})^{s/2} } \le \frac{\widetilde{C}(\tau)}{|\tau|^{1/2}}
	\end{equation}
	for all sufficiently small $\varepsilon > 0$. But estimate~(\ref{2.18}) is not true for large $|\tau|$ and $\varepsilon = |\tau|^{-1/2}$ 
	since $\lim_{\tau \to \infty} \widetilde{C}(\tau)/ |\tau|^{1/2} = 0$. This contradiction completes the proof. 
\end{proof}

\section{Approximation of the sandwiched operator exponential}
\label{abstr_sndw_section}
\subsection{The operator family $A(t) = M^* \widehat{A} (t) M$}
\label{abstr_A_and_Ahat_section}
Let $\widehat{\mathfrak{H}}$ be yet another separable Hilbert space. Let $\widehat{X} (t) = \widehat{X}_0 + t \widehat{X}_1 \colon \widehat{\mathfrak{H}} \to \mathfrak{H}_* $ be a family of operators of the same form as $X(t)$, and suppose that $\widehat{X} (t)$ satisfies the assumptions of Subsection~\ref{abstr_X_A_section}. Let $M \colon \mathfrak{H} \to \widehat{\mathfrak{H}}$ be an isomorphism. Suppose that $M \Dom X_0 = \Dom \widehat{X}_0$, $X(t) = \widehat{X} (t) M$, and then also $X_0 = \widehat{X}_0 M$, $X_1 = \widehat{X}_1 M$. In $\widehat{\mathfrak{H}}$, we consider the family of selfadjoint operators $\widehat{A} (t) = \widehat{X} (t)^* \widehat{X} (t)$. Then, obviously, 
\begin{equation}
\label{abstr_A_Ahat}
A(t) = M^* \widehat{A} (t) M.
\end{equation} 
In what follows, all the objects corresponding to the family $\widehat{A}(t)$ are marked by the sign \textquotedblleft$\widehat{\phantom{m}} $\textquotedblright. Note that $\widehat{\mathfrak{N}} = M \mathfrak{N}$ and $\widehat{\mathfrak{N}}_* =  \mathfrak{N}_*$. 

In $\widehat{\mathfrak{H}}$ we consider the positive definite operator $Q \coloneqq (M M^*)^{-1}$. Let $Q_{\widehat{\mathfrak{N}}}$ be the block of $Q$ in the subspace $\widehat{\mathfrak{N}}$, i.e.
$Q_{\widehat{\mathfrak{N}}} = \widehat{P} Q|_{\widehat{\mathfrak{N}}}$. Obviously, $Q_{\widehat{\mathfrak{N}}}$ is an isomorphism in $\widehat{\mathfrak{N}}$.

Condition~\ref{abstr_nondegeneracy_cond} implies that for $\widehat{A} (t)$ we have
\begin{equation*}
\widehat{A} (t) \ge \widehat{c}_* t^2 I, \quad \widehat{c}_* = c_* \|M\|^{-2}, \qquad |t| \le t^0.
\end{equation*}

According to~\cite[Proposition~1.2]{Su2007}, the orthogonal projection $P$ of $\mathfrak{H}$ onto $\mathfrak{N}$ and  the orthogonal projection $\widehat{P}$ of $\widehat{\mathfrak{H}}$ onto $\widehat{\mathfrak{N}}$ satisfy the following relation
\begin{equation}
\label{abstr_P_Phat}
P = M^{-1} (Q_{\widehat{\mathfrak{N}}})^{-1} \widehat{P} (M^*)^{-1}.
\end{equation}
Let $\widehat{S} \colon \widehat{\mathfrak{N}} \to \widehat{\mathfrak{N}}$ be the spectral germ of $\widehat{A} (t)$ at $t = 0$, and let $S$ be the germ of $A (t)$. The following identity was obtained in~\cite[Chapter~1, Subsection~1.5]{BSu2003}:
\begin{equation}
\label{abstr_S_Shat}
S = P M^* \widehat{S} M |_\mathfrak{N}.
\end{equation}

\subsection{The operators $\widehat{Z}_Q$ and $\widehat{N}_Q$}
\label{abstr_hatZ_Q_and_hatN_Q_section}
For the operator family $\widehat{A} (t)$ we introduce the operator $\widehat{Z}_Q$ acting in $\widehat{\mathfrak{H}}$ and taking an element $\widehat{u} \in \widehat{\mathfrak{H}}$ to the solution $\widehat{\psi}_Q$ of the problem $\widehat{X}^*_0 (\widehat{X}_0 \widehat{\psi}_Q + \widehat{X}_1 \widehat{\omega}) = 0$, $Q \widehat{\psi}_Q \perp \widehat{\mathfrak{N}}$, where $\widehat{\omega} = \widehat{P} \widehat{u}$. According to~\cite[\S6]{BSu2005}, the operator $Z$ for $A(t)$ and the operator $\widehat{Z}_Q$ introduced above satisfy
\begin{equation}
\label{abstr_Z_hatZ_Q}
\widehat{Z}_Q =M Z M^{-1} \widehat{P}.
\end{equation}
Next, we put $\widehat{N}_Q \coloneqq \widehat{Z}_Q^* \widehat{X}_1^* \widehat{R} \widehat{P} + (\widehat{R} \widehat{P})^* \widehat{X}_1  \widehat{Z}_Q$. According to~\cite[\S6]{BSu2005}, the operator $N$ for $A(t)$ and the operator $\widehat{N}_Q$ satisfy
\begin{equation}
\label{abstr_N_hatN_Q}
\widehat{N}_Q = \widehat{P} (M^*)^{-1} N M^{-1} \widehat{P}.
\end{equation}
Since $N = N_0 + N_*$, we have $\widehat{N}_Q = \widehat{N}_{0,Q} + \widehat{N}_{*,Q}$, where
\begin{equation}
\label{abstr_N0*_hatN0*_Q}
\widehat{N}_{0,Q} = \widehat{P} (M^*)^{-1} N_0 M^{-1} \widehat{P}, \qquad \widehat{N}_{*,Q} = \widehat{P} (M^*)^{-1} N_* M^{-1} \widehat{P}.
\end{equation}

The following lemma was proved in~\cite[Lemma~5.1]{Su2017}.
\begin{lemma}[\cite{Su2017}]
	\label{abstr_N_hatNQ_lemma}
	The relation $N = 0$ is equivalent to the relation $\widehat{N}_Q = 0$. The relation $N_0 = 0$ is equivalent to the relation $\widehat{N}_{0,Q} = 0$.
\end{lemma}

\subsection{The operators $\widehat{Z}_{2,Q}$, $\widehat{R}_{2,Q}$, and $\widehat{N}_{1,Q}^0$}
\label{abstr_hatZ2_Q_hatR2_Q_N1^0_Q_section}
Let $\widehat{u} \in \widehat{\mathfrak{H}}$ and let $\widehat{\phi}_Q = \widehat{\phi}_Q(\widehat{u}) \in \Dom \widehat{X}_0$ be a (weak) solution of the equation
\begin{equation*}
\widehat{X}^*_0 (\widehat{X}_0 \widehat{\phi}_Q + \widehat{X}_1 \widehat{Z}_Q \widehat{\omega}) = -\widehat{X}_1^* \widehat{R} \widehat{\omega} + Q (Q_{\widehat{\mathfrak{N}}})^{-1} \widehat{P} \widehat{X}_1^* \widehat{R} \widehat{\omega}, \qquad Q \widehat{\phi}_Q \perp \widehat{\mathfrak{N}},
\end{equation*}
where $\widehat{\omega} = \widehat{P} \widehat{u}$. Clearly, the right-hand side of this equation belongs to $\widehat{\mathfrak{N}}^{\perp} = \Ran \widehat{X}_0^*$, thereby the solvability condition is satisfied. We define an operator 
$\widehat{Z}_{2,Q} \colon \widehat{\mathfrak{H}} \to \widehat{\mathfrak{H}}$ by the formula $\widehat{Z}_{2,Q} \widehat{u} = \widehat{\phi}_Q(\widehat{u})$. 

Now, we introduce an operator $\widehat{R}_{2,Q} \colon \widehat{\mathfrak{N}} \to \mathfrak{H}_*$ by the formula
$\widehat{R}_{2,Q} = \widehat{X}_0 \widehat{Z}_{2,Q}  + \widehat{X}_1 \widehat{Z}_Q $.
Finally, we define the operator $\widehat{N}_{1,Q}^0$:
\begin{equation*}
\widehat{N}_{1,Q}^0 = \widehat{Z}_{2,Q}^* \widehat{X}_1^* \widehat{R} \widehat{P} + (\widehat{R} \widehat{P})^* \widehat{X}_1 \widehat{Z}_{2,Q} + \widehat{R}_{2,Q}^* \widehat{R}_{2,Q} \widehat{P}.
\end{equation*}
According to~\cite[Section~6.3]{VSu2011}, we have
\begin{gather}
\notag
\widehat{Z}_{2,Q} =M Z_2 M^{-1} \widehat{P}, \\
\notag
R_2 = \widehat{R}_{2,Q} M |_{\mathfrak{N}}, \quad \widehat{R}_{2,Q} = R_2 M^{-1} |_{\widehat{\mathfrak{N}}},\\
\label{abstr_N1^0_N1Q^0_hat}
\widehat{N}_{1,Q}^0 = \widehat{P} (M^*)^{-1} N_1^0 M^{-1} \widehat{P}.
\end{gather}

\subsection{Relations between the operators and the coefficients of the power series expansions}
Now, we describe the relations between the coefficients of the power series expansions~(\ref{abstr_A(t)_eigenvalues_series}), (\ref{abstr_A(t)_eigenvectors_series}) and the operators $\widehat{S}$ and $Q_{\widehat{\mathfrak{N}}}$. (See~\cite[Sections~1.6,~1.7]{BSu2005}.) Denote $\zeta_l \coloneqq M \omega_l \in \widehat{\mathfrak{N}}$, $l = 1, \ldots, n$. Then relations~(\ref{abstr_S_eigenvectors}) and~(\ref{abstr_P_Phat}), (\ref{abstr_S_Shat}) show that
\begin{equation}
\label{abstr_hatS_gener_spec_problem}
\widehat{S} \zeta_l  = \gamma_l Q_{\widehat{\mathfrak{N}}} \zeta_l, \qquad l = 1, \ldots, n. 
\end{equation}
The set $\zeta_1, \ldots, \zeta_n$ forms a basis in $\widehat{\mathfrak{N}}$ orthonormal with the weight $Q_{\widehat{\mathfrak{N}}}$: $(Q_{\widehat{\mathfrak{N}}} \zeta_l, \zeta_j) = \delta_{lj}$, $l,j = 1,\ldots,n$.

The operators $\widehat{N}_{0,Q}$ and $\widehat{N}_{*,Q}$ can be described in terms of the coefficients of the expansions~(\ref{abstr_A(t)_eigenvalues_series}) and~(\ref{abstr_A(t)_eigenvectors_series}); cf.~(\ref{abstr_N0_N*}). We put $\widetilde{\zeta}_l \coloneqq M \widetilde{\omega}_l \in \widehat{\mathfrak{N}}$, $l = 1, \ldots,n$. Then
\begin{equation}
\label{abstr_hatN_0Q_N_*Q}
\widehat{N}_{0,Q} = \sum_{k=1}^{n} \mu_k (\cdot, Q_{\widehat{\mathfrak{N}}} \zeta_k) Q_{\widehat{\mathfrak{N}}} \zeta_k, \qquad \widehat{N}_{*,Q} = \sum_{k=1}^{n} \gamma_k \left( (\cdot, Q_{\widehat{\mathfrak{N}}} \widetilde{\zeta}_k) Q_{\widehat{\mathfrak{N}}} \zeta_k + (\cdot, Q_{\widehat{\mathfrak{N}}} \zeta_k) Q_{\widehat{\mathfrak{N}}} \widetilde{\zeta}_k \right). 
\end{equation}

Now, we return to the notation of Section~\ref{abstr_cluster_section}. Recall that the different eigenvalues of the germ $S$ are denoted by $\gamma^{\circ}_q$, $q = 1,\ldots,p$, and the corresponding eigenspaces by $\mathfrak{N}_q$. The set of the vectors $\omega_l$, $l=i, \ldots, i+k_q-1$, where $i = i(q) = k_1+\ldots+k_{q-1}+1$, forms an orthonormal basis in~$\mathfrak{N}_q$. Then the same numbers $\gamma^{\circ}_q$, $q = 1,\ldots,p$, are the different eigenvalues of the problem~(\ref{abstr_hatS_gener_spec_problem}) and $M \mathfrak{N}_q \eqqcolon \widehat{\mathfrak{N}}_{q,Q}$ are the corresponding eigenspaces. The vectors $\zeta_l = M \omega_l$, $l=i, \ldots, i+k_q-1$, form a basis in $\widehat{\mathfrak{N}}_{q,Q}$ (orthonormal with the weight $Q_{\widehat{\mathfrak{N}}}$). By $\mathcal{P}_q$ we denote the \textquotedblleft skew\textquotedblright \ projection of $\widehat{\mathfrak H}$ onto $\widehat{\mathfrak{N}}_{q,Q}$ that is orthogonal with respect to the inner product $(Q_{\widehat{\mathfrak{N}}} \cdot, \cdot)$, i.e. $\mathcal{P}_q = \sum_{l=i}^{i+k_q-1} (\cdot, Q_{\widehat{\mathfrak{N}}} \zeta_l) \zeta_l$. It is easily seen that $\mathcal{P}_q =M P_q M^{-1} \widehat{P}$.
Using~(\ref{abstr_N0_N*_invar_repr}), (\ref{abstr_N_hatN_Q}), and~(\ref{abstr_N0*_hatN0*_Q}), it is easy to check that
\begin{equation}
\label{abstr_hatN_0Q_N_*Q_invar_repr}
\widehat{N}_{0,Q} = \sum_{j=1}^{p} \mathcal{P}_j^* \widehat{N}_Q \mathcal{P}_j, \qquad \widehat{N}_{*,Q} = \sum_{\substack{1 \le l,j \le p \\ j \ne l}} \mathcal{P}_l^* \widehat{N}_Q \mathcal{P}_j.
\end{equation}

Next, we  find a relationship between the eigenvalues and eigenvectors of  problem~(\ref{abstr_N_eigenvalues}) and the operator $\widehat{N}_Q$. Let $\gamma^{\circ}_q$ be 
the $q$-th eigenvalue of problem~(\ref{abstr_hatS_gener_spec_problem}) of multiplicity $k_q$. Then from~\eqref{abstr_N_eigenvalues},~(\ref{abstr_N_hatN_Q}) and the obvious identity $M P_q = \widehat{P}_{q,Q} M P_q$, where $\widehat{P}_{q,Q}$ is the orthogonal projection of $\widehat{\mathfrak H}$ onto $\widehat{\mathfrak{N}}_{q,Q}$, it is seen that
\begin{equation}
\label{abstr_hatN_Q_gener_spec_problem}
\widehat{P}_{q,Q} \widehat{N}_Q \zeta_l = \mu_l Q_{\widehat{\mathfrak{N}}_{q,Q}} \zeta_l, \qquad l = i(q), \ldots, i(q)+k_q-1,
\end{equation}
where $Q_{\widehat{\mathfrak{N}}_{q,Q}} = \widehat{P}_{q,Q} Q|_{\widehat{\mathfrak{N}}_{q,Q}}$. Recall that the different eigenvalues of problem~(\ref{abstr_N_eigenvalues}) are denoted by $\mu^{\circ}_{q',q}$, $q' = 1,\ldots,p'(q)$, and the corresponding eigenspaces by $\mathfrak{N}_{q',q}$. Then the same numbers $\mu^{\circ}_{q',q}$, $q' = 1,\ldots,p'(q)$, are different eigenvalues of problem~(\ref{abstr_hatN_Q_gener_spec_problem}), and $M \mathfrak{N}_{q',q} \eqqcolon \widehat{\mathfrak{N}}_{q',q,Q}$ are the corresponding eigenspaces.

Finally, we connect the eigenvalues and eigenvectors of the problem~(\ref{abstr_srcN^q_eigenvalues}) and the operator
\begin{equation*}
\widehat{\mathcal{N}}_Q^{(q',q)} =  \widehat{P}_{q',q,Q} \left. \left( \widehat{N}_{1,Q}^0 - \frac{1}{2} \widehat{Z}_Q^* Q \widehat{Z}_Q Q^{-1} \widehat{S} \widehat{P} - \frac{1}{2} \widehat{S} \widehat{P} Q^{-1} \widehat{Z}_Q^* Q \widehat{Z}_Q \right)\right|_{\widehat{\mathfrak{N}}_{q',q,Q}} + \widehat{\mathcal{N}}^{(q',q)}_{0,Q},
\end{equation*}
where $\widehat{\mathcal{N}}^{(q',q)}_{0,Q}$ is the operator in $\widehat{\mathfrak{N}}_{q',q,Q}$ generated by the form
\begin{equation*}
\widehat{\mathfrak{n}}_{0,Q}^{(q',q)}[\cdot,\cdot] =  \sum_{\substack{j\in\{1,\ldots,p\} \\ j \ne q}} \frac{( \widehat{P}_{j,Q} (M M^*) \widehat{P}_{j,Q} \widehat{N}_Q  \cdot,\widehat{N}_Q \cdot)}{\gamma^{\circ}_q - \gamma^{\circ}_j},
\end{equation*}
and $\widehat{P}_{q',q,Q}$ is the orthogonal projection onto $\widehat{\mathfrak{N}}_{q',q,Q}$. From~\eqref{abstr_srcN^q_eigenvalues}, \eqref{abstr_S_Shat}--\eqref{abstr_N_hatN_Q}, (\ref{abstr_N1^0_N1Q^0_hat}), and the identities $M P_j = \widehat{P}_{j,Q} M P_j$, $\widehat{P}_{j,Q} M (I - P_j) = 0$, $j=1,\ldots,p$, $M P_{q',q} = \widehat{P}_{q',q,Q} M P_{q',q}$, it is seen that
\begin{equation}
\label{abstr_scrNhat_M^(q)_gener_spec_problem}
\widehat{\mathcal{N}}_Q^{(q',q)} \zeta_l = \nu_l Q_{q',q,Q} \zeta_l, \qquad l = i', \ldots, i'+k_{q',q}-1.
\end{equation}
Here $i' = i'(q',q) = i(q)+k_{1,q}+\ldots+k_{q'-1,q}$ and $Q_{q',q,Q} = \widehat{P}_{q',q,Q} Q|_{\widehat{\mathfrak{N}}_{q',q,Q}}$.

\begin{remark}
	Let $\widehat{N}_{0,Q} = 0$. By \eqref{abstr_hatN_0Q_N_*Q}, this condition is equivalent to the relations $\mu_l = 0$ for all $l = 1, \ldots, n$. In this case, we have $\mathfrak{N}_{1,q} = \mathfrak{N}_{q}$, $q=1,\ldots,p$. Then we shall write $\widehat{\mathcal{N}}^{(q)}_Q$ instead of $\widehat{\mathcal{N}}^{(1,q)}_Q$. 
\end{remark}

\begin{remark}
	\label{abstr_N^(q',q)_hatNQ^(q',q)_rem}
	Let $\widehat{\mathcal{N}}_Q^{(q',q)} \ne 0$ for some $q$ and $q'$. Then, by~\eqref{abstr_scrNhat_M^(q)_gener_spec_problem}, $\nu_l \ne0$ for some $l = i'(q',q), \ldots, i'(q',q)+k_{q',q}-1$. From~\eqref{abstr_srcN^q_eigenvalues} it directly follows that $\mathcal{N}^{(q',q)} \ne 0$.
\end{remark}

\subsection{Approximation of the sandwiched operator exponential}
\label{abstr_sndw_exp_section}
In this subsection we find approximation for the operator exponential $e^{-i \tau \varepsilon^{-2} A(t)}$ of the family~(\ref{abstr_A_Ahat}) in terms of the germ $\widehat{S}$ of $\widehat{A}(t)$ and the isomorphism $M$. It is convenient to border the exponential by appropriate factors.

Denote $M_0 \coloneqq (Q_{\widehat{\mathfrak{N}}})^{-1/2}$. The following estimates were proved in~\cite[Lemma~5.3]{Su2017}:
\begin{align}
\label{abstr_sndw_exp_est_1}
\| M e^{-i\tau A(t)} M^{-1} \widehat{P} -  M_0 e^{-i \tau t^2 M_0 \widehat{S} M_0} M_0^{-1} \widehat{P}\| \le  \| M \|^2 \| M^{-1} \|^2 \| e^{-i \tau A(t)} P -  e^{-i \tau t^2 S P} P \|,\\
\label{abstr_sndw_exp_est_2}
\| e^{-i \tau A(t)} P -  e^{-i \tau t^2 S P} P \| \le  \| M \|^2 \| M^{-1} \|^2 \| M e^{-i\tau A(t)} M^{-1} \widehat{P} -  M_0 e^{-i \tau t^2 M_0 \widehat{S} M_0} M_0^{-1} \widehat{P}\|. 
\end{align}
Theorems~\ref{abstr_exp_general_thrm}, \ref{abstr_exp_enchcd_thrm_1}, \ref{abstr_exp_enchcd_thrm_2}, Lemma~\ref{abstr_N_hatNQ_lemma}, and  inequality~(\ref{abstr_sndw_exp_est_1}) directly imply the following results.

\begin{thrm}[\cite{BSu2008}]
	\label{abstr_sndw_exp_general_thrm}
	Under the assumptions of Subsection~\emph{\ref{abstr_A_and_Ahat_section}} for $\tau \in \mathbb{R}$, $\varepsilon > 0$, and $|t| \le t^0$ we have
	\begin{equation*}
	\|M e^{-i \tau \varepsilon^{-2} A(t)} M^{-1} \widehat{P} -  M_0 e^{-i \tau \varepsilon^{-2} t^2 M_0 \widehat{S} M_0} M_0^{-1} \widehat{P}\| \varepsilon^{3} (t^2 + \varepsilon^2)^{-3/2} \le \| M \|^2 \| M^{-1} \|^2 (C_1  + C_2 |\tau| ) \varepsilon.
	\end{equation*}
\end{thrm}

\begin{thrm}
	\label{abstr_sndw_exp_enchcd_thrm_1}
	Suppose that the assumptions of Subsection~\emph{\ref{abstr_A_and_Ahat_section}} are satisfied. Suppose that $\widehat{N}_Q = 0$. Then for $\tau \in \mathbb{R}$, $\varepsilon > 0$, and $|t| \le t^0$ we have
	\begin{equation*}
	\| M e^{-i \tau \varepsilon^{-2} A(t)} M^{-1} \widehat{P} -  M_0 e^{-i \tau \varepsilon^{-2} t^2 M_0 \widehat{S} M_0} M_0^{-1} \widehat{P}\| \varepsilon^{2} (t^2 + \varepsilon^2)^{-1} \le \| M \|^2 \| M^{-1} \|^2 ( C_1  + C'_4 | \tau |^{1/2} ) \varepsilon.
	\end{equation*}
\end{thrm}

\begin{thrm}
	\label{abstr_sndw_exp_enchcd_thrm_2}
	Suppose that the assumptions of Subsection~\emph{\ref{abstr_A_and_Ahat_section}} and Condition~\emph{\ref{abstr_nondegeneracy_cond}} are satisfied. Suppose that $\widehat{N}_{0,Q} = 0$. Then for $\tau \in \mathbb{R}$, $\varepsilon > 0$, and $|t| \le t^{00}$ we have
	\begin{equation*}
	\| M e^{-i \tau \varepsilon^{-2} A(t)} M^{-1} \widehat{P} -  M_0 e^{-i \tau \varepsilon^{-2} t^2 M_0 \widehat{S} M_0} M_0^{-1} \widehat{P} \| \varepsilon^{2} (t^2 + \varepsilon^2)^{-1} \le \| M \|^2 \| M^{-1} \|^2 (C_5 + C'_6 | \tau |^{1/2} ) \varepsilon.
	\end{equation*}
\end{thrm}

Theorem~\ref{abstr_sndw_exp_general_thrm} was proved in~\cite[Theorem~3.2]{BSu2008}.
\begin{remark}
	Theorems~\emph{\ref{abstr_sndw_exp_enchcd_thrm_1}} and~\emph{\ref{abstr_sndw_exp_enchcd_thrm_2}} improve the results of Theorems~\emph{5.8}, \emph{5.9} from~\emph{\cite{Su2017}} with respect to $\tau$. 
\end{remark}

\subsection{The sharpness of the results}
Theorems~\ref{abstr_exp_smooth_shrp_thrm_1}, \ref{abstr_exp_smooth_shrp_thrm_2}, \ref{abstr_exp_time_shrp_thrm_1}, \ref{abstr_exp_time_shrp_thrm_2}, Lemma~\ref{abstr_N_hatNQ_lemma}, Remark~\ref{abstr_N^(q',q)_hatNQ^(q',q)_rem}, and inequality~(\ref{abstr_sndw_exp_est_2}) directly imply the following statements.

\begin{thrm}[\cite{Su2017}]
	\label{abstr_sndw_exp_smooth_shrp_thrm_1}
	Suppose that $\widehat{N}_{0,Q} \ne 0$. Let $\tau \ne 0$ and $0 \le s < 3$. Then there does not exist a constant $C(\tau) >0$ such the estimate
	\begin{equation}
	\label{abstr_sndw_exp_shrp_est}
	\bigl\| M e^{-i \tau \varepsilon^{-2} A(t)} M^{-1} \widehat{P} -  M_0 e^{-i \tau \varepsilon^{-2} t^2 M_0 \widehat{S} M_0} M_0^{-1} \widehat{P} \bigr\| \varepsilon^s (t^2 + \varepsilon^2)^{-s/2} \le  C(\tau) \varepsilon
	\end{equation}
	holds for all sufficiently small $|t|$ and $\varepsilon$.
\end{thrm}

\begin{thrm}
	\label{abstr_sndw_exp_smooth_shrp_thrm_2}
	Let $\widehat{N}_{0,Q} = 0$ and $\widehat{\mathcal{N}}_Q^{(q)} \ne 0$ for some $q \in \{1, \ldots, p\}$.  Let $\tau \ne 0$ and $0 \le s < 2$. Then there does not exist a constant $C(\tau) >0$ such that estimate~\eqref{abstr_sndw_exp_shrp_est} holds for all sufficiently small $|t|$ and~$\varepsilon$.
\end{thrm}

\begin{thrm}
	\label{abstr_sndw_exp_time_shrp_thrm_1}
	Suppose that $\widehat{N}_{0,Q} \ne 0$. Let $s \ge 3$. Then there does not exist a positive function $C(\tau)$ such that $\lim_{\tau \to \infty} C(\tau)/ |\tau| = 0$ and estimate~\eqref{abstr_sndw_exp_shrp_est} holds for all $\tau \in \mathbb{R}$ and for all sufficiently small $|t|$ and $\varepsilon > 0$.
\end{thrm}

\begin{thrm}
	\label{abstr_sndw_exp_time_shrp_thrm_2}
	Suppose that $\widehat{N}_{0,Q} = 0$ and $\widehat{\mathcal{N}}_Q^{(q)} \ne 0$ for some $q \in \{1, \ldots, p\}$. Let $s \ge 2$. Then there does not exist a positive function $C(\tau)$ such that $\lim_{\tau \to \infty} C(\tau)/ |\tau|^{1/2} = 0$ and estimate~\eqref{abstr_sndw_exp_shrp_est} holds for all $\tau \in \mathbb{R}$ and for all sufficiently small $|t|$ and $\varepsilon > 0$.
\end{thrm}

Theorem~\ref{abstr_sndw_exp_smooth_shrp_thrm_1} was proved in~\cite[Theorem~5.10]{Su2017}.

\part{Periodic differential operators in $L_2(\mathbb{R}^d; \mathbb{C}^n)$}
\label{L2_operators_part}

\section{The class of periodic differential operators}

\subsection{Preliminaries: lattices and the Gelfand transformation}
Let $\Gamma$ be a lattice in $\mathbb{R}^d$ generated by the basis $\mathbf{a}_1, \ldots , \mathbf{a}_d$, i.e., $\Gamma = \left\{ \mathbf{a} \in \mathbb{R}^d \colon \mathbf{a} = \sum_{j=1}^{d} n_j \mathbf{a}_j, \; n_j \in \mathbb{Z} \right\}$,
and let $\Omega$ be the elementary cell of this lattice: 
$\Omega \coloneqq \left\{ \mathbf{x} \in \mathbb{R}^d \colon \mathbf{x} = \sum_{j=1}^{d} \xi_j \mathbf{a}_j, \; 0 < \xi_j < 1 \right\}$. 
The basis $\mathbf{b}_1, \ldots, \mathbf{b}_d$ dual to $\mathbf{a}_1, \ldots , \mathbf{a}_d$ is defined by the relations $\left< \mathbf{b}_l, \mathbf{a}_j \right> = 2 \pi \delta_{lj}$. This basis generates the \emph{lattice} $\widetilde \Gamma$ \emph{dual to} $\Gamma$. Denote by $\widetilde \Omega$ the central \emph{Brillouin zone} of $\widetilde \Gamma$:
\begin{equation}
\label{Brillouin_zone}
\widetilde \Omega = \left\{ \mathbf{k} \in \mathbb{R}^d \colon | \mathbf{k} | < | \mathbf{k} - \mathbf{b} |, \; 0 \ne \mathbf{b} \in \widetilde \Gamma \right\}.
\end{equation}
Denote $| \Omega | = \meas \Omega$, $| \widetilde \Omega | = \meas \widetilde \Omega$. Note that $| \Omega |  | \widetilde \Omega | = (2 \pi)^d$. Let $r_0$ be the radius of the ball \emph{inscribed} in $\clos \widetilde \Omega$. We have
$2 r_0 = \min|\mathbf{b}|$, $0 \ne \mathbf{b} \in \widetilde \Gamma$. 

With the lattice $ \Gamma$, we associate the discrete Fourier transformation
$\{ \hat{\mathbf{u}}_{\mathbf{b}} \} \mapsto \mathbf{u}$:
\begin{equation*}
\mathbf{u}(\mathbf{x}) = | \Omega |^{-1/2} \sum_{\mathbf{b} \in \widetilde \Gamma} \hat{\mathbf{u}}_{\mathbf{b}} e^{i \left<\mathbf{b}, \mathbf{x} \right>},
\end{equation*}
which is a unitary mapping of $l_2 (\widetilde \Gamma; \mathbb{C}^n) $ onto $L_2 (\Omega; \mathbb{C}^n)$. \textit{By $\widetilde H^\sigma(\Omega; \mathbb{C}^n)$ we denote the subspace of functions from $ H^\sigma(\Omega; \mathbb{C}^n)$ whose $\Gamma$-periodic extension to $\mathbb{R}^d$ belongs to $H^\sigma_{\mathrm{loc}}(\mathbb{R}^d; \mathbb{C}^n)$}. We have
\begin{equation}
\label{D_and_fourier}
\int_{\Omega} |(\mathbf{D} + \mathbf{k}) \mathbf{u}|^2\, d\mathbf{x} = \sum_{\mathbf{b} \in \widetilde{\Gamma}} |\mathbf{b} + \mathbf{k} |^2 |\hat{\mathbf{u}}_{\mathbf{b}}|^2, \qquad \mathbf{u} \in \widetilde{H}^1(\Omega; \mathbb{C}^n), \; \mathbf{k} \in \mathbb{R}^d,
\end{equation} 
and convergence of the series in the right-hand side of~(\ref{D_and_fourier}) is equivalent to the inclusion $\mathbf{u} \in \widetilde{H}^1(\Omega; \mathbb{C}^n)$. From~(\ref{Brillouin_zone})~and~(\ref{D_and_fourier}) it follows that
\begin{equation}
\label{(D+k)u_est}
\int_{\Omega} |(\mathbf{D} + \mathbf{k}) \mathbf{u}|^2\, d\mathbf{x} \ge \sum_{\mathbf{b} \in \widetilde{\Gamma}} | \mathbf{k} |^2 |\hat{\mathbf{u}}_{\mathbf{b}}|^2 = | \mathbf{k} |^2 \int_{\Omega} |\mathbf{u}|^2\, d\mathbf{x}, \qquad  \mathbf{u} \in \widetilde{H}^1(\Omega; \mathbb{C}^n), \; \mathbf{k} \in \widetilde{\Omega}.
\end{equation}  

Initially, the Gelfand transformation $\mathscr{U}$ is defined on the functions from the Schwartz class $\mathbf{v} \in \mathcal{S}(\mathbb{R}^d; \mathbb{C}^n)$ by the formula
\begin{equation*}
\tilde{\mathbf{v}} ( \mathbf{k}, \mathbf{x}) = (\mathscr{U} \- \mathbf{v}) (\mathbf{k}, \mathbf{x}) = | \widetilde \Omega |^{-1/2} \sum_{\mathbf{a} \in \Gamma} e^{- i \left< \mathbf{k}, \mathbf{x} + \mathbf{a} \right>} \mathbf{v} ( \mathbf{x} + \mathbf{a}), \qquad
\mathbf{x} \in \Omega, \; \mathbf{k} \in \widetilde \Omega,
\end{equation*}
and extends by continuity up to a unitary mapping:
\begin{equation*}
\mathscr{U} \colon L_2 (\mathbb{R}^d; \mathbb{C}^n) \to \int_{\widetilde \Omega} \oplus  L_2 (\Omega; \mathbb{C}^n) \, d \mathbf{k} \eqqcolon  \mathcal{K}.
\end{equation*}

\subsection{Factorized second order operators $\mathcal{A}$}
\label{A_section}
Let $b (\mathbf{D})= \sum_{l=1}^d b_l D_l$, where $b_l$ are constant ($ m \times n $)-matrices (in general, with complex entries).
\emph{It is assumed that $m \ge n$}. Consider the symbol $b(\boldsymbol{\xi}) = \sum_{l=1}^d b_l \xi_l$, $\boldsymbol{\xi} \in \mathbb{R}^d$.
\emph{Suppose that} $\rank b( \boldsymbol{\xi} ) = n$, $0 \ne  \boldsymbol{\xi} \in \mathbb{R}^d$. This condition is equivalent to the inequalities 
\begin{equation}
\label{alpha0_alpha1}
\alpha_0 \mathbf{1}_n \le b( \boldsymbol{\theta} )^* b( \boldsymbol{\theta} ) \le \alpha_1 \mathbf{1}_n, \quad  \boldsymbol{\theta} \in \mathbb{S}^{d-1}, \quad 0 < \alpha_0 \le \alpha_1 < \infty,
\end{equation}
with some positive constants $\alpha_0$, $\alpha_1$.
Let $f(\mathbf{x})$ be a $\Gamma$-periodic  ($n \times n$)-matrix-valued function and $h(\mathbf{x})$ be a $\Gamma$-periodic  ($m \times m$)-matrix-valued function such that
\begin{equation}
\label{h_f_Linfty}
f, f^{-1} \in L_{\infty} (\mathbb{R}^d); \quad h, h^{-1} \in L_{\infty} (\mathbb{R}^d).
\end{equation}
Consider the closed operator $\mathcal{X}  \colon  L_2 (\mathbb{R}^d ; \mathbb{C}^n) \to  L_2 (\mathbb{R}^d ; \mathbb{C}^m)$ given by 
$\mathcal{X} = h b( \mathbf{D} ) f$ on the domain
$\Dom \mathcal{X} = \left\lbrace \mathbf{u} \in L_2 (\mathbb{R}^d ; \mathbb{C}^n) \colon f \mathbf{u} \in H^1  (\mathbb{R}^d ; \mathbb{C}^n) \right\rbrace$. The selfadjoint operator $\mathcal{A} = \mathcal{X}^* \mathcal{X}$ in $L_2 (\mathbb{R}^d ; \mathbb{C}^n)$ is generated by the closed quadratic form $\mathfrak{a}[\mathbf{u}, \mathbf{u}] = \| \mathcal{X} \mathbf{u} \|^2_{L_2(\mathbb{R}^d)}$, $ \mathbf{u} \in 
\Dom \mathcal{X}$. Formally, we have
\begin{equation}
\label{A}
\mathcal{A} = f (\mathbf{x})^* b( \mathbf{D} )^* g( \mathbf{x} )  b( \mathbf{D} ) f(\mathbf{x}),
\end{equation}
where $g( \mathbf{x} ) \coloneqq h( \mathbf{x} )^*  h( \mathbf{x} )$. Using the Fourier transformation, and~(\ref{alpha0_alpha1}),~(\ref{h_f_Linfty}), it is easily seen that
\begin{equation*}
\alpha_0 \| g^{-1} \|_{L_{\infty}}^{-1} \| \mathbf{D} (f \mathbf{u}) \|_{L_2(\mathbb{R}^d)}^2 \le \mathfrak{a}[\mathbf{u}, \mathbf{u}] \le \alpha_1 \| g \|_{L_{\infty}} \| \mathbf{D} (f \mathbf{u}) \|_{L_2(\mathbb{R}^d)}^2, \qquad \mathbf{u} \in \Dom \mathcal{X}.
\end{equation*}

\subsection{The operators $\mathcal{A}(\mathbf{k})$}
We put
\begin{equation}
\label{Spaces_H}
\mathfrak{H} = L_2 (\Omega; \mathbb{C}^n), \qquad \mathfrak{H}_* = L_2 (\Omega; \mathbb{C}^m)
\end{equation}
and consider the closed operator $\mathcal{X} (\mathbf{k}) \colon \mathfrak{H} \to \mathfrak{H}_*$ depending on the parameter $\mathbf{k} \in \mathbb{R}^d$ and given by $\mathcal{X} (\mathbf{k}) = hb(\mathbf{D} + \mathbf{k})f$ on the domain $\Dom \mathcal{X} (\mathbf{k}) = \bigl\lbrace \mathbf{u} \in \mathfrak{H} \colon   f \mathbf{u} \in \widetilde{H}^1 (\Omega; \mathbb{C}^n)\bigr\rbrace \eqqcolon  \mathfrak{d}$.
The selfadjoint operator $\mathcal{A} (\mathbf{k}) =\mathcal{X} (\mathbf{k})^* \mathcal{X} (\mathbf{k}) \colon \mathfrak{H} \to \mathfrak{H}$ is generated by the quadratic form $\mathfrak{a}(\mathbf{k})[\mathbf{u}, \mathbf{u}] = \| \mathcal{X}(\mathbf{k}) \mathbf{u} \|_{\mathfrak{H}_*}^2$, $\mathbf{u} \in \mathfrak{d}$.
Using the Fourier series expansion for $\mathbf{v} = f\mathbf{u}$ and conditions~(\ref{alpha0_alpha1}), (\ref{h_f_Linfty}), it is easy to check that
\begin{equation}
\label{a(k)_est}
\alpha_0 \|g^{-1} \|_{L_\infty}^{-1} \|(\mathbf{D} + \mathbf{k}) f \mathbf{u} \|_{L_2(\Omega)}^2 \le \mathfrak{a}(\mathbf{k})[\mathbf{u}, \mathbf{u}] \le \alpha_1 \|g \|_{L_\infty} \|(\mathbf{D} + \mathbf{k}) f \mathbf{u} \|_{L_2(\Omega)}^2, \quad \mathbf{u} \in \mathfrak{d}.
\end{equation}

From~(\ref{(D+k)u_est}) and the lower estimate~(\ref{a(k)_est}) it follows that
\begin{equation}
\label{c_*}	
\mathcal{A} (\mathbf{k}) \ge c_* |\mathbf{k}|^2 I, \qquad \mathbf{k} \in \widetilde{\Omega}, \; c_* = \alpha_0\|f^{-1} \|_{L_\infty}^{-2} \|g^{-1} \|_{L_\infty}^{-1} .
\end{equation}

We put $\mathfrak{N} \coloneqq  \Ker \mathcal{A} (0) = \Ker \mathcal{X} (0)$. Relations~(\ref{a(k)_est}) with $\mathbf{k} = 0$ show that
\begin{equation}
\label{frakN}
\mathfrak{N} = \left\lbrace \mathbf{u} \in L_2 (\Omega; \mathbb{C}^n) \colon f \mathbf{u} = \mathbf{c} \in \mathbb{C}^n \right\rbrace, \qquad \dim \mathfrak{N} = n. 
\end{equation}

\subsection{The band functions}
Let $E_j(\mathbf{k})$, $j \in \mathbb{N}$, be the consecutive eigenvalues of the operator $\mathcal{A}(\mathbf{k})$ (the band functions):
\begin{equation*}
E_1(\mathbf{k}) \le E_2(\mathbf{k}) \le \ldots \le E_j(\mathbf{k}) \le \ldots, \qquad \mathbf{k} \in \mathbb{R}^d.
\end{equation*}
The band functions $E_j(\mathbf{k})$ are continuous and $\widetilde{\Gamma}$-periodic.
In~\cite[Chapter~2, Subsection~2.2]{BSu2003}, by simple variational arguments it was shown that 
\begin{alignat*}{2}
E_j(\mathbf{k})  &\ge c_* | \mathbf{k} |^2, \qquad  &&\mathbf{k} \in \clos \widetilde{\Omega}, \qquad j = 1, \ldots, n, \\
E_{n+1}(\mathbf{k}) &\ge c_* r_0^2, \qquad &&\mathbf{k} \in \clos \widetilde{\Omega}, \\
E_{n+1}(0)  &\ge 4 c_* r_0^2.
\end{alignat*}

\subsection{The direct integral for the operator $\mathcal{A}$}
Under the Gelfand transformation $\mathscr{U}$ the operator $\mathcal{A}$ expands in the direct integral of the operators $\mathcal{A} (\mathbf{k})$:
\begin{equation}
\label{Gelfand_A_decompose}
\mathscr{U} \mathcal{A}  \mathscr{U}^{-1} = \int_{\widetilde \Omega} \oplus \mathcal{A} (\mathbf{k}) \, d \mathbf{k}.
\end{equation}
This means the following. If $\mathbf{v} \in \Dom \mathcal{X}$, then $\tilde{\mathbf{v}}(\mathbf{k}, \cdot) \in \mathfrak{d}$ for a.e. $\mathbf{k} \in \widetilde \Omega$ and
\begin{equation}
\label{Gelfand_a_form}
\mathfrak{a}[\mathbf{v}, \mathbf{v}] = \int_{\widetilde{\Omega}} \mathfrak{a}(\mathbf{k}) [\tilde{\mathbf{v}}(\mathbf{k}, \cdot), \tilde{\mathbf{v}}(\mathbf{k}, \cdot)] \, d \mathbf{k} .
\end{equation}
Conversely, if $\tilde{\mathbf{v}} \in \mathcal{K}$ satisfies $\tilde{\mathbf{v}}(\mathbf{k}, \cdot) \in \mathfrak{d}$ for a.e. $\mathbf{k} \in \widetilde \Omega$ and the integral in~(\ref{Gelfand_a_form}) is finite, then $\mathbf{v} \in \Dom \mathcal{X}$ and~(\ref{Gelfand_a_form}) is valid.

\subsection{Incorporation of the operators $\mathcal{A} (\mathbf{k})$ in the abstract scheme}
For $d > 1$, the operators  $\mathcal{A} (\mathbf{k})$ depend on the multidimensional parameter $\mathbf{k}$. According to \cite[Chapter~2]{BSu2003}, we introduce the one-dimensional parameter $t = | \mathbf{k}|$. We shall apply the method described in Chapter~\ref{abstr_part}. Now, all constructions will depend on the additional parameter $\boldsymbol{\theta} = \mathbf{k} / | \mathbf{k}| \in \mathbb{S}^{d-1}$ and we need to make our estimates uniform with respect to $\boldsymbol{\theta}$. The spaces $\mathfrak{H}$ and $\mathfrak{H}_*$ are defined by~(\ref{Spaces_H}). We put $X(t) = X(t; \boldsymbol{\theta}) \coloneqq \mathcal{X}(t \boldsymbol{\theta})$. Then $X(t; \boldsymbol{\theta}) = X_0 + t  X_1 (\boldsymbol{\theta})$, where $X_0 = h(\mathbf{x}) b (\mathbf{D}) f(\mathbf{x})$, $\Dom X_0 = \mathfrak{d}$, and $X_1 (\boldsymbol{\theta})$ is the bounded operator of multiplication by the matrix $h(\mathbf{x}) b(\boldsymbol{\theta}) f(\mathbf{x})$. Next, we put $A(t) = A(t; \boldsymbol{\theta}) \coloneqq \mathcal{A}(t \boldsymbol{\theta})$. The kernel $\mathfrak{N} = \Ker X_0$ is described by~(\ref{frakN}). As was shown in~\cite[Chapter~2, \S3]{BSu2003}, the distance $d^0$ from the point $\lambda_0 = 0$ to the rest of the spectrum of the operator $\mathcal{A}(0)$ is subject to the estimate $d^0 \ge 4 c_* r_0^2$. The condition $n \le n_* = \operatorname{dim} \Ker X^*_0$ is also fulfilled. Moreover, either $n_* = n$ (if $m = n$), or $n_* = \infty$ (if $m > n$).

In Subsection~\ref{abstr_X_A_section}, it was required to fix a number $\delta \in (0, d^0/8)$. Since $d^0 \ge 4 c_* r_0^2$, we choose
\begin{equation}
\label{delta_fixation}
\delta = \frac{1}{4} c_* r^2_0 = \frac{1}{4} \alpha_0\|f^{-1} \|_{L_\infty}^{-2} \|g^{-1} \|_{L_\infty}^{-1} r^2_0.
\end{equation}  
Note that by~(\ref{alpha0_alpha1}) and~(\ref{h_f_Linfty}) we have
\begin{equation}
\label{X_1_estimate}
\| X_1 (\boldsymbol{\theta}) \| \le  \alpha^{1/2}_1 \| h \|_{L_{\infty}} \| f \|_{L_{\infty}}, \qquad \boldsymbol{\theta} \in \mathbb{S}^{d-1}.
\end{equation}
We put (see~(\ref{abstr_t0}))
\begin{equation}
\label{t0_fixation}
t^0 = \delta^{1/2} \alpha_1^{-1/2} \| h \|_{L_{\infty}}^{-1} \|f\|_{L_{\infty}}^{-1} = \frac{r_0}{2} \alpha_0^{1/2} \alpha_1^{-1/2} \left( \| h \|_{L_{\infty}} \| h^{-1} \|_{L_{\infty}} \|f \|_{L_\infty} \|f^{-1} \|_{L_\infty} \right)^{-1}.
\end{equation}
Note that $t^0 \le r_0/2$. Thus, the ball $|\mathbf{k}| \le t^0$ lies inside $\widetilde{\Omega}$. It is important that $c_*$, $\delta$, $t^0$ (see~(\ref{c_*}), (\ref{delta_fixation}), (\ref{t0_fixation})) do not depend on $\boldsymbol{\theta}$. By~(\ref{c_*}),  Condition~\ref{abstr_nondegeneracy_cond} is fulfilled. The germ $S(\boldsymbol{\theta})$ of the operator $A(t; \boldsymbol{\theta})$ is nondegenerate  uniformly in $\boldsymbol{\theta}$: $S(\boldsymbol{\theta}) \ge c_* I_{\mathfrak{N}}$ (cf.~(\ref{abstr_S_nondegeneracy})).

\section{The effective characteristics of the operator  $\widehat{\mathcal{A}} = b(\mathbf{D})^* g(\mathbf{x}) b(\mathbf{D})$}
\subsection{The operator $A(t; \boldsymbol{\theta})$ in the case where $f = \mathbf{1}_n$}
The operator $A(t; \boldsymbol{\theta})$ in the case where $f = \mathbf{1}_n$ plays a special role. In this case we agree to mark all the associated objects by hat \textquotedblleft$\, \widehat{\phantom{\_}} \,$\textquotedblright. Then for the operator
\begin{equation}
\label{hatA}
\widehat{\mathcal{A}} = b(\mathbf{D})^* g(\mathbf{x}) b(\mathbf{D})
\end{equation}
the family $\widehat{\mathcal{A}} (\mathbf{k})$ is denoted by $\widehat{A} (t; \boldsymbol{\theta})$. The kernel~(\ref{frakN}) takes the form
\begin{equation}
\label{Ker3}
\widehat{\mathfrak{N}} = \left\lbrace \mathbf{u} \in L_2 (\Omega; \mathbb{C}^n) \colon \mathbf{u} = \mathbf{c} \in \mathbb{C}^n \right\rbrace, 
\end{equation}
i.e., $\widehat{\mathfrak{N}}$ consists of constant vector-valued functions. The orthogonal projection $\widehat{P}$ of the space $L_2 (\Omega; \mathbb{C}^n)$ onto the subspace~(\ref{Ker3}) is the operator of averaging over the cell:
\begin{equation}
\label{Phat_projector}
\widehat{P} \mathbf{u} = |\Omega|^{-1} \int_{\Omega} \mathbf{u} (\mathbf{x}) \, d\mathbf{x}.
\end{equation}
According to~\cite[Chapter~3,~\S1]{BSu2003}, the spectral germ $\widehat{S} (\boldsymbol{\theta}) \colon \widehat{\mathfrak{N}} \to \widehat{\mathfrak{N}} $ of the family $\widehat{A}(t; \boldsymbol{\theta})$ is represented as 
$\widehat{S} (\boldsymbol{\theta}) = b(\boldsymbol{\theta})^* g^0 b(\boldsymbol{\theta})$,  $\boldsymbol{\theta} \in \mathbb{S}^{d-1}$, where $g^0$ is the so-called \emph{effective matrix}. The constant ($m \times m$)-matrix $g^0$ is defined as follows. Let $\Lambda \in \widetilde{H}^1 (\Omega)$ be a $\Gamma$-periodic ($n \times m$)-matrix-valued function satisfying the equation
\begin{equation}
\label{equation_for_Lambda}
b(\mathbf{D})^* g(\mathbf{x}) (b(\mathbf{D}) \Lambda (\mathbf{x}) + \mathbf{1}_m) = 0, \qquad \int_{\Omega} \Lambda (\mathbf{x}) \, d \mathbf{x} = 0.
\end{equation}
The effective matrix $g^0$ can be described in terms of the matrix $\Lambda (\mathbf{x})$:
\begin{align}
\label{g0}
&g^0 = | \Omega |^{-1} \int_{\Omega} \widetilde{g} (\mathbf{x}) \, d \mathbf{x},\\
\label{g_tilde}
&\widetilde{g} (\mathbf{x}) \coloneqq g(\mathbf{x})( b(\mathbf{D}) \Lambda (\mathbf{x}) + \mathbf{1}_m).
\end{align}
It turns out that the matrix $g^0$ is positive definite. Consider the symbol
\begin{equation}
\label{effective_oper_symb}
\widehat{S} (\mathbf{k}) \coloneqq t^2 \widehat{S} (\boldsymbol{\theta}) = b(\mathbf{k})^* g^0 b(\mathbf{k}), \qquad \mathbf{k} \in \mathbb{R}^{d}.
\end{equation}
Expression~(\ref{effective_oper_symb}) is the symbol of the DO
\begin{equation}
\label{hatA0}
\widehat{\mathcal{A}}^0 = b(\mathbf{D})^* g^0 b(\mathbf{D})
\end{equation}
acting in $L_2(\mathbb{R}^d; \mathbb{C}^n)$ and called the \emph{effective operator} for the operator $\widehat{\mathcal{A}}$.

Let $\widehat{\mathcal{A}}^0 (\mathbf{k})$ be the operator family in $L_2(\Omega; \mathbb{C}^n)$ corresponding to operator~(\ref{hatA0}). Then $\widehat{\mathcal{A}}^0 (\mathbf{k})$ is given by $\widehat{\mathcal{A}}^0 (\mathbf{k}) = b(\mathbf{D} + \mathbf{k})^* g^0 b(\mathbf{D} + \mathbf{k})$ with periodic boundary conditions. Taking into account~(\ref{Phat_projector}) and~(\ref{effective_oper_symb}), we have
\begin{equation}
\label{hatS_P=hatA^0_P}
\widehat{S} (\mathbf{k}) \widehat{P} = \widehat{\mathcal{A}}^0 (\mathbf{k}) \widehat{P}.
\end{equation}

\subsection{Properties of the effective matrix}

The following properties of $g^0$ were checked in~\cite[Chapter~3, Theorem~1.5]{BSu2003}.
\begin{proposition}[\cite{BSu2003}]
	The effective matrix satisfies the following estimates
	\begin{equation}
	\label{Voigt_Reuss}
	\underline{g} \le g^0 \le \overline{g},
	\end{equation}
	where $\overline{g} \coloneqq | \Omega |^{-1} \int_{\Omega} g (\mathbf{x}) \, d \mathbf{x}$ and $\underline{g} \coloneqq \left( | \Omega |^{-1} \int_{\Omega} g (\mathbf{x})^{-1} \, d \mathbf{x}\right)^{-1}$. If $m = n$, then $g^0 = \underline{g}$.
\end{proposition}
For specific DOs, estimates~(\ref{Voigt_Reuss}) are known as the Voight-Reuss bracketing. Now, we distinguish the cases where one of the inequalities in~(\ref{Voigt_Reuss}) becomes an identity. The following statements were obtained in~\cite[Chapter~3, Propositions~1.6, 1.7]{BSu2003}.
\begin{proposition}[\cite{BSu2003}]
	The identity $g^0 = \overline{g}$ is equivalent to the relations
	\begin{equation}
	\label{g0=overline_g_relat}
	b(\mathbf{D})^* \mathbf{g}_k (\mathbf{x}) = 0, \quad k = 1, \ldots, m,
	\end{equation}
	where $\mathbf{g}_k (\mathbf{x})$, $k = 1, \ldots,m$, are the columns of the matrix $g (\mathbf{x})$.
\end{proposition}
\begin{proposition}[\cite{BSu2003}]
	The identity $g^0 = \underline{g}$ is equivalent to the representations
	\begin{equation}
	\label{g0=underline_g_relat}
	\mathbf{l}_k (\mathbf{x}) = \mathbf{l}^0_k + b(\mathbf{D}) \mathbf{w}_k(\mathbf{x}), \quad \mathbf{l}^0_k \in \mathbb{C}^m, \quad \mathbf{w}_k \in \widetilde{H}^1 (\Omega; \mathbb{C}^n), \quad k = 1, \ldots,m,
	\end{equation}
	where $\mathbf{l}_k (\mathbf{x})$, $k = 1, \ldots,m$, are the columns of the matrix $g (\mathbf{x})^{-1}$.
\end{proposition}

\subsection{The analytic branches of eigenvalues and eigenvectors}

The analytic (with respect to $t$) branches of the eigenvalues $\widehat{\lambda}_l (t, \boldsymbol{\theta})$ and the analytic branches of the eigenvectors $\widehat{\varphi}_l (t, \boldsymbol{\theta})$ of the operator $\widehat{A} (t; \boldsymbol{\theta})$ admit the power series expansions of the form~(\ref{abstr_A(t)_eigenvalues_series}),~(\ref{abstr_A(t)_eigenvectors_series}) with the coefficients depending on $\boldsymbol{\theta}$ (we do not control the interval of convergence $t = |\mathbf{k}| \le t_* (\boldsymbol{\theta})$):
\begin{align}
\label{hatA_eigenvalues_series}
\widehat{\lambda}_l (t, \boldsymbol{\theta}) &= \widehat{\gamma}_l (\boldsymbol{\theta}) t^2 + \widehat{\mu}_l (\boldsymbol{\theta}) t^3 + \widehat{\nu}_l (\boldsymbol{\theta}) t^4 + \ldots, & l &= 1, \ldots, n,   \\
\label{hatA_eigenvectors_series}
\widehat{\varphi}_l (t, \boldsymbol{\theta}) &= \widehat{\omega}_l (\boldsymbol{\theta}) + t \widehat{\psi}^{(1)}_l (\boldsymbol{\theta}) + \ldots, & l &= 1, \ldots, n.
\end{align}
According to~(\ref{abstr_S_eigenvectors}), $\widehat{\gamma}_l (\boldsymbol{\theta})$ and $\widehat{\omega}_l (\boldsymbol{\theta})$ are the eigenvalues and the eigenvectors of the germ: $
b(\boldsymbol{\theta})^* g^0 b(\boldsymbol{\theta}) \widehat{\omega}_l (\boldsymbol{\theta}) = \widehat{\gamma}_l (\boldsymbol{\theta}) \widehat{\omega}_l (\boldsymbol{\theta})$, $l = 1, \ldots, n$.

\subsection{The operator $\widehat{N} (\boldsymbol{\theta})$}

We need to describe the operator $N$ (which in abstract terms is defined in Theorem~\ref{abstr_threshold_approx_thrm_2}). According to~\cite[\S4]{BSu2005-2}, for the family $\widehat{A} (t; \boldsymbol{\theta})$ this operator takes the form
\begin{align}
\label{hatN(theta)}
\widehat{N} (\boldsymbol{\theta}) &= b(\boldsymbol{\theta})^* L(\boldsymbol{\theta}) b(\boldsymbol{\theta}) \widehat{P}, \\ 
\notag
L (\boldsymbol{\theta}) &\coloneqq | \Omega |^{-1} \int_{\Omega} (\Lambda (\mathbf{x})^* b(\boldsymbol{\theta})^* \widetilde{g}(\mathbf{x}) + \widetilde{g}(\mathbf{x})^* b(\boldsymbol{\theta}) \Lambda (\mathbf{x}) ) \, d \mathbf{x}.
\end{align}
Here $\Lambda (\mathbf{x})$ is the $\Gamma$-periodic solution of problem~(\ref{equation_for_Lambda}), and $\widetilde{g}(\mathbf{x})$ is given by~(\ref{g_tilde}).

Some cases where $\widehat{N}(\boldsymbol{\theta}) = 0$ were distinguished in~\cite[\S4]{BSu2005-2}.
\begin{proposition}[\cite{BSu2005-2}]
	\label{N=0_proposit}
	Suppose that at least one of the following conditions is fulfilled\emph{:}
	\begin{enumerate}[label=\emph{\arabic*$^{\circ}.$}, ref=\arabic*$^{\circ}$,parsep=-3pt]
		\item \label{N=0_proposit_it1} $\widehat{\mathcal{A}} = \mathbf{D}^* g(\mathbf{x}) \mathbf{D}$, where $g(\mathbf{x})$ is a symmetric matrix with real entries.
		\item Relations~\emph{(\ref{g0=overline_g_relat})} are satisfied, i.e. $g^0 = \overline{g}$.
		\item Relations~\emph{(\ref{g0=underline_g_relat})} are satisfied, i.e. $g^0 = \underline{g}$. \emph{(}If $m = n$, this is the case.\emph{)} 
	\end{enumerate}
	Then $\widehat{N} (\boldsymbol{\theta}) = 0$ for any $\boldsymbol{\theta} \in \mathbb{S}^{d-1}$.
\end{proposition}

Recall (see Remark~\ref{abstr_N_remark}) that $\widehat{N} (\boldsymbol{\theta}) = \widehat{N}_0 (\boldsymbol{\theta}) + \widehat{N}_* (\boldsymbol{\theta})$, where the operator $\widehat{N}_0 (\boldsymbol{\theta})$ is diagonal in the basis $\{ \widehat{\omega}_l (\boldsymbol{\theta})\}_{l=1}^n$, while the operator $\widehat{N}_* (\boldsymbol{\theta})$ has zero diagonal elements. We have
\begin{equation*}
(\widehat{N} (\boldsymbol{\theta}) \widehat{\omega}_l (\boldsymbol{\theta}), \widehat{\omega}_l (\boldsymbol{\theta}))_{L_2 (\Omega)} = (\widehat{N}_0 (\boldsymbol{\theta}) \widehat{\omega}_l (\boldsymbol{\theta}), \widehat{\omega}_l (\boldsymbol{\theta}))_{L_2 (\Omega)} = \widehat{\mu}_l (\boldsymbol{\theta}), \qquad l=1, \ldots, n.
\end{equation*}

In~\cite[Subsection~4.3]{BSu2005-2} the following statement was proved.

\begin{proposition}[\cite{BSu2005-2}]
	Suppose that the matrices $b(\boldsymbol{\theta})$ and $g (\mathbf{x})$ have real entries. Suppose that in the expansions~\emph{(\ref{hatA_eigenvectors_series})} the  \textquotedblleft embryos\textquotedblright \ $\widehat{\omega}_l (\boldsymbol{\theta})$, $l = 1, \ldots, n$, can be chosen to be real. Then in~\emph{(\ref{hatA_eigenvalues_series})} we have $\widehat{\mu}_l (\boldsymbol{\theta}) = 0$, $l=1, \ldots, n$, i.e., $\widehat{N}_0 (\boldsymbol{\theta}) = 0$.
\end{proposition}

In the \textquotedblleft real\textquotedblright \ case under consideration, the germ $\widehat{S} (\boldsymbol{\theta})$  is a symmetric matrix with real entries. Clearly, if the eigenvalue $\widehat{\gamma}_j (\boldsymbol{\theta})$ of the germ is simple, then the embryo  $\widehat{\omega}_j (\boldsymbol{\theta})$ is defined uniquely up to a phase factor, so we can always choose it to be real. We arrive at the following corollary.

\begin{corollary}
	\label{real_S_spec_simple_coroll}
	Suppose that $b(\boldsymbol{\theta})$ and $g (\mathbf{x})$ have real entries and the spectrum of the germ $\widehat{S} (\boldsymbol{\theta})$ is simple. Then $\widehat{N}_0 (\boldsymbol{\theta}) = 0$.
\end{corollary}

However, as is seen from~\cite[Example~8.7]{Su2017}, \cite[Subsection~14.3]{DSu2018}, in the \textquotedblleft real\textquotedblright \ case it is not always possible to choose the vectors $\widehat{\omega}_l (\boldsymbol{\theta})$ to be real. It may happen that $\widehat{N}_0 (\boldsymbol{\theta}) \ne 0$ at some isolated points $\boldsymbol{\theta}$.

\subsection{The operators $\widehat{Z}_2(\boldsymbol{\theta})$, $\widehat{R}_2(\boldsymbol{\theta})$, $\widehat{N}_1^0(\boldsymbol{\theta})$}
We need to describe the operators $Z_2$, $R_2$, $N_1^0$ (which in abstract terms are defined in Subsections~\ref{abstr_Z2_R2_section} and \ref{abstr_nu_section}) for the family $\widehat{A} (t; \boldsymbol{\theta})$. Let $\Lambda^{(2)}_l (\mathbf{x})$ be a  $\Gamma$-periodic solution of the problem
\begin{equation*}
b(\mathbf{D})^* g(\mathbf{x}) (b(\mathbf{D}) \Lambda^{(2)}_l (\mathbf{x}) + b_l \Lambda(\mathbf{x})) = b_l^* (g^0 - \widetilde{g} (\mathbf{x})), \qquad \int_{\Omega} \Lambda^{(2)}_l (\mathbf{x}) \, d \mathbf{x} = 0.
\end{equation*}
We put $
\Lambda^{(2)} (\mathbf{x}; \boldsymbol{\theta}) \coloneqq \sum_{l=1}^{d} \Lambda^{(2)}_l (\mathbf{x}) \theta_l$.
In~\cite[Subsection~6.3]{VSu2012} it was checked that
\begin{gather*}
\widehat{Z}_2(\boldsymbol{\theta}) = \Lambda^{(2)} (\mathbf{x}; \boldsymbol{\theta}) b(\boldsymbol{\theta}) \widehat{P}, \qquad
\widehat{R}_2(\boldsymbol{\theta}) = h(\mathbf{x}) (b(\mathbf{D}) \Lambda^{(2)} (\mathbf{x}; \boldsymbol{\theta}) + b(\boldsymbol{\theta}) \Lambda(\mathbf{x})) b(\boldsymbol{\theta}).
\end{gather*}
Finally, in~\cite[Subsection~6.4]{VSu2012} the following representation was obtained:
\begin{align*}
&\widehat{N}_1^0(\boldsymbol{\theta}) = b(\boldsymbol{\theta})^* L_2(\boldsymbol{\theta}) b(\boldsymbol{\theta}) \widehat{P}, 
\\
&\begin{multlined}[c]
L_2(\boldsymbol{\theta}) \coloneqq |\Omega|^{-1} \int_{\Omega} \left(\Lambda^{(2)} (\mathbf{x}; \boldsymbol{\theta})^* b(\boldsymbol{\theta})^* \widetilde{g}(\mathbf{x}) + \widetilde{g}(\mathbf{x})^* b(\boldsymbol{\theta}) \Lambda^{(2)} (\mathbf{x}; \boldsymbol{\theta}) \right) \, d \mathbf{x}  \\ 
+
|\Omega|^{-1} \int_{\Omega} \left(b(\mathbf{D}) \Lambda^{(2)} (\mathbf{x}; \boldsymbol{\theta}) + b(\boldsymbol{\theta}) \Lambda(\mathbf{x}) \right)^* g(\mathbf{x}) \left(b(\mathbf{D}) \Lambda^{(2)} (\mathbf{x}; \boldsymbol{\theta}) + b(\boldsymbol{\theta}) \Lambda(\mathbf{x})\right) \, d \mathbf{x}.
\end{multlined}
\end{align*}

\subsection{Multiplicities of the eigenvalues of the germ}
\label{eigenval_multipl_section}
Considerations of this subsection concern the case where $n \ge 2$. Now, we return to the notation of Subsection~\ref{abstr_cluster_section}, tracing the multiplicities of the eigenvalues of the germ $\widehat{S} (\boldsymbol{\theta})$. In general, the number $p(\boldsymbol{\theta})$ of different eigenvalues $\widehat{\gamma}^{\circ}_1 (\boldsymbol{\theta}), \ldots, \widehat{\gamma}^{\circ}_{p(\boldsymbol{\theta})} (\boldsymbol{\theta})$ of $\widehat{S}(\boldsymbol{\theta})$ and their multiplicities $k_1 (\boldsymbol{\theta}), \ldots, k_{p(\boldsymbol{\theta})} (\boldsymbol{\theta})$ depend on the parameter $\boldsymbol{\theta} \in \mathbb{S}^{d-1}$. For a fixed $\boldsymbol{\theta}$ denote by $\widehat{P}_j (\boldsymbol{\theta})$ the orthogonal projection of $L_2 (\Omega; \mathbb{C}^n)$ onto the eigenspace of  $\widehat{S}(\boldsymbol{\theta})$ corresponding to $\widehat{\gamma}_j^{\circ} (\boldsymbol{\theta})$. According to~(\ref{abstr_N0_N*_invar_repr}), the operators $\widehat{N}_0 (\boldsymbol{\theta})$ and $\widehat{N}_* (\boldsymbol{\theta})$ admit the following invariant representations:
\begin{equation}
\label{N0_N*_invar_repr}
\widehat{N}_0 (\boldsymbol{\theta}) = \sum_{j=1}^{p(\boldsymbol{\theta})} \widehat{P}_j (\boldsymbol{\theta}) \widehat{N} (\boldsymbol{\theta}) \widehat{P}_j (\boldsymbol{\theta}), \qquad \widehat{N}_* (\boldsymbol{\theta}) = \sum_{\substack{1 \le l,j \le p(\boldsymbol{\theta})\\ j \ne l}} \widehat{P}_j (\boldsymbol{\theta}) \widehat{N} (\boldsymbol{\theta}) \widehat{P}_l (\boldsymbol{\theta}).
\end{equation}

\subsection{The coefficients $\widehat{\nu}_l (\boldsymbol{\theta})$, $l=1, \ldots,n$}
The number $p'(q, \boldsymbol{\theta})$ of different eigenvalues $\widehat{\mu}^{\circ}_{1,q} (\boldsymbol{\theta}), \hm{\ldots}, \widehat{\mu}^{\circ}_{p'(q, \boldsymbol{\theta}),q} (\boldsymbol{\theta})$ of the operator $\widehat{P}_q(\boldsymbol{\theta}) \widehat{N}(\boldsymbol{\theta}) |_{\widehat{\mathfrak{N}}_q(\boldsymbol{\theta})}$ and their multiplicities $k_{1,q} (\boldsymbol{\theta}), \ldots, k_{p'(\boldsymbol{\theta}),q} (\boldsymbol{\theta})$ also depend on the parameter $\boldsymbol{\theta} \in \mathbb{S}^{d-1}$. For a fixed $\boldsymbol{\theta}$ denote by $\widehat{P}_{q',q} (\boldsymbol{\theta})$ the orthogonal projection of $L_2 (\Omega; \mathbb{C}^n)$ onto the eigenspace $\widehat{\mathfrak{N}}_{q',q}(\boldsymbol{\theta})$ of 
the operator $\widehat{P}_q(\boldsymbol{\theta}) \widehat{N}(\boldsymbol{\theta}) |_{\widehat{\mathfrak{N}}_q(\boldsymbol{\theta})}$
corresponding to $\widehat{\mu}^{\circ}_{q',q} (\boldsymbol{\theta})$. 

The coefficients $\widehat{\nu}_l (\boldsymbol{\theta})$, $l = i'(q',q,\boldsymbol{\theta}), \ldots, i'(q',q,\boldsymbol{\theta})+k_{q',q}(\boldsymbol{\theta})-1$, where $i'(q',q,\boldsymbol{\theta}) = i(q,\boldsymbol{\theta})+k_{1,q}(\boldsymbol{\theta})+\ldots+k_{q'-1,q}(\boldsymbol{\theta})$, $i(q,\boldsymbol{\theta}) = k_1(\boldsymbol{\theta})+\ldots+k_{q-1}(\boldsymbol{\theta})+1$, are the eigenvalues of the following problem
\begin{equation*}
\widehat{\mathcal{N}}^{(q',q)}(\boldsymbol{\theta}) \widehat{\omega}_l(\boldsymbol{\theta}) = \widehat{\nu}_l(\boldsymbol{\theta}) \widehat{\omega}_l(\boldsymbol{\theta}), \qquad l = i'(q',q,\boldsymbol{\theta}), \ldots, i'(q',q,\boldsymbol{\theta})+k_{q',q}(\boldsymbol{\theta})-1,
\end{equation*}
where
\begin{multline*}
\widehat{\mathcal{N}}^{(q',q)}(\boldsymbol{\theta}) \coloneqq \widehat{P}_{q',q}(\boldsymbol{\theta}) \left. \left( \widehat{N}_1^0(\boldsymbol{\theta}) - \frac{1}{2} \widehat{Z}(\boldsymbol{\theta})^* \widehat{Z}(\boldsymbol{\theta}) \widehat{S}(\boldsymbol{\theta}) \widehat{P} - \frac{1}{2} \widehat{S}(\boldsymbol{\theta}) \widehat{P} \widehat{Z}(\boldsymbol{\theta})^* \widehat{Z}(\boldsymbol{\theta}) \right)\right|_{\widehat{\mathfrak{N}}_{q',q}} \\ +
\sum_{\substack{j\in\{1,\ldots,p(\boldsymbol{\theta})\} \\ j \ne q}} \bigl(\gamma^{\circ}_q(\boldsymbol{\theta}) - \gamma^{\circ}_{j}(\boldsymbol{\theta})\bigr)^{-1} \widehat{P}_{q',q}(\boldsymbol{\theta}) \widehat{N}(\boldsymbol{\theta}) \widehat{P}_{j}(\boldsymbol{\theta}) \widehat{N}(\boldsymbol{\theta})|_{\widehat{\mathfrak{N}}_{q',q}(\boldsymbol{\theta})}.
\end{multline*}

Note that in the case where $\widehat{N}_0(\boldsymbol{\theta}) = 0$, we have $\widehat{\mathfrak{N}}_{1,q}(\boldsymbol{\theta}) = \widehat{\mathfrak{N}}_{q}(\boldsymbol{\theta})$, $q=1, \ldots, p(\boldsymbol{\theta})$. Then we shall write $\widehat{\mathcal{N}}^{(q)}(\boldsymbol{\theta})$ instead of $\widehat{\mathcal{N}}^{(1,q)}(\boldsymbol{\theta})$.

\section{Approximations for the operator $e^{-i \tau \varepsilon^{-2} \widehat{\mathcal{A}}(\mathbf{k})}$}
\subsection{The general case}

Consider the operator $\mathcal{H}_0 = -\Delta$ in $L_2 (\mathbb{R}^d; \mathbb{C}^n)$. Under the Gelfand transformation, the operator $\mathcal{H}_0$ expands in the direct integral of the operators $\mathcal{H}_0 (\mathbf{k})$ acting in $L_2 (\Omega; \mathbb{C}^n)$ and 
given by $| \mathbf{D} + \mathbf{k} |^2$ with periodic boundary conditions. Denote
\begin{equation}
\label{R(k,eps)}
\mathcal{R}(\mathbf{k}, \varepsilon) \coloneqq \varepsilon^2 (\mathcal{H}_0 (\mathbf{k}) + \varepsilon^2 I)^{-1}.
\end{equation}
Obviously, 
\begin{equation}
\label{R_P}
\mathcal{R}(\mathbf{k}, \varepsilon)^{s/2}\widehat{P} = \varepsilon^s (t^2 + \varepsilon^2)^{-s/2} \widehat{P}, \qquad s > 0.
\end{equation}
Note that for $ |\mathbf{k}| > \widehat{t}^{\,0}$ we have
\begin{equation}
\label{R_hatP_est}
\| \mathcal{R}(\mathbf{k}, \varepsilon)^{s/2}\widehat{P} \|_{L_2 (\Omega) \to L_2 (\Omega)} \le (\widehat{t}^{\,0})^{-s} \varepsilon^s, \qquad \varepsilon > 0, \; \mathbf{k} \in \widetilde{\Omega}, \; |\mathbf{k}| > \widehat{t}^{\,0}.
\end{equation}
Next, using the Fourier series decomposition, we see that
\begin{equation}
\label{R(k,eps)(I-P)_est}
\| \mathcal{R}(\mathbf{k}, \varepsilon)^{s/2} (I - \widehat{P}) \|_{L_2(\Omega) \to L_2 (\Omega) }  = \sup_{0 \ne \mathbf{b} \in \widetilde{\Gamma}} \varepsilon^s (|\mathbf{b} + \mathbf{k}|^2 + \varepsilon^2)^{-s/2} \le r_0^{-s} \varepsilon^s, \qquad
\varepsilon > 0, \; \mathbf{k} \in \widetilde{\Omega}.
\end{equation}

Denote
\begin{equation}
\label{Jhat(k,eps)}
\widehat{J} (\mathbf{k}, \varepsilon; \tau) \coloneqq e^{-i \tau \varepsilon^{-2} \widehat{\mathcal{A}} (\mathbf{k})} - e^{-i \tau \varepsilon^{-2} \widehat{\mathcal{A}}^0 (\mathbf{k})}.
\end{equation}
We shall apply theorems of~\S\ref{abstr_exp_section} to the operator $\widehat{A}(t; \boldsymbol{\theta}) = \widehat{\mathcal{A}}(\mathbf{k})$. In doing so, we may trace the dependence of the constants in estimates on the problem data. Note that $\widehat{c}_*$, $\widehat{\delta}$, and $\widehat{t}^ 0$ do not depend on $\boldsymbol{\theta}$ (see~(\ref{c_*}), (\ref{delta_fixation}), (\ref{t0_fixation}) with $f = \mathbf{1}_n$). According to~(\ref{X_1_estimate}) (with $f = \mathbf{1}_n$), the norm $\| \widehat{X}_1 (\boldsymbol{\theta}) \|$ can be replaced by $\alpha_1^{1/2} \| g \|_{L_{\infty}}^{1/2}$. Hence, the constants in Theorems~\ref{abstr_exp_general_thrm} and~\ref{abstr_exp_enchcd_thrm_1} (applied to the operator $\widehat{\mathcal{A}}(\mathbf{k})$) will be independent of $\boldsymbol{\theta}$. They depend only on $\alpha_0$, $\alpha_1$, $\|g\|_{L_\infty}$, $\|g^{-1}\|_{L_\infty}$, and $r_0$.

Applying Theorem~\ref{abstr_exp_general_thrm} and taking into account~(\ref{hatS_P=hatA^0_P}), (\ref{R_P})--(\ref{R(k,eps)(I-P)_est}), we arrive at the following statement proved before in~\cite[Theorem~7.1]{BSu2008}.

\begin{thrm}[\cite{BSu2008}]
	\label{hatA(k)_exp_general_thrm}
	For $\tau \in \mathbb{R}$, $\varepsilon > 0$, and $\mathbf{k} \in \widetilde{\Omega}$ we have
	\begin{equation*}
	\| \widehat{J} (\mathbf{k}, \varepsilon; \tau) \mathcal{R}(\mathbf{k}, \varepsilon)^{3/2}\|_{L_2(\Omega) \to L_2 (\Omega) }  \le \widehat{\mathcal{C}}_1 (1 + |\tau|) \varepsilon,
	\end{equation*}
	where the constant $\widehat{\mathcal{C}}_1$ depends only on $\alpha_0$, $\alpha_1$, $\|g\|_{L_\infty}$, $\|g^{-1}\|_{L_\infty}$, and $r_0$.
\end{thrm}

\subsection{The case where $\widehat{N}(\boldsymbol{\theta}) = 0$}
Now, we apply Theorem~\ref{abstr_exp_enchcd_thrm_1}, assuming that $\widehat{N}(\boldsymbol{\theta}) = 0$ for any $\boldsymbol{\theta} \in \mathbb{S}^{d-1}$. Taking~(\ref{hatS_P=hatA^0_P}), (\ref{R_P})--(\ref{R(k,eps)(I-P)_est}) into account, we obtain the following result.

\begin{thrm}
	\label{hatA(k)_exp_enchcd_thrm_1}
	Let $\widehat{N}(\boldsymbol{\theta})$ be the operator defined in~\emph{(\ref{hatN(theta)})}. Suppose that $\widehat{N}(\boldsymbol{\theta})~=~0$ for any $\boldsymbol{\theta} \in \mathbb{S}^{d-1}$. Then for $\tau \in \mathbb{R}$, $\varepsilon > 0$, and $\mathbf{k} \in \widetilde{\Omega}$ we have
	\begin{equation*}
	\| \widehat{J}(\mathbf{k}, \varepsilon; \tau ) \mathcal{R}(\mathbf{k}, \varepsilon)\|_{L_2(\Omega) \to L_2 (\Omega) }  \le \widehat{\mathcal{C}}_2 (1+ |\tau|^{1/2}) \varepsilon,
	\end{equation*}
	where the constant $\widehat{\mathcal{C}}_2$ depends only on $\alpha_0$, $\alpha_1$, $\|g\|_{L_\infty}$, $\|g^{-1}\|_{L_\infty}$, and $r_0$.
\end{thrm}

\subsection{The case where $\widehat{N}_0(\boldsymbol{\theta}) = 0$}
\label{ench_approx2_section}
Now, we reject the assumptions of Theorem~\ref{hatA(k)_exp_enchcd_thrm_1}, but we assume instead that $\widehat{N}_0(\boldsymbol{\theta}) = 0$ for all $\boldsymbol{\theta}$. We would like to apply the results of Theorem~\ref{abstr_exp_enchcd_thrm_2}. However, there is an additional difficulty: the multiplicities of the eigenvalues of the germ  $\widehat{S} (\boldsymbol{\theta})$ may change at some points $\boldsymbol{\theta}$. Near such points the distance between some pair of different eigenvalues  tends to zero and we are not able to choose the parameters $\widehat{c}^{\circ}_{jl}$, $\widehat{t}^{\,00}_{jl}$ to be independent on $\boldsymbol{\theta}$. Therefore, we are forced to impose additional conditions. We have to take care only about those eigenvalues for which the corresponding term in the second formula in~(\ref{N0_N*_invar_repr}) is not zero. Now, it is more convenient to use the initial enumeration of the eigenvalues  $\widehat{\gamma}_1 (\boldsymbol{\theta}), \ldots , \widehat{\gamma}_n (\boldsymbol{\theta})$ of $\widehat{S} (\boldsymbol{\theta})$ (each eigenvalue is repeated according to its multiplicity); we agree to enumerate them in the nondecreasing order: $\widehat{\gamma}_1 (\boldsymbol{\theta}) \le \widehat{\gamma}_2 (\boldsymbol{\theta}) \le \ldots \le \widehat{\gamma}_n (\boldsymbol{\theta})$. Denote by $\widehat{P}^{(k)} (\boldsymbol{\theta})$ the orthogonal projection of $L_2 (\Omega; \mathbb{C}^n)$ onto the eigenspace of $\widehat{S} (\boldsymbol{\theta})$ corresponding to the eigenvalue $\widehat{\gamma}_k (\boldsymbol{\theta})$. Clearly, for each $\boldsymbol{\theta}$ the operator $\widehat{P}^{(k)} (\boldsymbol{\theta})$ coincides with one of the projections $\widehat{P}_j (\boldsymbol{\theta})$ introduced in Subsection~\ref{eigenval_multipl_section} (but the number $j$ may depend on $\boldsymbol{\theta}$). 
\begin{condition}
	\label{cond1}
	\begin{enumerate*}[label=\emph{\arabic*$^{\circ}.$}, ref=\arabic*$^{\circ}$]
	\item $\widehat{N}_0(\boldsymbol{\theta})=0$ for any $\boldsymbol{\theta} \in \mathbb{S}^{d-1}$.
	\item \label{cond1_it2} For any pair of indices $(k,r)$, $1 \le k,r \le n$, $k \ne r$, such that $\widehat{\gamma}_k (\boldsymbol{\theta}_0) = \widehat{\gamma}_r (\boldsymbol{\theta}_0)$ for some $\boldsymbol{\theta}_0 \in \mathbb{S}^{d-1}$, we have  $\widehat{P}^{(k)} (\boldsymbol{\theta}) \widehat{N} (\boldsymbol{\theta}) \widehat{P}^{(r)} (\boldsymbol{\theta}) = 0$ for any $\boldsymbol{\theta} \in \mathbb{S}^{d-1}$.
	\end{enumerate*}    
\end{condition}

Condition~\ref{cond1}(\ref{cond1_it2}) can be reformulated: we assume that, for the \textquotedblleft blocks\textquotedblright \ $\widehat{P}^{(k)} (\boldsymbol{\theta}) \widehat{N} (\boldsymbol{\theta}) \widehat{P}^{(r)} (\boldsymbol{\theta})$ of the operator $\widehat{N} (\boldsymbol{\theta})$ that are not identically zero, the corresponding branches of the eigenvalues  $\widehat{\gamma}_k (\boldsymbol{\theta})$ and  $\widehat{\gamma}_r (\boldsymbol{\theta})$ do not intersect.

Obviously, Condition~\ref{cond1} is ensured by the following more restrictive condition.
\begin{condition}
	\label{cond2}	
	\begin{enumerate*}[label=\emph{\arabic*$^{\circ}.$}, ref=\arabic*$^{\circ}$]
	\item	$\widehat{N}_0(\boldsymbol{\theta})=0$ for any $\boldsymbol{\theta} \in \mathbb{S}^{d-1}$.
	\item \label{cond2_it2} The number $p$ of different eigenvalues of the germ $\widehat{S}(\boldsymbol{\theta})$ does not depend on $\boldsymbol{\theta} \in \mathbb{S}^{d-1}$.        
	\end{enumerate*}
\end{condition}

Under Condition~\ref{cond2} denote different eigenvalues of the germ enumerated in the increasing order by $\widehat{\gamma}^{\circ}_1(\boldsymbol{\theta}), \ldots, \widehat{\gamma}^{\circ}_p(\boldsymbol{\theta})$. Then their multiplicities $k_1, \ldots, k_p$ do not depend on $\boldsymbol{\theta} \in \mathbb{S}^{d-1}$.  

\begin{remark}
	\begin{enumerate*}[label=\emph{\arabic*$^{\circ}.$}, ref=\arabic*$^{\circ}$]
	\item Assumption~\emph{\ref{cond2_it2}} of Condition~\emph{\ref{cond2}} is a fortiori satisfied, if the spectrum of the germ $\widehat{S}(\boldsymbol{\theta})$ is simple for any $\boldsymbol{\theta} \in \mathbb{S}^{d-1}$.
	\item  From Corollary~\emph{\ref{real_S_spec_simple_coroll}} it follows that Condition~\emph{\ref{cond2}} is satisfied if the matrices 
	$b (\boldsymbol{\theta})$ and $g (\mathbf{x})$ have real entries and the spectrum of the germ $\widehat{S}(\boldsymbol{\theta})$ is simple for any $\boldsymbol{\theta} \in \mathbb{S}^{d-1}$.
	\end{enumerate*}
\end{remark}

So, we suppose that Condition~\ref{cond1} is fulfilled. Denote
\begin{equation*}
\widehat{\mathcal{K}} \coloneqq \{ (k,r): 1 \le k,r \le n, \;  k \ne r, \; \widehat{P}^{(k)} (\boldsymbol{\theta}) \widehat{N} (\boldsymbol{\theta}) \widehat{P}^{(r)} (\boldsymbol{\theta}) \not\equiv 0 \}.
\end{equation*}
Let $\widehat{c}^{\circ}_{kr} (\boldsymbol{\theta}) \coloneqq \min \{\widehat{c}_*, n^{-1} |\widehat{\gamma}_k (\boldsymbol{\theta}) - \widehat{\gamma}_r (\boldsymbol{\theta})| \}$, $(k,r) \in \widehat{\mathcal{K}}$. Since $\widehat{S} (\boldsymbol{\theta})$ depends on $\boldsymbol{\theta} \in \mathbb{S}^{d-1}$ continuously, then the perturbation theory implies that $\widehat{\gamma}_j (\boldsymbol{\theta})$ are continuous on $\mathbb{S}^{d-1}$. By Condition~\ref{cond1}$(2^\circ)$,  for $(k,r) \in \widehat{\mathcal{K}}$ we have~$|\widehat{\gamma}_k (\boldsymbol{\theta}) - \widehat{\gamma}_r (\boldsymbol{\theta})| > 0$ for any $\boldsymbol{\theta} \in \mathbb{S}^{d-1}$, whence $\widehat{c}^{\circ}_{kr} \coloneqq \min_{\boldsymbol{\theta} \in \mathbb{S}^{d-1}} \widehat{c}^{\circ}_{kr} (\boldsymbol{\theta}) > 0$ for $(k,r) \in \widehat{\mathcal{K}}$. We put
\begin{equation}
\label{hatc^circ}
\widehat{c}^{\circ} \coloneqq \min_{(k,r) \in \widehat{\mathcal{K}}} \widehat{c}^{\circ}_{kr}.
\end{equation}

Clearly, the number~(\ref{hatc^circ}) is a realization of~(\ref{abstr_c^circ}) chosen independently of $\boldsymbol{\theta}$.
Under Condition~\ref{cond1} the number $\widehat{t}^{\,00}$ subject to~(\ref{abstr_t00}) also can be chosen independently of $\boldsymbol{\theta} \in \mathbb{S}^{d-1}$. Taking~(\ref{delta_fixation}) and~(\ref{X_1_estimate}) into account (with $f = \mathbf{1}_n$), we put
\begin{equation*}
\widehat{t}^{\,00} = (8 \beta_2)^{-1} r_0 \alpha_1^{-3/2} \alpha_0^{1/2} \| g\|_{L_{\infty}}^{-3/2} \| g^{-1}\|_{L_{\infty}}^{-1/2} \widehat{c}^{\circ},
\end{equation*}
where $\widehat{c}^{\circ}$ is defined in~(\ref{hatc^circ}). (The condition $\widehat{t}^{\,00} \le \widehat{t}^{\,0}$ is valid automatically, since $\widehat{c}^{\circ} \le \| \widehat{S} (\boldsymbol{\theta}) \| \le \alpha_1 \|g\|_{L_{\infty}}$.)

\begin{remark}
	Unlike $\widehat{t}^{\,0}$ \emph{(}see~\emph{(\ref{t0_fixation})} with $f = \mathbf{1}_n$\emph{)} that is controlled only in terms of $r_0$, $\alpha_0$, $\alpha_1$, $\|g\|_{L_{\infty}}$ and $\|g^{-1}\|_{L_{\infty}}$, the number $\widehat{t}^{\,00}$ depends on the spectral characteristics of the germ~--- on the minimal distance between its different eigenvalues $\widehat{\gamma}_k (\boldsymbol{\theta})$ and $ \widehat{\gamma}_r (\boldsymbol{\theta})$ \emph{(}where $(k,r)$ runs through $\widehat{\mathcal{K}}$\emph{)}.
\end{remark}

Applying Theorem~\ref{abstr_exp_enchcd_thrm_2}, we deduce the following result.
\begin{thrm}
	\label{hatA(k)_exp_enchcd_thrm_2}
	Suppose that Condition~\emph{\ref{cond1}} \emph{(}or more restrictive Condition~\emph{\ref{cond2}}\emph{)} is satisfied. Then for $\tau \in \mathbb{R}$, $\varepsilon > 0$, and $\mathbf{k} \in \widetilde{\Omega}$ we have
	\begin{equation*}
	\| \widehat{J}(\mathbf{k}, \varepsilon; \tau ) \mathcal{R}(\mathbf{k}, \varepsilon)\|_{L_2(\Omega) \to L_2 (\Omega)}  \le \widehat{\mathcal{C}}_3 (1+ |\tau|^{1/2}) \varepsilon,
	\end{equation*}
	where the constant $\widehat{\mathcal{C}}_3$ depends only on $\alpha_0$, $\alpha_1$, $\|g\|_{L_\infty}$, $\|g^{-1}\|_{L_\infty}$, $r_0$, and also on $n$ and the number $\widehat{c}^{\circ}$.
\end{thrm}

\subsection{The sharpness of the results with respect to the smoothing operator}
Application of Theorems~\ref{abstr_exp_smooth_shrp_thrm_1}, \ref{abstr_exp_smooth_shrp_thrm_2} allows us to confirm that the results of Theorems~\ref{hatA(k)_exp_general_thrm}, \ref{hatA(k)_exp_enchcd_thrm_1}, \ref{hatA(k)_exp_enchcd_thrm_2} are sharp with respect to the
smoothing operator.

\begin{thrm}[\cite{Su2017}]
	\label{hatA(k)_exp_smooth_shrp_thrm_1}
	Suppose that $\widehat{N}_0 (\boldsymbol{\theta}_0) \ne 0$ for some $\boldsymbol{\theta}_0 \in \mathbb{S}^{d-1}$. Let $\tau \ne 0$ and $0 \le s < 3$. Then there does not exist a constant $\mathcal{C}(\tau) >0$ such that the estimate
	\begin{equation}
	\label{6.6a}
	\bigl\| \bigl( e^{-i \tau \varepsilon^{-2} \widehat{\mathcal{A}} (\mathbf{k})}  - e^{-i \tau \varepsilon^{-2} \widehat{\mathcal{A}}^0 (\mathbf{k})} \bigr) \mathcal{R} (\mathbf{k}, \varepsilon)^{s/2} \bigr\|_{L_2(\Omega) \to L_2 (\Omega)} \le \mathcal{C}(\tau) \varepsilon
	\end{equation}
	holds for almost~every $\mathbf{k}  \in \widetilde{\Omega}$ and sufficiently small $\varepsilon > 0$.
\end{thrm}

\begin{thrm}
	\label{hatA(k)_exp_smooth_shrp_thrm_2}
	Suppose that $\widehat{N}_0 (\boldsymbol{\theta}) = 0$ for any $\boldsymbol{\theta} \in \mathbb{S}^{d-1}$ and  $\widehat{\mathcal{N}}^{(q)} (\boldsymbol{\theta}_0) \ne 0$  for some $\boldsymbol{\theta}_0 \in \mathbb{S}^{d-1}$ and some  $q \in \{1,\ldots,p(\boldsymbol{\theta}_0)\}$.
	Let $\tau \ne 0$ and $0 \le s < 2$. Then there does not exist a constant $\mathcal{C}(\tau) >0$ such that estimate \eqref{6.6a}
	holds for almost~every $\mathbf{k} \in \widetilde{\Omega}$ and sufficiently small $\varepsilon > 0$.
\end{thrm}

Theorem~\ref{hatA(k)_exp_smooth_shrp_thrm_1} was proved in~\cite[Theorem~9.8]{Su2017}. Theorem~\ref{hatA(k)_exp_smooth_shrp_thrm_2} is proved with the help of Theorem~\ref{abstr_exp_smooth_shrp_thrm_2}  in a similar way as~\cite[Theorem~9.8]{Su2017}.

\subsection{The sharpness of the results with respect to dependence of the estimates on time}
Application of Theorems~\ref{abstr_exp_time_shrp_thrm_1}, \ref{abstr_exp_time_shrp_thrm_2} allows us to confirm that the results of Theorems~\ref{hatA(k)_exp_general_thrm}, \ref{hatA(k)_exp_enchcd_thrm_1}, \ref{hatA(k)_exp_enchcd_thrm_2} are sharp with respect to dependence of the estimates on time.

\begin{thrm}
	\label{hatA(k)_exp_time_shrp_thrm_1}
	Suppose that $\widehat{N}_0 (\boldsymbol{\theta}_0) \ne 0$ for some $\boldsymbol{\theta}_0 \in \mathbb{S}^{d-1}$. Let $s \ge 3$. 
	Then there does not exist a positive function $\mathcal{C}(\tau)$ such that $\lim_{\tau \to \infty} \mathcal{C}(\tau)/ |\tau| = 0$ and estimate
	\eqref{6.6a} holds for all $\tau \in \mathbb{R}$, almost~every $\mathbf{k} \in \widetilde{\Omega}$, and sufficiently small $\varepsilon > 0$.
\end{thrm}

\begin{proof} We prove by contradiction. Suppose that for some $s \ge 3$ there exists a function $\mathcal{C}(\tau) > 0$ such that $\lim_{\tau \to \infty} \mathcal{C}(\tau)/ |\tau| = 0$ and estimate~(\ref{6.6a}) holds for almost~every $\mathbf{k} \in \widetilde{\Omega}$ and sufficiently small $\varepsilon > 0$. By~(\ref{R_P}), (\ref{R(k,eps)(I-P)_est}), and the estimate 
	\begin{equation}
	\label{hatA(k)_exp_shrp_f1}
	\bigl\| \widehat{F} (\mathbf{k}) - \widehat{P} \bigr\|_{L_2 (\Omega) \to L_2 (\Omega)} \le \widehat{C}_1 |\mathbf{k}|, \qquad |\mathbf{k}| \le \widehat{t}^{\,0},
	\end{equation}
	(see~(\ref{abstr_F(t)_threshold})), it follows that there exists a function $\widetilde{\mathcal{C}}(\tau) > 0$ such that $\lim_{\tau \to \infty} \widetilde{\mathcal{C}}(\tau)/ |\tau| = 0$ and the estimate
	\begin{equation}
	\label{hatA(k)_exp_shrp_f2}
	\bigl\|  e^{-i \tau \varepsilon^{-2} \widehat{\mathcal{A}} ( \mathbf{k})} \widehat{F} (\mathbf{k})  - e^{-i \tau \varepsilon^{-2} \widehat{\mathcal{A}}^0 (\mathbf{k})} \widehat{P} \bigr\|_{L_2(\Omega) \to L_2 (\Omega) } \varepsilon^s (|\mathbf{k}|^2 + \varepsilon^2)^{-s/2}  \le \widetilde{\mathcal{C}}(\tau) \varepsilon
	\end{equation}
	holds for almost~every $\mathbf{k} \in \widetilde{\Omega}$ in the ball $|\mathbf{k}| \le \widehat{t}^{\,0}$ and sufficiently small $\varepsilon > 0$.
	For fixed $\tau$ and $\varepsilon$, the operator under the norm sign in~(\ref{hatA(k)_exp_shrp_f2}) is continuous with respect to $\mathbf{k}$ in the ball $|\mathbf{k}| \le \widehat{t}^{\,0}$  (see~\cite[Lemma~9.9]{Su2017}). Hence, estimate~(\ref{hatA(k)_exp_shrp_f2}) holds for all $\mathbf{k}$ in this ball, in particular, for $\mathbf{k} = t\boldsymbol{\theta}_0$ if $t \le \widehat{t}^{\,0}$. Applying~(\ref{hatA(k)_exp_shrp_f1}) once more, we see that
	\begin{equation}
	\label{hatA(k)_exp_shrp_f3}
	\bigl\| \bigl( e^{-i \tau \varepsilon^{-2} \widehat{\mathcal{A}} ( t \boldsymbol{\theta}_0)}  - e^{-i \tau \varepsilon^{-2} \widehat{\mathcal{A}}^0 (t \boldsymbol{\theta}_0)} \bigr) \widehat{P} \bigr\|_{L_2(\Omega) \to L_2(\Omega)} \varepsilon^s (t^2 + \varepsilon^2)^{-s/2}  \le \check{\mathcal{C}}(\tau) \varepsilon
	\end{equation}
	for all $t \le \widehat{t}^{\,0}$ and sufficiently small $\varepsilon > 0$, where $\check{\mathcal{C}}(\tau) > 0$ and $\lim_{\tau \to \infty} \check{\mathcal{C}}(\tau)/ |\tau| = 0$.

	In the abstract terms, estimate~(\ref{hatA(k)_exp_shrp_f3}) corresponds to~(\ref{abstr_exp_shrp_est}). Since it is assumed that  \hbox{$\widehat{N}_{0}(\boldsymbol{\theta}_0) \ne 0$}, applying Theorem~\ref{abstr_exp_time_shrp_thrm_1}, we arrive at a contradiction.
\end{proof}

Similarly, application of Theorem~\ref{abstr_exp_time_shrp_thrm_2} allows us to confirm the sharpness of Theorems~\ref{hatA(k)_exp_enchcd_thrm_1} and \ref{hatA(k)_exp_enchcd_thrm_2}.

\begin{thrm}
	\label{hatA(k)_exp_time_shrp_thrm_2}
	Suppose that $\widehat{N}_0 (\boldsymbol{\theta}) = 0$ for any $\boldsymbol{\theta} \in \mathbb{S}^{d-1}$ and  $\widehat{\mathcal{N}}^{(q)} (\boldsymbol{\theta}_0) \ne 0$  for some $\boldsymbol{\theta}_0 \in \mathbb{S}^{d-1}$ and some $q \in \{1,\ldots,p(\boldsymbol{\theta}_0)\}$. Let $s \ge 2$. Then there does not exist a positive function $\mathcal{C}(\tau)$ such that $\lim_{\tau \to \infty} \mathcal{C}(\tau)/|\tau|^{1/2} = 0$ and estimate
	\eqref{6.6a} holds for all $\tau \in \mathbb{R}$, almost~every $\mathbf{k}  \in \widetilde{\Omega}$, and sufficiently 
	small $\varepsilon > 0$.
\end{thrm}

\section{The operator $\mathcal{A} (\mathbf{k})$. Application of the scheme of~\S\ref{abstr_sndw_section}}
\subsection{The operator $\mathcal{A} (\mathbf{k})$}

We apply the scheme of~\S\ref{abstr_sndw_section} to study the operator $\mathcal{A} (\mathbf{k}) = f^* \widehat{\mathcal{A}} (\mathbf{k}) f$. 
Now, $\mathfrak{H} = \widehat{\mathfrak{H}} = L_2 (\Omega; \mathbb{C}^n)$, $\mathfrak{H}_* = L_2 (\Omega; \mathbb{C}^m)$, the role of $A(t)$ is played by $A(t; \boldsymbol{\theta}) = \mathcal{A}(\mathbf{k})$, the role of $\widehat{A}(t)$ is played by $\widehat{A}(t; \boldsymbol{\theta}) = \widehat{\mathcal{A}}(\mathbf{k})$. The isomorphism $M$ is the operator of multiplication by the matrix-valued function $f(\mathbf{x})$. The operator $Q$ is the operator of multiplication by the matrix-valued function $Q(\mathbf{x}) = (f (\mathbf{x}) f (\mathbf{x})^*)^{-1}$. The block of $Q$ in the subspace $\widehat{\mathfrak{N}}$ (see~(\ref{Ker3})) is the operator of multiplication by the constant matrix $\overline{Q} = (\underline{f f^*})^{-1} = |\Omega|^{-1} \int_{\Omega} (f (\mathbf{x}) f (\mathbf{x})^*)^{-1} d \mathbf{x}$. Next, $M_0$ is the operator of multiplication by the constant matrix
\begin{equation}
\label{f0}
f_0 = (\overline{Q})^{-1/2} = (\underline{f f^*})^{1/2}.
\end{equation}
Note that $| f_0 | \le \| f \|_{L_{\infty}}$, $| f_0^{-1} | \le \| f^{-1} \|_{L_{\infty}}$.

In $L_2 (\mathbb{R}^d; \mathbb{C}^n)$, we define the operator
\begin{equation}
\label{A0}
\mathcal{A}^0 \coloneqq f_0 \widehat{\mathcal{A}}^0 f_0 = f_0 b(\mathbf{D})^* g^0 b(\mathbf{D}) f_0.
\end{equation}
Let $\mathcal{A}^0 (\mathbf{k})$ be the corresponding family of operators in $L_2 (\Omega; \mathbb{C}^n)$. Then $\mathcal{A}^0 (\mathbf{k}) = f_0 \widehat{\mathcal{A}}^0 (\mathbf{k}) f_0$. By~(\ref{Ker3}) and~(\ref{hatS_P=hatA^0_P}), 
\begin{equation}
\label{A^0(k)P}
f_0 \widehat{S} (\mathbf{k}) f_0 \widehat{P} = \mathcal{A}^0 (\mathbf{k}) \widehat{P}.
\end{equation}

\subsection{The analytic branches of eigenvalues and eigenvectors}
According to~(\ref{abstr_S_Shat}), the spectral germ $S(\boldsymbol{\theta})$ of the operator $A (t; \boldsymbol{\theta})$ acting in the subspace $\mathfrak{N}$ (see~(\ref{frakN})) is represented as $S(\boldsymbol{\theta}) = P f^* b(\boldsymbol{\theta})^* g^0 b(\boldsymbol{\theta}) f|_{\mathfrak{N}}$, where $P$ is the orthogonal projection of $L_2 (\Omega; \mathbb{C}^n)$ onto $\mathfrak{N}$.

The analytic (in $t$) branches of the eigenvalues $\lambda_l (t, \boldsymbol{\theta})$ and the eigenvectors $\varphi_l (t, \boldsymbol{\theta})$ of the operator $A (t; \boldsymbol{\theta})$ admit the power series expansions of the form~(\ref{abstr_A(t)_eigenvalues_series}), (\ref{abstr_A(t)_eigenvectors_series}) with the coefficients depending on $\boldsymbol{\theta}$:
\begin{align}
\label{A_eigenvalues_series}
\lambda_l (t, \boldsymbol{\theta}) &= \gamma_l (\boldsymbol{\theta}) t^2 + \mu_l (\boldsymbol{\theta}) t^3 + \nu_l (\boldsymbol{\theta}) t^4 + \ldots, & l &= 1, \ldots, n, 
\\
\label{A_eigenvectors_series}
\varphi_l (t, \boldsymbol{\theta}) &= \omega_l (\boldsymbol{\theta}) + t \psi^{(1)}_l (\boldsymbol{\theta}) + \ldots, & l &= 1, \ldots, n.
\end{align}

The vectors $\omega_1 (\boldsymbol{\theta}), \ldots, \omega_n (\boldsymbol{\theta})$ form an orthonormal basis in the subspace $\mathfrak{N}$, and the vectors $\zeta_l (\boldsymbol{\theta}) = f \omega_l (\boldsymbol{\theta})$, $l = 1, \ldots, n$, form a basis in $\widehat{\mathfrak{N}}$~(see~(\ref{Ker3})) orthonormal with the weight $\overline{Q}$. The numbers $\gamma_l (\boldsymbol{\theta})$ and the elements $\omega_l (\boldsymbol{\theta})$ are eigenvalues and eigenvectors of the spectral germ $S(\boldsymbol{\theta})$. According to~(\ref{abstr_hatS_gener_spec_problem}),
\begin{equation}
\label{hatS_gener_spec_problem}
b(\boldsymbol{\theta})^* g^0 b(\boldsymbol{\theta}) \zeta_l (\boldsymbol{\theta}) = \gamma_l (\boldsymbol{\theta}) \overline{Q} \zeta_l (\boldsymbol{\theta}), \qquad l = 1, \ldots, n.
\end{equation}

\subsection{The operator $\widehat{N}_Q (\boldsymbol{\theta})$}

We need to describe the operator $\widehat{N}_Q$ (see Subsection~\ref{abstr_hatZ_Q_and_hatN_Q_section}). Let $\Lambda_Q(\mathbf{x})$ be a $\Gamma$-periodic solution of the problem
\begin{equation*}
b(\mathbf{D})^* g(\mathbf{x}) (b(\mathbf{D}) \Lambda_Q(\mathbf{x}) + \mathbf{1}_m) = 0, \qquad \int_{\Omega} Q(\mathbf{x}) \Lambda_Q(\mathbf{x}) \, d \mathbf{x} = 0.
\end{equation*}
Clearly, $\Lambda_Q(\mathbf{x}) = \Lambda(\mathbf{x}) - (\overline{Q})^{-1} (\overline{Q \Lambda})$. As shown in~\cite[\S5]{BSu2005-2}, the operator $\widehat{N}_Q (\boldsymbol{\theta})$ takes the form
\begin{align}
\label{N_Q(theta)}
\widehat{N}_Q (\boldsymbol{\theta}) &= b(\boldsymbol{\theta})^* L_Q (\boldsymbol{\theta}) b(\boldsymbol{\theta}) \widehat{P},\\
\notag
L_Q (\boldsymbol{\theta}) &\coloneqq | \Omega |^{-1} \int_{\Omega} (\Lambda_Q(\mathbf{x})^*b(\boldsymbol{\theta})^* \widetilde{g} (\mathbf{x}) + \widetilde{g} (\mathbf{x})^* b(\boldsymbol{\theta}) \Lambda_Q(\mathbf{x}))\, d \mathbf{x}.
\end{align}
Some sufficient conditions where $\widehat{N}_Q (\boldsymbol{\theta}) = 0$  were distinguished in~\cite[\S5]{BSu2005-2}.

\begin{proposition}[\cite{BSu2005-2}]
	\label{hatN_Q=0_proposit}	
	Suppose that at least one of the following conditions is fulfilled\emph{:}	
	\begin{enumerate}[label=\emph{\arabic*$^{\circ}.$}, ref=\arabic*$^{\circ}$,parsep=-3pt]
	\item \label{hatN_Q=0_proposit_it1} $\mathcal{A} = f(\mathbf{x})^*\mathbf{D}^* g(\mathbf{x}) \mathbf{D}f(\mathbf{x})$, where $g(\mathbf{x})$ is a symmetric matrix with real entries.
	\item Relations~\emph{(\ref{g0=overline_g_relat})} are satisfied, i.e. $g^0 = \overline{g}$.
	\end{enumerate}
	Then $\widehat{N}_Q (\boldsymbol{\theta}) = 0$ for any $\boldsymbol{\theta} \in \mathbb{S}^{d-1}$.
\end{proposition}

Recall (see Subsection~\ref{abstr_hatZ_Q_and_hatN_Q_section}) that  $\widehat{N}_Q (\boldsymbol{\theta}) = \widehat{N}_{0, Q} (\boldsymbol{\theta}) + \widehat{N}_{*,Q} (\boldsymbol{\theta})$. By~(\ref{abstr_hatN_0Q_N_*Q}),
\begin{equation*}
\widehat{N}_{0, Q} (\boldsymbol{\theta}) = \sum_{l=1}^{n} \mu_l (\boldsymbol{\theta}) (\cdot, \overline{Q} \zeta_l(\boldsymbol{\theta}))_{L_2(\Omega)} \overline{Q} \zeta_l(\boldsymbol{\theta}).
\end{equation*}
We have
\begin{equation*}
(\widehat{N}_Q (\boldsymbol{\theta}) \zeta_l (\boldsymbol{\theta}), \zeta_l (\boldsymbol{\theta}))_{L_2 (\Omega)} = (\widehat{N}_{0,Q} (\boldsymbol{\theta}) \zeta_l (\boldsymbol{\theta}), \zeta_l (\boldsymbol{\theta}))_{L_2 (\Omega)} = \mu_l (\boldsymbol{\theta}), \qquad l=1, \ldots, n.
\end{equation*}

In~\cite[Proposition~5.2]{BSu2005-2} the following proposition was proved.
\begin{proposition}[\cite{BSu2005-2}]
	Suppose that the matrices $b(\boldsymbol{\theta})$, $g (\mathbf{x})$, and $Q(\mathbf{x})$ have real entries. Suppose that in the expansions~\emph{(\ref{A_eigenvectors_series})} the  \textquotedblleft embryos\textquotedblright \ $\omega_l (\boldsymbol{\theta})$, $l = 1, \ldots, n$, can be chosen so that the vectors  $\zeta_l (\boldsymbol{\theta}) = f \omega_l (\boldsymbol{\theta})$ are real. Then in~\emph{(\ref{A_eigenvalues_series})} we have $\mu_l (\boldsymbol{\theta}) = 0$, $l=1, \ldots, n$, i.e., $\widehat{N}_{0,Q} (\boldsymbol{\theta}) = 0$ for any $\boldsymbol{\theta} \in \mathbb{S}^{d-1}$.
\end{proposition}

In the \textquotedblleft real\textquotedblright \ case under consideration,  $\widehat{S} (\boldsymbol{\theta})$ and $\overline{Q}$ are symmetric matrices with real entries. Clearly, if the eigenvalue $\gamma_j (\boldsymbol{\theta})$ of the generalized spectral problem~(\ref{hatS_gener_spec_problem}) is simple, then the vector $\zeta_j (\boldsymbol{\theta}) = f \omega_j (\boldsymbol{\theta})$ is defined uniquely up to a phase factor, so we can always choose it to be real. We arrive at the following corollary.
\begin{corollary}
	\label{sndw_real_spec_simple_coroll}
	Suppose that the matrices $b(\boldsymbol{\theta})$, $g (\mathbf{x})$, and $Q(\mathbf{x})$ have real entries and the spectrum of the generalized spectral problem~\emph{(\ref{hatS_gener_spec_problem})} is simple. Then $\widehat{N}_{0,Q} (\boldsymbol{\theta}) = 0$ for any $\boldsymbol{\theta} \in \mathbb{S}^{d-1}$.
\end{corollary}

\subsection{The operators $\widehat{Z}_{2,Q}(\boldsymbol{\theta})$, $\widehat{R}_{2,Q}(\boldsymbol{\theta})$, $\widehat{N}_{1,Q}^0(\boldsymbol{\theta})$}
We need to describe the operators $\widehat{Z}_{2,Q}$, $\widehat{R}_{2,Q}$, $\widehat{N}_{1,Q}^0$ (which in the abstract terms are defined in Subsection~\ref{abstr_hatZ2_Q_hatR2_Q_N1^0_Q_section}). Let $\Lambda^{(2)}_{Q,l} (\mathbf{x})$ be a $\Gamma$-periodic solution of the problem
\begin{equation*}
b(\mathbf{D})^* g(\mathbf{x}) (b(\mathbf{D}) \Lambda^{(2)}_{Q,l} (\mathbf{x}) + b_l \Lambda_{Q}(\mathbf{x})) = -b_l^* \widetilde{g} (\mathbf{x}) + Q(\mathbf{x}) (\overline{Q})^{-1} b_l^* g^0, \qquad \int_{\Omega} Q(\mathbf{x}) \Lambda^{(2)}_{Q,l} (\mathbf{x}) \, d \mathbf{x} = 0.
\end{equation*}
We put $\Lambda^{(2)}_Q (\mathbf{x}; \boldsymbol{\theta}) \coloneqq \sum_{l=1}^{d} \Lambda^{(2)}_{Q,l} (\mathbf{x}) \theta_l$. In~\cite[Subsection~8.4]{VSu2012}, it was shown that
\begin{gather*}
\widehat{Z}_{2,Q}(\boldsymbol{\theta}) = \Lambda_Q^{(2)} (\mathbf{x}; \boldsymbol{\theta}) b(\boldsymbol{\theta}) \widehat{P}, \qquad
\widehat{R}_{2,Q}(\boldsymbol{\theta}) = h(\mathbf{x}) (b(\mathbf{D}) \Lambda^{(2)}_Q (\mathbf{x}; \boldsymbol{\theta}) + b(\boldsymbol{\theta}) \Lambda_Q(\mathbf{x})) b(\boldsymbol{\theta}).
\end{gather*}
Finally, in~\cite[Subsection~8.5]{VSu2012} the following representation was obtained:
\begin{align*}
&\widehat{N}_{1,Q}^0(\boldsymbol{\theta}) = b(\boldsymbol{\theta})^* L_{2,Q}(\boldsymbol{\theta}) b(\boldsymbol{\theta}) \widehat{P}, 
\\
&\begin{multlined}[c]
L_{2,Q}(\boldsymbol{\theta}) \coloneqq |\Omega|^{-1} \int_{\Omega} \bigl(\Lambda^{(2)}_Q (\mathbf{x}; \boldsymbol{\theta})^* b(\boldsymbol{\theta})^* \widetilde{g}(\mathbf{x}) + \widetilde{g}(\mathbf{x})^* b(\boldsymbol{\theta}) \Lambda^{(2)}_Q (\mathbf{x}; \boldsymbol{\theta}) \bigr) \, d \mathbf{x}  \\ 
+
|\Omega|^{-1} \int_{\Omega} \bigl(b(\mathbf{D}) \Lambda^{(2)}_Q (\mathbf{x}; \boldsymbol{\theta}) + b(\boldsymbol{\theta}) \Lambda_Q (\mathbf{x})\bigr)^* g(\mathbf{x}) \bigl(b(\mathbf{D}) \Lambda^{(2)}_Q (\mathbf{x}; \boldsymbol{\theta}) + b(\boldsymbol{\theta}) \Lambda_Q(\mathbf{x}) \bigr) \, d \mathbf{x}.
\end{multlined}
\end{align*}

\subsection{Multiplicities of the eigenvalues of the germ}
\label{sndw_eigenval_multipl_section}
Considerations of this subsection concern the case where $n \ge 2$. Now, we return to the notation of Subsection~\ref{abstr_cluster_section}. In general, the number $p(\boldsymbol{\theta})$ of different eigenvalues $\gamma^{\circ}_1 (\boldsymbol{\theta}), \ldots, \gamma^{\circ}_{p(\boldsymbol{\theta})} (\boldsymbol{\theta})$ of $S(\boldsymbol{\theta})$ (or of problem~(\ref{hatS_gener_spec_problem})) and their multiplicities $k_1 (\boldsymbol{\theta}), \ldots, k_{p(\boldsymbol{\theta})} (\boldsymbol{\theta})$ depend on the parameter $\boldsymbol{\theta} \in \mathbb{S}^{d-1}$. For a fixed $\boldsymbol{\theta}$ denote by $\mathfrak{N}_j (\boldsymbol{\theta})$ the eigenspace of the germ $S (\boldsymbol{\theta})$ corresponding to the eigenvalue $\gamma^{\circ}_j (\boldsymbol{\theta})$. Then $f \mathfrak{N}_j (\boldsymbol{\theta}) \eqqcolon \widehat{\mathfrak{N}}_{j,Q} (\boldsymbol{\theta})$ is the eigenspace of  problem~(\ref{hatS_gener_spec_problem}) corresponding to the same eigenvalue $\gamma^{\circ}_j (\boldsymbol{\theta})$. We introduce the notation $\mathcal{P}_j (\boldsymbol{\theta})$ for the \textquotedblleft skew\textquotedblright \ projection of $L_2(\Omega; \mathbb{C}^n)$ onto the subspace $\widehat{\mathfrak{N}}_{j,Q} (\boldsymbol{\theta})$; $\mathcal{P}_j (\boldsymbol{\theta})$ is orthogonal with respect to the inner product with the weight $\overline{Q}$. By~(\ref{abstr_hatN_0Q_N_*Q_invar_repr}),
\begin{equation*}
\widehat{N}_{0,Q} (\boldsymbol{\theta}) = \sum_{j=1}^{p(\boldsymbol{\theta})} \mathcal{P}_j (\boldsymbol{\theta})^* \widehat{N}_Q (\boldsymbol{\theta}) \mathcal{P}_j (\boldsymbol{\theta}), \qquad \widehat{N}_{*,Q} (\boldsymbol{\theta}) = \sum_{\substack{1 \le l,j \le p(\boldsymbol{\theta})\\ j \ne l}} \mathcal{P}_j (\boldsymbol{\theta})^* \widehat{N}_Q (\boldsymbol{\theta}) \mathcal{P}_l (\boldsymbol{\theta}).
\end{equation*}

\subsection{The coefficients $\nu_l(\boldsymbol{\theta})$, $l=1, \ldots,n$}
According to~(\ref{abstr_N_eigenvalues}), the numbers $\mu_l (\boldsymbol{\theta})$ and the elements $\omega_l (\boldsymbol{\theta})$, $l = i(q,\boldsymbol{\theta}),\ldots, i(q,\boldsymbol{\theta})+k_q(\boldsymbol{\theta})-1$, where $i(q,\boldsymbol{\theta}) = k_1(\boldsymbol{\theta})+\ldots+k_{q-1}(\boldsymbol{\theta})+1$,  are eigenvalues and eigenvectors of the operator $P_q (\boldsymbol{\theta}) N(\boldsymbol{\theta}) |_{\mathfrak{N}_q(\boldsymbol{\theta})}$. Then, by~(\ref{abstr_hatN_Q_gener_spec_problem}), we have
\begin{equation}
\label{hatN_Q_gener_spec_problem}
\widehat{P}_{q,Q} (\boldsymbol{\theta}) \widehat{N}_Q (\boldsymbol{\theta}) \zeta_l (\boldsymbol{\theta}) = \mu_l (\boldsymbol{\theta}) \widehat{P}_{q,Q}(\boldsymbol{\theta}) \overline{Q} \zeta_l (\boldsymbol{\theta}), \qquad l = i(q,\boldsymbol{\theta}),\ldots, i(q,\boldsymbol{\theta})+k_q(\boldsymbol{\theta})-1,
\end{equation} 
where $\widehat{P}_{q,Q}(\boldsymbol{\theta})$ is the orthogonal projection onto $\widehat{\mathfrak{N}}_{q,Q} (\boldsymbol{\theta})$.

The number $p'(q, \boldsymbol{\theta})$ of different eigenvalues $\mu^{\circ}_{1,q} (\boldsymbol{\theta}), \ldots, \mu^{\circ}_{p'(q, \boldsymbol{\theta}),q} (\boldsymbol{\theta})$ of the operator $P_q (\boldsymbol{\theta}) N(\boldsymbol{\theta}) |_{\mathfrak{N}_q(\boldsymbol{\theta})}$ 
and their multiplicities $k_{1,q} (\boldsymbol{\theta}), \ldots, k_{p'(\boldsymbol{\theta}),q} (\boldsymbol{\theta})$ depend on the parameter $\boldsymbol{\theta} \in \mathbb{S}^{d-1}$.
For a fixed $\boldsymbol{\theta}$ we denote by $\mathfrak{N}_{q',q} (\boldsymbol{\theta})$ the eigenspace corresponding to the eigenvalue $\mu^{\circ}_{q',q} (\boldsymbol{\theta})$. Then $f \mathfrak{N}_{q',q} (\boldsymbol{\theta}) \eqqcolon  \widehat{\mathfrak{N}}_{q',q,Q} (\boldsymbol{\theta})$ is the eigenspace of problem~(\ref{hatN_Q_gener_spec_problem}) corresponding to the same eigenvalue $\mu^{\circ}_{q',q} (\boldsymbol{\theta})$. 

Finally, according to~(\ref{abstr_scrNhat_M^(q)_gener_spec_problem}), the numbers $\nu_l (\boldsymbol{\theta})$ and the elements $\zeta_l (\boldsymbol{\theta})$, $l = i'(q',q,\boldsymbol{\theta}), {\ldots}, i'(q',q,\boldsymbol{\theta})+k_{q',q}(\boldsymbol{\theta})-1$, where $i'(q',q,\boldsymbol{\theta}) = i(q, \boldsymbol{\theta}) +k_{1,q}(\boldsymbol{\theta})+\ldots+k_{q'-1,q}(\boldsymbol{\theta})$, are the eigenvalues and the eigenvectors of the following generalized spectral problem:
\begin{equation*}
\widehat{\mathcal{N}}_Q^{(q',q)}(\boldsymbol{\theta}) \zeta_l(\boldsymbol{\theta}) = \nu_l(\boldsymbol{\theta}) \widehat{P}_{q',q,Q} (\boldsymbol{\theta}) \overline{Q} \zeta_l(\boldsymbol{\theta}), \qquad l = i'(q',q,\boldsymbol{\theta}),\ldots, i'(q',q,\boldsymbol{\theta})+k_q(\boldsymbol{\theta})-1,
\end{equation*}
where
\begin{multline*}
\widehat{\mathcal{N}}_Q^{(q',q)}(\boldsymbol{\theta}) \coloneqq 
\\ 
\widehat{P}_{q',q,Q} (\boldsymbol{\theta}) \left. \left(\widehat{N}_{1,Q}^0(\boldsymbol{\theta}) - \frac{1}{2} \widehat{Z}_Q^*(\boldsymbol{\theta}) Q \widehat{Z}_Q(\boldsymbol{\theta}) Q^{-1} \widehat{S}(\boldsymbol{\theta}) \widehat{P} - \frac{1}{2} \widehat{S}(\boldsymbol{\theta}) \widehat{P} Q^{-1} \widehat{Z}_Q(\boldsymbol{\theta})^* Q \widehat{Z}_Q(\boldsymbol{\theta}) \right)\right|_{\widehat{\mathfrak{N}}_{q',q,Q} (\boldsymbol{\theta})} 
\\ +
\sum_{\substack{j\in\{1,\ldots,p(\boldsymbol{\theta})\} \\ j \ne q}} \bigl(\gamma^{\circ}_q(\boldsymbol{\theta}) - \gamma^{\circ}_j(\boldsymbol{\theta})\bigr)^{-1} \widehat{P}_{q',q,Q}(\boldsymbol{\theta}) \widehat{N}_Q(\boldsymbol{\theta}) \widehat{P}_{j,Q}(\boldsymbol{\theta}) Q^{-1} \widehat{P}_{j,Q}(\boldsymbol{\theta}) \widehat{N}_Q(\boldsymbol{\theta})|_{\widehat{\mathfrak{N}}_{q',q,Q} (\boldsymbol{\theta})}
\end{multline*}
and $\widehat{P}_{q',q,Q}(\boldsymbol{\theta})$ is the orthogonal projection onto $\widehat{\mathfrak{N}}_{q',q,Q} (\boldsymbol{\theta})$.

Note that in the case where $\widehat{N}_{0,Q}(\boldsymbol{\theta}) = 0$ we have $\widehat{\mathfrak{N}}_{1,q,Q} (\boldsymbol{\theta}) = \widehat{\mathfrak{N}}_{q,Q} (\boldsymbol{\theta})$, $q=1, \ldots,p(\boldsymbol{\theta})$. Then we shall write $\widehat{\mathcal{N}}_Q^{(q)}(\boldsymbol{\theta})$ instead of $\widehat{\mathcal{N}}_Q^{(1,q)}(\boldsymbol{\theta})$.

\section{Approximations for the sandwiched operator $e^{-i \tau \varepsilon^{-2}  \mathcal{A}(\mathbf{k})}$}

\subsection{The general case}
Denote
\begin{equation}
\label{J(k,eps)}
J (\mathbf{k}, \varepsilon; \tau) \coloneqq f e^{-i \tau \varepsilon^{-2} \mathcal{A} (\mathbf{k})} f^{-1} - f_0 e^{-i \tau \varepsilon^{-2} \mathcal{A}^0 (\mathbf{k})} f_0^{-1}.
\end{equation}

We shall apply theorems of Subsection~\ref{abstr_sndw_exp_section} to the operator $A(t; \boldsymbol{\theta}) = \mathcal{A}(\mathbf{k})$. In doing so, we may trace the dependence of the constants in estimates on the problem data. Note that $c_*$, $\delta$, and $t^0$ do not depend on $\boldsymbol{\theta}$ (see~(\ref{c_*}), (\ref{delta_fixation}), (\ref{t0_fixation})). According to~(\ref{X_1_estimate}) the norm $\| X_1 (\boldsymbol{\theta}) \|$ can be replaced by $\alpha_1^{1/2} \| g \|_{L_{\infty}}^{1/2} \| f \|_{L_{\infty}}$. Hence, the constants in Theorems~\ref{abstr_sndw_exp_general_thrm} and~\ref{abstr_sndw_exp_enchcd_thrm_1} (applied to the operator $\mathcal{A}(\mathbf{k})$) will be independent of $\boldsymbol{\theta}$. They depend only on $\alpha_0$, $\alpha_1$, $\|g\|_{L_\infty}$, $\|g^{-1}\|_{L_\infty}$, $\|f\|_{L_{\infty}}$, $\|f^{-1}\|_{L_{\infty}}$, and $r_0$.

Applying Theorem~\ref{abstr_sndw_exp_general_thrm} and taking~(\ref{R_P})--(\ref{R(k,eps)(I-P)_est}), (\ref{A^0(k)P}) into account, we arrive at the following statement proved before in~\cite[Theorem~8.1]{BSu2008}.

\begin{thrm}[\cite{BSu2008}]
	\label{sndw_A(k)_exp_general_thrm}
	For $\tau \in \mathbb{R}$, $\varepsilon > 0$, and $\mathbf{k} \in \widetilde{\Omega}$ we have
	\begin{equation*}
	\| J (\mathbf{k}, \varepsilon; \tau) \mathcal{R}(\mathbf{k}, \varepsilon)^{3/2}\|_{L_2(\Omega) \to L_2 (\Omega) }  \le \mathcal{C}_1 (1 + |\tau|) \varepsilon,
	\end{equation*}
	where the constant $\mathcal{C}_1$ depends only on $\alpha_0$, $\alpha_1$, $\|g\|_{L_\infty}$, $\|g^{-1}\|_{L_\infty}$, $\|f\|_{L_{\infty}}$, $\|f^{-1}\|_{L_{\infty}}$, and $r_0$.
\end{thrm}

\subsection{The case where $\widehat{N}_Q(\boldsymbol{\theta}) = 0$}
Now, we apply Theorem~\ref{abstr_sndw_exp_enchcd_thrm_1} assuming that $\widehat{N}_Q(\boldsymbol{\theta}) = 0$ for any $\boldsymbol{\theta} \in \mathbb{S}^{d-1}$. Taking~(\ref{R_P})--(\ref{R(k,eps)(I-P)_est}), (\ref{A^0(k)P}) into account, we obtain the following result.

\begin{thrm}
	\label{sndw_A(k)_exp_enchcd_thrm_1}
	Let $\widehat{N}_Q(\boldsymbol{\theta})$ be the operator defined by~\emph{(\ref{N_Q(theta)})}. Suppose that $\widehat{N}_Q(\boldsymbol{\theta}) = 0$ for any $\boldsymbol{\theta} \in \mathbb{S}^{d-1}$. Then for $\tau \in \mathbb{R}$, $\varepsilon > 0$, and $\mathbf{k} \in \widetilde{\Omega}$ we have
	\begin{equation*}
	\| J (\mathbf{k}, \varepsilon; \tau) \mathcal{R}(\mathbf{k}, \varepsilon)\|_{L_2(\Omega) \to L_2 (\Omega)}  \le \mathcal{C}_2 (1+ |\tau|^{1/2}) \varepsilon,
	\end{equation*}
	where the constant $\mathcal{C}_2$ depends only on $\alpha_0$, $\alpha_1$, $\|g\|_{L_\infty}$, $\|g^{-1}\|_{L_\infty}$, $\|f\|_{L_{\infty}}$, $\|f^{-1}\|_{L_{\infty}}$, and $r_0$.
\end{thrm}

\subsection{The case where $\widehat{N}_{0,Q}(\boldsymbol{\theta}) = 0$}
Now, we reject the assumptions of Theorem~\ref{sndw_A(k)_exp_enchcd_thrm_1}, but we assume instead that $\widehat{N}_{0,Q}(\boldsymbol{\theta}) = 0$ for any $\boldsymbol{\theta}$. As in Subsection~\ref{ench_approx2_section}, in order to apply Theorem~\ref{abstr_sndw_exp_enchcd_thrm_2}, we have to impose some additional conditions. We use the initial enumeration of the eigenvalues $\gamma_1 (\boldsymbol{\theta}) \le \ldots \le \gamma_n (\boldsymbol{\theta})$ of the germ $S (\boldsymbol{\theta})$. They are also the eigenvalues of the generalized spectral problem~(\ref{hatS_gener_spec_problem}). For each $\boldsymbol{\theta}$, let $\mathcal{P}^{(k)} (\boldsymbol{\theta})$ be the \textquotedblleft skew\textquotedblright \ projection (orthogonal with the weight $\overline{Q}$) of $L_2 (\Omega; \mathbb{C}^n)$ onto the eigenspace of  problem~(\ref{hatS_gener_spec_problem}) corresponding to the eigenvalue $\gamma_k (\boldsymbol{\theta})$. Clearly, for each $\boldsymbol{\theta}$ the operator $\mathcal{P}^{(k)} (\boldsymbol{\theta})$ coincides with one of the projections $\mathcal{P}_j (\boldsymbol{\theta})$ introduced in Subsection~\ref{sndw_eigenval_multipl_section} (but the number $j$ may depend on $\boldsymbol{\theta}$).
\begin{condition}
	\label{sndw_cond1}
	\begin{enumerate*}[label=\emph{\arabic*$^{\circ}.$}, ref=\arabic*$^{\circ}$]
	\item $\widehat{N}_{0,Q}(\boldsymbol{\theta})=0$ for any $\boldsymbol{\theta} \in \mathbb{S}^{d-1}$.
	\item \label{sndw_cond1_it2} For any pair of indices $(k,r)$, $1 \le k,r \le n$, $k \ne r$, such that $\gamma_k (\boldsymbol{\theta}_0) = \gamma_r (\boldsymbol{\theta}_0) $ for some $\boldsymbol{\theta}_0 \in \mathbb{S}^{d-1}$, we have $(\mathcal{P}^{(k)} (\boldsymbol{\theta}))^* \widehat{N}_Q (\boldsymbol{\theta}) \mathcal{P}^{(r)} (\boldsymbol{\theta}) = 0$ for any $\boldsymbol{\theta} \in \mathbb{S}^{d-1}$.   
	\end{enumerate*}
\end{condition}

Condition~\ref{sndw_cond1}(\ref{sndw_cond1_it2}) can be reformulated: we assume that for the \textquotedblleft blocks\textquotedblright \ $(\mathcal{P}^{(k)} (\boldsymbol{\theta}))^* \widehat{N}_Q (\boldsymbol{\theta}) \mathcal{P}^{(r)} (\boldsymbol{\theta})$ of the operator $\widehat{N}_Q (\boldsymbol{\theta})$ that are not identically zero, the corresponding branches of the eigenvalues $\gamma_k (\boldsymbol{\theta})$ and $\gamma_r (\boldsymbol{\theta})$ do not intersect.

Obviously, Condition~\ref{sndw_cond1} is ensured by the following more restrictive condition.
\begin{condition}
	\label{sndw_cond2}
	\begin{enumerate*}[label=\emph{\arabic*$^{\circ}.$}, ref=\arabic*$^{\circ}$]
	\item $\widehat{N}_{0,Q}(\boldsymbol{\theta})=0$ for any $\boldsymbol{\theta} \in \mathbb{S}^{d-1}$.
	\item \label{sndw_cond2_it2} The number $p$ of different eigenvalues of the generalized spectral problem~\emph{(\ref{hatS_gener_spec_problem})} does not depend on $\boldsymbol{\theta} \in \mathbb{S}^{d-1}$.  
	\end{enumerate*}     
\end{condition}

Under Condition~\ref{sndw_cond2}, denote the different eigenvalues of the germ enumerated in the increasing order by $\gamma^{\circ}_1(\boldsymbol{\theta}), \ldots, \gamma^{\circ}_p(\boldsymbol{\theta})$. Then their multiplicities $k_1, \ldots, k_p$ do not depend on $\boldsymbol{\theta} \in \mathbb{S}^{d-1}$.  

\begin{remark}
	\begin{enumerate*}[label=\emph{\arabic*$^{\circ}.$}, ref=\arabic*$^{\circ}$]
	\item Assumption~\emph{\ref{sndw_cond2_it2}} of Condition~\emph{\ref{sndw_cond2}} is a fortiori satisfied, if the spectrum of  problem~\emph{(\ref{hatS_gener_spec_problem})} is simple for any $\boldsymbol{\theta} \in \mathbb{S}^{d-1}$.
	\item From Corollary~\emph{\ref{sndw_real_spec_simple_coroll}} it follows that Condition~\emph{\ref{sndw_cond2}} is satisfied if the matrices $b (\boldsymbol{\theta})$, $g (\mathbf{x})$, and $Q (\mathbf{x})$ have real entries, and the spectrum of  problem~\emph{(\ref{hatS_gener_spec_problem})} is simple for any $\boldsymbol{\theta} \in \mathbb{S}^{d-1}$.
	\end{enumerate*}
\end{remark}

So, suppose that Condition~\ref{sndw_cond1} is satisfied and put
\begin{equation*}
\mathcal{K} \coloneqq \{ (k,r) \colon 1 \le k,r \le n, \; k \ne r, \;  (\mathcal{P}^{(k)} (\boldsymbol{\theta}))^* \widehat{N}_Q (\boldsymbol{\theta}) \mathcal{P}^{(r)} (\boldsymbol{\theta}) \not\equiv 0 \}.
\end{equation*}
Denote $c^{\circ}_{kr} (\boldsymbol{\theta}) \coloneqq \min \{c_*, n^{-1} |\gamma_k (\boldsymbol{\theta}) - \gamma_r (\boldsymbol{\theta})| \}$, $(k,r) \in \mathcal{K}$.

Since $S (\boldsymbol{\theta})$ depends on $\boldsymbol{\theta} \in \mathbb{S}^{d-1}$ continuously, then the perturbation theory implies that $\gamma_j (\boldsymbol{\theta})$ are continuous on $\mathbb{S}^{d-1}$. By Condition~\ref{sndw_cond1}(\ref{sndw_cond1_it2}), for $(k,r) \in \mathcal{K}$ we have~$|\gamma_k (\boldsymbol{\theta}) - \gamma_r (\boldsymbol{\theta})| > 0$ for any $\boldsymbol{\theta} \in \mathbb{S}^{d-1}$, whence $c^{\circ}_{kr} \coloneqq \min_{\boldsymbol{\theta} \in \mathbb{S}^{d-1}} c^{\circ}_{kr} (\boldsymbol{\theta}) > 0$ for $(k,r) \in \mathcal{K}$. We put
\begin{equation}
\label{c^circ}
c^{\circ} \coloneqq \min_{(k,r) \in \mathcal{K}} c^{\circ}_{kr}.
\end{equation}
Clearly, the number~(\ref{c^circ}) is a realization of~(\ref{abstr_c^circ}) chosen independently of $\boldsymbol{\theta}$.
Under Condition~\ref{sndw_cond1}, the number $t^{00}$ subject to~(\ref{abstr_t00}) also can be chosen independently of $\boldsymbol{\theta} \in \mathbb{S}^{d-1}$. Taking~(\ref{delta_fixation}) and~(\ref{X_1_estimate}) into account, we put
\begin{equation*}
t^{00} = (8 \beta_2)^{-1} r_0 \alpha_1^{-3/2} \alpha_0^{1/2} \| g\|_{L_{\infty}}^{-3/2} \| g^{-1}\|_{L_{\infty}}^{-1/2} \|f\|_{L_\infty}^{-3} \|f^{-1}\|_{L_\infty}^{-1} c^{\circ}.
\end{equation*}
(The condition $t^{00} \le t^{0}$ is valid automatically, since $c^{\circ} \le \| S (\boldsymbol{\theta}) \| \le \alpha_1 \|g\|_{L_{\infty}}\|f\|_{L_\infty}^2$.)

Applying Theorem~\ref{abstr_sndw_exp_enchcd_thrm_2}, we deduce the following result.
\begin{thrm}
	\label{sndw_A(k)_exp_enchcd_thrm_2}
	Suppose that Condition~\emph{\ref{sndw_cond1}} \emph{(}or more restrictive Condition~\emph{\ref{sndw_cond2}}\emph{)} is satisfied. Then for $\tau \in \mathbb{R}$, $\varepsilon > 0$, and $\mathbf{k} \in \widetilde{\Omega}$ we have
	\begin{equation*}
	\| J (\mathbf{k}, \varepsilon; \tau) \mathcal{R}(\mathbf{k}, \varepsilon)\|_{L_2(\Omega) \to L_2 (\Omega)}  \le \mathcal{C}_3 (1+ |\tau|^{1/2}) \varepsilon,
	\end{equation*}
	where the constant $\mathcal{C}_3$ depends only on $\alpha_0$, $\alpha_1$, $\|g\|_{L_\infty}$, $\|g^{-1}\|_{L_\infty}$, $\|f\|_{L_{\infty}}$, $\|f^{-1}\|_{L_{\infty}}$, $r_0$, and also on $n$ and the number $c^{\circ}$.
\end{thrm}

\subsection{The sharpness of the results with respect to the smoothing operator}
Application of Theorems~\ref{abstr_sndw_exp_smooth_shrp_thrm_1}, \ref{abstr_sndw_exp_smooth_shrp_thrm_2} allows us to confirm that the results of Theorems~\ref{sndw_A(k)_exp_general_thrm}, \ref{sndw_A(k)_exp_enchcd_thrm_1}, and \ref{sndw_A(k)_exp_enchcd_thrm_2} are sharp with respect to the smoothing operator.

\begin{thrm}[\cite{Su2017}]
	\label{sndw_A(k)_exp_smooth_shrp_thrm_1}
	Suppose that $\widehat{N}_{0,Q} (\boldsymbol{\theta}_0) \ne 0$ for some $\boldsymbol{\theta}_0 \in \mathbb{S}^{d-1}$. Let $\tau \ne 0$ and $0 \le s < 3$. Then there does not exist a constant $\mathcal{C}(\tau) >0$ such that the estimate
	\begin{equation}
	\label{8.2a}
	\bigl\| \bigl( f e^{-i \tau \varepsilon^{-2} \mathcal{A} (\mathbf{k})} f^{-1} - f_0 e^{-i \tau \varepsilon^{-2} \mathcal{A}^0 (\mathbf{k})} f_0^{-1} \bigr) \mathcal{R} (\mathbf{k}, \varepsilon)^{s/2} \bigr\|_{L_2(\Omega) \to L_2 (\Omega)} \le \mathcal{C}(\tau) \varepsilon
	\end{equation}
	holds for almost~every $\mathbf{k} \in \widetilde{\Omega}$ and sufficiently small $\varepsilon > 0$.
\end{thrm}

\begin{thrm}
	\label{sndw_A(k)_exp_smooth_shrp_thrm_2}
	Suppose that $\widehat{N}_{0,Q} (\boldsymbol{\theta}) = 0$ for any $\boldsymbol{\theta} \in \mathbb{S}^{d-1}$ and $\widehat{\mathcal{N}}_Q^{(q)} (\boldsymbol{\theta}_0) \ne 0$   for some $\boldsymbol{\theta}_0 \in \mathbb{S}^{d-1}$ and some $q \in \{1,\ldots,p(\boldsymbol{\theta}_0)\}$.  Let $\tau \ne 0$ and $0 \le s < 2$. Then there does not exist a constant $\mathcal{C}(\tau) >0$ such that estimate
	\eqref{8.2a} holds for almost~every $\mathbf{k} \in \widetilde{\Omega}$ and sufficiently small $\varepsilon > 0$.
\end{thrm}
Theorem~\ref{sndw_A(k)_exp_smooth_shrp_thrm_1} was proved in~\cite[Theorem~11.7]{Su2017}. Theorem~\ref{sndw_A(k)_exp_smooth_shrp_thrm_2} is proved with the help of  Theorem~\ref{abstr_sndw_exp_smooth_shrp_thrm_2}  in a similar way as~\cite[Theorem~11.7]{Su2017}.

\subsection{The sharpness of the results with respect to dependence of the estimates on time}
Application of Theorem~\ref{abstr_sndw_exp_time_shrp_thrm_1} allows us to confirm that the result of Theorem~\ref{sndw_A(k)_exp_general_thrm} is sharp with respect to time.

\begin{thrm}
	\label{sndw_A(k)_exp_time_shrp_thrm_1}
	Suppose that $\widehat{N}_{0,Q} (\boldsymbol{\theta}_0) \ne 0$ for some $\boldsymbol{\theta}_0 \in \mathbb{S}^{d-1}$. Let $s \ge 3$. Then there does not exist a positive function $\mathcal{C}(\tau)$ such that $\lim_{\tau \to \infty} \mathcal{C}(\tau)/ |\tau| = 0$ and estimate
	\eqref{8.2a} holds for all $\tau \in \mathbb{R}$, almost~every $\mathbf{k}\in \widetilde{\Omega}$, and sufficiently small $\varepsilon > 0$.
\end{thrm}

\begin{proof} We prove by contradiction. Suppose that for some $s \ge 3$ there exists a function \hbox{$\mathcal{C}(\tau) > 0$} such that $\lim_{\tau \to \infty} \mathcal{C}(\tau)/ |\tau| = 0$ and estimate~(\ref{8.2a}) holds for almost~every $\mathbf{k} \in \widetilde{\Omega}$ and sufficiently small $\varepsilon > 0$. By~(\ref{R_P}), (\ref{R(k,eps)(I-P)_est}), it follows that there exists a function $\widetilde{\mathcal{C}}(\tau) > 0$ such that $\lim_{\tau \to \infty} \widetilde{\mathcal{C}}(\tau)/ |\tau| = 0$ and the estimate
	\begin{equation}
	\label{sndw_A(k)_exp_shrp_f1}
	\bigl\| \bigl( f e^{-i \tau \varepsilon^{-2} \mathcal{A} (\mathbf{k})} f^{-1} - f_0 e^{-i \tau \varepsilon^{-2} \mathcal{A}^0 (\mathbf{k})} f_0^{-1}\bigr) \widehat{P} \bigr\|_{L_2(\Omega) \to L_2(\Omega)} \varepsilon^s (|\mathbf{k}|^2 + \varepsilon^2)^{-s/2}  \le \widetilde{\mathcal{C}}(\tau) \varepsilon
	\end{equation}
	holds for almost~every $\mathbf{k} \in \widetilde{\Omega}$ and sufficiently small $\varepsilon > 0$.
	
	By~(\ref{abstr_P_Phat}), we have $f^{-1} \widehat{P} = P f^* \overline{Q}$, where $P$ is the orthogonal projection of $L_2 (\Omega; \mathbb{C}^n)$ onto the subspace~$\mathfrak{N}$ (see~(\ref{frakN})). Then the operator under the norm sign in~(\ref{sndw_A(k)_exp_shrp_f1}) can be written as $f e^{-i \tau \varepsilon^{-2} \mathcal{A} (\mathbf{k})} P f^* \overline{Q} - f_0 e^{-i \tau \varepsilon^{-2} \mathcal{A}^0 (\mathbf{k})} f_0^{-1} \widehat{P}$.
	
	Then, using the inequality 
	\begin{equation}
	\label{sndw_A(k)_exp_shrp_f2}
	\bigl\| F(\mathbf{k}) - P \bigr\|_{L_2 (\Omega) \to L_2 (\Omega)} \le C_1 |\mathbf{k}|, \qquad |\mathbf{k}| \le t^0,
	\end{equation}
	(see~(\ref{abstr_F(t)_threshold})), we conclude that there exists a function $\check{\mathcal{C}}(\tau) > 0$ such that $\lim_{\tau \to \infty} \check{\mathcal{C}}(\tau)/ |\tau| = 0$ and the estimate
	\begin{equation}
	\label{sndw_A(k)_exp_shrp_f3}
	\bigl\|  f e^{-i \tau \varepsilon^{-2} \mathcal{A} (\mathbf{k})} F(\mathbf{k}) f^* \overline{Q} - f_0 e^{-i \tau \varepsilon^{-2} \mathcal{A}^0 (\mathbf{k})} f_0^{-1} \widehat{P} \bigr\|_{L_2(\Omega) \to L_2 (\Omega) } \varepsilon^s (|\mathbf{k}|^2 + \varepsilon^2)^{-s/2}  \le \check{\mathcal{C}}(\tau) \varepsilon
	\end{equation}
	holds for almost~every $\mathbf{k} \in \widetilde{\Omega}$ in the ball $|\mathbf{k}| \le t^0$ and sufficiently small $\varepsilon > 0$.
	
	For fixed $\tau$ and $\varepsilon$, the operator under the norm sign in~(\ref{sndw_A(k)_exp_shrp_f3}) is continuous with respect to $\mathbf{k}$ in the ball $|\mathbf{k}| \le t^0$ (see~\cite[Lemma~11.8]{Su2017}). Hence, estimate~(\ref{sndw_A(k)_exp_shrp_f3}) holds for all $\mathbf{k}$ in this ball, in particular, for $\mathbf{k} = t\boldsymbol{\theta}_0$ if $t \le t^0$. Applying inequality~(\ref{sndw_A(k)_exp_shrp_f2}) and the identity 
	$P f^* \overline{Q} = f^{-1} \widehat{P}$ once again, we obtain the estimate
	\begin{equation}
	\label{sndw_A(k)_exp_shrp_f4}
	\bigl\| \bigl( f e^{-i \tau \varepsilon^{-2} \mathcal{A} (t\boldsymbol{\theta}_0)} f^{-1} - f_0 e^{-i \tau \varepsilon^{-2} \mathcal{A}^0 (t\boldsymbol{\theta}_0)} f_0^{-1}\bigr) \widehat{P} \bigr\|_{L_2(\Omega) \to L_2(\Omega)} \varepsilon^s (t^2 + \varepsilon^2)^{-s/2}  \le \check{\mathcal{C}}'(\tau) \varepsilon
	\end{equation}
	for all $t \le t^{0}$ and sufficiently small $\varepsilon > 0$, where $\check{\mathcal{C}}'(\tau) > 0$ and $\lim_{\tau \to \infty} \check{\mathcal{C}}'(\tau)/ |\tau| = 0$.
	
	In the abstract terms, estimate~(\ref{sndw_A(k)_exp_shrp_f4}) corresponds to~(\ref{abstr_sndw_exp_shrp_est}). Since it is assumed that $\widehat{N}_{0,Q}(\boldsymbol{\theta}_0) \ne 0$, applying Theorem~\ref{abstr_sndw_exp_time_shrp_thrm_1}, we arrive at a contradiction.
\end{proof}

Similarly, application of Theorem~\ref{abstr_sndw_exp_time_shrp_thrm_2} allows us to confirm the sharpness of Theorems~\ref{sndw_A(k)_exp_enchcd_thrm_1}, \ref{sndw_A(k)_exp_enchcd_thrm_2}.

\begin{thrm}
	\label{sndw_A(k)_exp_time_shrp_thrm_2}
	Suppose that $\widehat{N}_{0,Q} (\boldsymbol{\theta}) = 0$ for any $\boldsymbol{\theta} \in \mathbb{S}^{d-1}$ and $\widehat{\mathcal{N}}_Q^{\,(q)} (\boldsymbol{\theta}_0) \ne 0$ for some $\boldsymbol{\theta}_0 \in \mathbb{S}^{d-1}$ and some 
	$q \in \{1,\ldots,p(\boldsymbol{\theta}_0)\}$. Let $s \ge 2$. Then there does not exist a positive function $\mathcal{C}(\tau)$ such that $\lim_{\tau \to \infty} \mathcal{C}(\tau)/ |\tau|^{1/2} = 0$ and estimate \eqref{8.2a} holds for all $\tau \in \mathbb{R}$, almost~every $\mathbf{k} \in \widetilde{\Omega}$, and sufficiently small $\varepsilon > 0$.
\end{thrm}

\section{Approximation of the operator exponential $e^{-i \tau \varepsilon^{-2} {\mathcal{A}}}$}

\subsection{Approximation of the operator $e^{-i\tau \varepsilon^{-2} \widehat{\mathcal{A}}}$}
In $L_2 (\mathbb{R}^d; \mathbb{C}^n)$, we consider the operator~$\widehat{\mathcal{A}}$ given by (\ref{hatA}). Let $\widehat{\mathcal{A}}^0$ be the effective operator~(\ref{hatA0}). Denote $\widehat{J}(\varepsilon; \tau) \coloneqq e^{-i\tau \varepsilon^{-2} \widehat{\mathcal{A}}} - e^{-i\tau \varepsilon^{-2} \widehat{\mathcal{A}}^0}$. Recall the notation $\mathcal{H}_0 = - \Delta$ and put
\begin{equation}
\label{R(epsilon)}
\mathcal{R} (\varepsilon) \coloneqq \varepsilon^2 (\mathcal{H}_0 + \varepsilon^2 I)^{-1}.
\end{equation}
The operator $\mathcal{R} (\varepsilon)$ expands in the direct integral of the operators~(\ref{R(k,eps)}):
\begin{equation*}
\mathcal{R} (\varepsilon) = \mathscr{U}^{-1} \left( \int_{\widetilde{\Omega}} \oplus  \mathcal{R} (\mathbf{k}, \varepsilon) \, d \mathbf{k}  \right) \mathscr{U}.
\end{equation*}

Recall also notation~(\ref{Jhat(k,eps)}). By decomposition~(\ref{Gelfand_A_decompose}) for $\widehat{\mathcal{A}}$ and $\widehat{\mathcal{A}}^0$, we have
\begin{equation}
\label{hatA_exps_Gelfand}
\| \widehat{J}(\varepsilon; \tau) \mathcal{R}(\varepsilon)^{s/2} \|_{L_2(\mathbb{R}^d) \to L_2(\mathbb{R}^d)} = \underset{\mathbf{k} \in \widetilde{\Omega}}{\esssup} \| \widehat{J}(\mathbf{k}, \varepsilon; \tau) \mathcal{R}(\mathbf{k}, \varepsilon)^{s/2} \|_{L_2(\Omega) \to L_2(\Omega)}.
\end{equation}
Therefore, Theorems~\ref{hatA(k)_exp_general_thrm}, \ref{hatA(k)_exp_enchcd_thrm_1}, \ref{hatA(k)_exp_enchcd_thrm_2} imply the following results.

\begin{thrm}[\cite{BSu2008}]
	\label{hatA_exp_general_thrm}
	For $\tau \in \mathbb{R}$ and $\varepsilon > 0$ we have
	\begin{equation*}
	\bigl\| \widehat{J}(\varepsilon; \tau) \mathcal{R}(\varepsilon)^{3/2} \bigr\|_{L_2(\mathbb{R}^d) \to L_2(\mathbb{R}^d)} \le \widehat{\mathcal{C}}_1(1 + |\tau|) \varepsilon. 
	\end{equation*}
	The constant $\widehat{\mathcal{C}}_1$ depends only on $\alpha_0$, $\alpha_1$, $\|g\|_{L_\infty}$, $\|g^{-1}\|_{L_\infty}$, and $r_0$.
\end{thrm}

\begin{thrm}
	\label{hatA_exp_enchcd_thrm_1}
	Let $\widehat{N}(\boldsymbol{\theta})$ be the operator defined by~\emph{(\ref{hatN(theta)})}. Suppose that $\widehat{N}(\boldsymbol{\theta})=0$ for any $\boldsymbol{\theta} \in \mathbb{S}^{d-1}$. Then for $\tau \in \mathbb{R}$ and $\varepsilon > 0$ we have
	\begin{equation*}
	\bigl\| \widehat{J}(\varepsilon; \tau) \mathcal{R}(\varepsilon) \bigr\|_{L_2(\mathbb{R}^d) \to L_2(\mathbb{R}^d)} \le \widehat{\mathcal{C}}_2(1 + |\tau|^{1/2}) \varepsilon.
	\end{equation*}
	The constant $\widehat{\mathcal{C}}_2$ depends only on $\alpha_0$, $\alpha_1$, $\|g\|_{L_\infty}$, $\|g^{-1}\|_{L_\infty}$, and $r_0$.
\end{thrm}

\begin{thrm}
	\label{hatA_exp_enchcd_thrm_2}
	Suppose that Condition~\emph{\ref{cond1}} \emph{(}or more restrictive Condition~\emph{\ref{cond2}}\emph{)} is satisfied. Then for $\tau \in \mathbb{R}$ and $\varepsilon > 0$ we have
	\begin{equation*}
	\bigl\| \widehat{J}(\varepsilon; \tau) \mathcal{R}(\varepsilon) \bigr\|_{L_2(\mathbb{R}^d) \to L_2(\mathbb{R}^d)} \le \widehat{\mathcal{C}}_3 (1 + |\tau|^{1/2}) \varepsilon.
	\end{equation*}
	The constant $\widehat{\mathcal{C}}_3$ depends only on $\alpha_0$, $\alpha_1$, $\|g\|_{L_\infty}$, $\|g^{-1}\|_{L_\infty}$, $r_0$, and also on $n$ and the number $\widehat{c}^{\circ}$.
\end{thrm}

Theorem~\ref{hatA_exp_general_thrm} was proved in~\cite[Theorem~9.1]{BSu2008}. Theorems~\ref{hatA_exp_enchcd_thrm_1} and~\ref{hatA_exp_enchcd_thrm_2} improve the results of Theorems~12.2 and~12.3 from~\cite{Su2017} with respect to dependence of the estimates on $\tau$.

Application of Theorems~\ref{hatA(k)_exp_smooth_shrp_thrm_1}, \ref{hatA(k)_exp_smooth_shrp_thrm_2} allows us to confirm that the results of Theorems~\ref{hatA_exp_general_thrm}, \ref{hatA_exp_enchcd_thrm_1}, and \ref{hatA_exp_enchcd_thrm_2} are sharp with respect to the smoothing operator.

\begin{thrm}[\cite{Su2017}]
	\label{hatA_exp_smooth_shrp_thrm_1}
	Suppose that $\widehat{N}_0 (\boldsymbol{\theta}_0) \ne 0$ for some $\boldsymbol{\theta}_0 \in \mathbb{S}^{d-1}$. Let $\tau \ne 0$ and $0 \le s < 3$. Then there does not exist a constant $\mathcal{C}(\tau) >0$ such that the estimate
	\begin{equation}
	\label{9.2a}
	\bigl\| \widehat{J}(\varepsilon; \tau) \mathcal{R}(\varepsilon)^{s/2} \bigr\|_{L_2(\mathbb{R}^d) \to L_2(\mathbb{R}^d)} \le \mathcal{C}(\tau) \varepsilon
	\end{equation}
	holds for all sufficiently small $\varepsilon > 0$.
\end{thrm}

\begin{thrm}
	\label{hatA_exp_smooth_shrp_thrm_2}
	Suppose that $\widehat{N}_0 (\boldsymbol{\theta}) = 0$ for any $\boldsymbol{\theta} \in \mathbb{S}^{d-1}$ and   $\widehat{\mathcal{N}}^{(q)} (\boldsymbol{\theta}_0) \ne 0$ for some $\boldsymbol{\theta}_0 \in \mathbb{S}^{d-1}$ and some $q \in \{1,\ldots,p(\boldsymbol{\theta}_0)\}$. 
	Let $\tau \ne 0$ and $0 \le s < 2$. Then there does not exist a constant $\mathcal{C}(\tau) >0$ such that estimate \eqref{9.2a}
	holds for all sufficiently small $\varepsilon > 0$.
\end{thrm}

Theorem~\ref{hatA_exp_smooth_shrp_thrm_1} was proved in~\cite[Theorem~12.4]{Su2017}.

Next, application of Theorems~\ref{hatA(k)_exp_time_shrp_thrm_1}, \ref{hatA(k)_exp_time_shrp_thrm_2} allows us to confirm that the results of Theorems~\ref{hatA_exp_general_thrm}, \ref{hatA_exp_enchcd_thrm_1}, and \ref{hatA_exp_enchcd_thrm_2} are sharp with respect to dependence of the estimates on time.

\begin{thrm}
	\label{hatA_exp_shrp_thrm_1}
	Suppose that $\widehat{N}_0 (\boldsymbol{\theta}_0) \ne 0$ for some $\boldsymbol{\theta}_0 \in \mathbb{S}^{d-1}$. Let $s \ge 3$. Then there does not exist a positive function $\mathcal{C}(\tau)$ such that $\lim_{\tau \to \infty} \mathcal{C}(\tau)/ |\tau| = 0$ and estimate \eqref{9.2a}
	holds for all $\tau \in \mathbb{R}$ and all sufficiently small $\varepsilon > 0$.
\end{thrm}

\begin{thrm}
	\label{hatA_exp_shrp_thrm_2}
	Suppose that $\widehat{N}_0 (\boldsymbol{\theta}) = 0$ for any $\boldsymbol{\theta} \in \mathbb{S}^{d-1}$ and  $\widehat{\mathcal{N}}^{(q)} (\boldsymbol{\theta}_0) \ne 0$ for some $\boldsymbol{\theta}_0 \in \mathbb{S}^{d-1}$ and some $q \in \{1,\ldots,p(\boldsymbol{\theta}_0)\}$. Let $s \ge 2$. Then there does not exist a positive function $\mathcal{C}(\tau)$ such that $\lim_{\tau \to \infty} \mathcal{C}(\tau)/ |\tau|^{1/2} = 0$ and estimate \eqref{9.2a}
	holds for all $\tau \in \mathbb{R}$ and all sufficiently small $\varepsilon > 0$.
\end{thrm}

\subsection{Approximation of the sandwiched operator $e^{-i \tau \varepsilon^{-2} \mathcal{A}}$}

In $L_2(\mathbb{R}^d; \mathbb{C}^n)$, we consider the operator~$\mathcal{A}$ given by (\ref{A}). Let $f_0$ be the matrix~(\ref{f0}), and let $\mathcal{A}^0$ be the operator~(\ref{A0}). Denote
\begin{equation*}
J(\varepsilon; \tau) \coloneqq f e^{-i \tau \varepsilon^{-2} \mathcal{A}} f^{-1} - f_0 e^{-i \tau \varepsilon^{-2} \mathcal{A}^0 } f_0^{-1}.
\end{equation*}
Similarly to~(\ref{hatA_exps_Gelfand}), we have
\begin{equation*}
\| J(\varepsilon; \tau) \mathcal{R}(\varepsilon)^{s/2} \|_{L_2(\mathbb{R}^d) \to L_2(\mathbb{R}^d)} = \underset{\mathbf{k} \in \widetilde{\Omega}}{\esssup} \| J(\mathbf{k}, \varepsilon; \tau) \mathcal{R}(\mathbf{k}, \varepsilon)^{s/2} \|_{L_2(\Omega) \to L_2(\Omega)}.
\end{equation*}
Here $J(\mathbf{k}, \varepsilon; \tau)$ is defined by~(\ref{J(k,eps)}). Therefore, we deduce the following results from Theorems~\ref{sndw_A(k)_exp_general_thrm}, \ref{sndw_A(k)_exp_enchcd_thrm_1}, and \ref{sndw_A(k)_exp_enchcd_thrm_2}.

\begin{thrm}[\cite{BSu2008}]
	\label{sndw_A_exp_general_thrm}
	For $\tau \in \mathbb{R}$ and $\varepsilon > 0$ we have
	\begin{equation*}
	\bigl\| J(\varepsilon; \tau) \mathcal{R}(\varepsilon)^{3/2} \bigr\|_{L_2(\mathbb{R}^d) \to L_2(\mathbb{R}^d)} \le \mathcal{C}_1(1 + |\tau|) \varepsilon. 
	\end{equation*}
	The constant $\mathcal{C}_1$ depends only on $\alpha_0$, $\alpha_1$, $\|g\|_{L_\infty}$, $\|g^{-1}\|_{L_\infty}$, $\|f\|_{L_{\infty}}$, $\|f^{-1}\|_{L_{\infty}}$, and $r_0$.
\end{thrm}

\begin{thrm}
	\label{sndw_A_exp_enchcd_thrm_1}
	Let $\widehat{N}_{Q}(\boldsymbol{\theta})$ be the operator defined by~\emph{(\ref{N_Q(theta)})}. Suppose that $\widehat{N}_{Q}(\boldsymbol{\theta})=0$ for any $\boldsymbol{\theta} \in \mathbb{S}^{d-1}$. Then for $\tau \in \mathbb{R}$ and $\varepsilon > 0$ we have
	\begin{equation*}
	\bigl\| J(\varepsilon; \tau) \mathcal{R}(\varepsilon) \bigr\|_{L_2(\mathbb{R}^d) \to L_2(\mathbb{R}^d)} \le \mathcal{C}_2(1 + |\tau|^{1/2}) \varepsilon.
	\end{equation*}
	The constant $\mathcal{C}_2$ depends only on $\alpha_0$, $\alpha_1$, $\|g\|_{L_\infty}$, $\|g^{-1}\|_{L_\infty}$, $\|f\|_{L_{\infty}}$, $\|f^{-1}\|_{L_{\infty}}$, and $r_0$.
\end{thrm}

\begin{thrm}
	\label{sndw_A_exp_enchcd_thrm_2}
	Suppose that Condition~\emph{\ref{sndw_cond1}} \emph{(}or more restrictive Condition~\emph{\ref{sndw_cond2}}\emph{)} is satisfied. Then for $\tau \in \mathbb{R}$ and $\varepsilon > 0$ we have
	\begin{equation*}
	\bigl\| J(\varepsilon; \tau) \mathcal{R}(\varepsilon) \bigr\|_{L_2(\mathbb{R}^d) \to L_2(\mathbb{R}^d)} \le \mathcal{C}_3(1 + |\tau|^{1/2}) \varepsilon.
	\end{equation*}
	The constant $\mathcal{C}_3$ depends only on $\alpha_0$, $\alpha_1$, $\|g\|_{L_\infty}$, $\|g^{-1}\|_{L_\infty}$, $\|f\|_{L_{\infty}}$, $\|f^{-1}\|_{L_{\infty}}$, $r_0$, and also on $n$ and the number $c^{\circ}$.
\end{thrm}

Theorem~\ref{sndw_A_exp_general_thrm}  was proved in~\cite[Theorem~10.1]{BSu2008}. Theorems~\ref{sndw_A_exp_enchcd_thrm_1} and~\ref{sndw_A_exp_enchcd_thrm_2} improve the results of Theorems~12.6 and~12.7 from~\cite{Su2017} with respect to dependence of the estimates on $\tau$.

Application of Theorems~\ref{sndw_A(k)_exp_smooth_shrp_thrm_1}, \ref{sndw_A(k)_exp_smooth_shrp_thrm_2} allows us to confirm that the results of Theorems~\ref{sndw_A_exp_general_thrm}, \ref{sndw_A_exp_enchcd_thrm_1}, and \ref{sndw_A_exp_enchcd_thrm_2} are sharp with respect to the smoothing operator.

\begin{thrm}[\cite{Su2017}]
	\label{sndw_A_exp_smooth_shrp_thrm_1}
	Suppose that $\widehat{N}_{0,Q} (\boldsymbol{\theta}_0) \ne 0$ for some $\boldsymbol{\theta}_0 \in \mathbb{S}^{d-1}$. Let $\tau \ne 0$ and $0 \le s < 3$. Then there does not exist a constant $\mathcal{C}(\tau) >0$ such that the estimate
	\begin{equation}
	\label{9.2b}
	\bigl\| J(\varepsilon; \tau) \mathcal{R}(\varepsilon)^{s/2} \bigr\|_{L_2(\mathbb{R}^d) \to L_2(\mathbb{R}^d)} \le \mathcal{C}(\tau) \varepsilon
	\end{equation}
	holds for all sufficiently small $\varepsilon > 0$.
\end{thrm}

\begin{thrm}
	\label{sndw_A_exp_smooth_shrp_thrm_2}
	Suppose that $\widehat{N}_{0,Q} (\boldsymbol{\theta}) = 0$ for any $\boldsymbol{\theta} \in \mathbb{S}^{d-1}$ and  $\widehat{\mathcal{N}}_Q^{(q)} (\boldsymbol{\theta}_0) \ne 0$  for some $\boldsymbol{\theta}_0 \in \mathbb{S}^{d-1}$ and some $q \in \{1,\ldots,p(\boldsymbol{\theta}_0)\}$. 
	Let $\tau \ne 0$ and $0 \le s < 2$. Then there does not exist a constant $\mathcal{C}(\tau) >0$ such that estimate \eqref{9.2b}
	holds for all sufficiently small $\varepsilon > 0$.
\end{thrm}

Theorem~\ref{sndw_A_exp_smooth_shrp_thrm_1} was proved in~\cite[Theorem~12.8]{Su2017}.

Next, application of Theorems~\ref{sndw_A(k)_exp_time_shrp_thrm_1}, \ref{sndw_A(k)_exp_time_shrp_thrm_2} allows us to confirm that the results of 
Theorems~\ref{sndw_A_exp_general_thrm}, \ref{sndw_A_exp_enchcd_thrm_1}, and \ref{sndw_A_exp_enchcd_thrm_2} are sharp with respect to the dependence of the estimates on time.

\begin{thrm}
	\label{sndw_A_exp_time_shrp_thrm_1}
	Suppose that $\widehat{N}_{0,Q} (\boldsymbol{\theta}_0) \ne 0$ for some $\boldsymbol{\theta}_0 \in \mathbb{S}^{d-1}$. Let $s \ge 3$. Then there does not exist a positive function $\mathcal{C}(\tau)$ such that $\lim_{\tau \to \infty} \mathcal{C}(\tau)/ |\tau| = 0$ and estimate \eqref{9.2b} 
	holds for all $\tau \in \mathbb{R}$ and all sufficiently small $\varepsilon > 0$.
\end{thrm}

\begin{thrm}
	\label{sndw_A_exp_time_shrp_thrm_2}
	Suppose that $\widehat{N}_{0,Q} (\boldsymbol{\theta}) = 0$ for any $\boldsymbol{\theta} \in \mathbb{S}^{d-1}$ and $\widehat{\mathcal{N}}_Q^{(q)} (\boldsymbol{\theta}_0) \ne 0$  for some $\boldsymbol{\theta}_0 \in \mathbb{S}^{d-1}$ and some $q \in \{1,\ldots,p(\boldsymbol{\theta}_0)\}$. Let $s \ge 2$. Then there does not exist a positive function $\mathcal{C}(\tau)$ such that $\lim_{\tau \to \infty} \mathcal{C}(\tau)/ |\tau|^{1/2} = 0$ and estimate \eqref{9.2b}
	holds for all $\tau \in \mathbb{R}$ and all sufficiently small $\varepsilon > 0$.
\end{thrm}

\part{Homogenization problems for nonstationary \\Schr\"{o}dinger-type equations}
\label{main_results_part}

\section{Approximation of the operator $e^{-i\tau\mathcal{A}_\varepsilon}$}
\label{main_results_exp_section}

\subsection{The operators $\widehat{\mathcal{A}}_\varepsilon$ and $\mathcal{A}_\varepsilon$. The scaling transformation}
If $\psi(\mathbf{x})$ is a $\Gamma$-periodic measurable function in $\mathbb{R}^d$, we agree to use the notation $\psi^{\varepsilon}(\mathbf{x}) \coloneqq \psi(\varepsilon^{-1} \mathbf{x}), \; \varepsilon > 0$. \emph{Our main objects} are the operators $\widehat{\mathcal{A}}_\varepsilon$ and $\mathcal{A}_\varepsilon$ acting in $L_2 (\mathbb{R}^d; \mathbb{C}^n)$ and formally given by
\begin{align}
\label{Ahat_eps}
\widehat{\mathcal{A}}_\varepsilon &\coloneqq b(\mathbf{D})^* g^{\varepsilon}(\mathbf{x}) b(\mathbf{D}), \\
\label{A_eps}
\mathcal{A}_\varepsilon &\coloneqq (f^{\varepsilon}(\mathbf{x}))^* b(\mathbf{D})^* g^{\varepsilon}(\mathbf{x}) b(\mathbf{D}) f^{\varepsilon}(\mathbf{x}).
\end{align}
The precise definitions are given in terms of the corresponding quadratic forms (cf.~Subsection~\ref{A_section}).

Let $T_{\varepsilon}$ be \emph{the unitary scaling transformation} in $L_2 (\mathbb{R}^d; \mathbb{C}^n)$ defined by 
$(T_{\varepsilon} \mathbf{u})(\mathbf{x}) = \varepsilon^{d/2} \mathbf{u} (\varepsilon \mathbf{x})$, $\varepsilon > 0$.
Then $\mathcal{A}_\varepsilon = \varepsilon^{-2}T_{\varepsilon}^* \mathcal{A} T_{\varepsilon}$. Hence,
\begin{equation}
\label{exp_scale_transform}
e^{-i\tau\mathcal{A}_{\varepsilon}} = T_{\varepsilon}^* e^{-i\tau\varepsilon^{-2} \mathcal{A}} T_{\varepsilon}.
\end{equation}
The operator $\widehat{\mathcal{A}}_{\varepsilon}$ satisfies similar relations. Applying the scaling transformation to the resolvent of the operator $\mathcal{H}_0 = - \Delta$ and using the notation~(\ref{R(epsilon)}), we obtain
\begin{equation}
\label{H0_resolv_scale_transform}
(\mathcal{H}_0 + I)^{-1} = \varepsilon^2 T_\varepsilon^* (\mathcal{H}_0 + \varepsilon^2 I)^{-1} T_\varepsilon = T_\varepsilon^* \mathcal{R} (\varepsilon) T_\varepsilon.
\end{equation}
Finally, if $\psi(\mathbf{x})$ is a $\Gamma$-periodic function, then $[\psi^{\varepsilon}] = T_\varepsilon^* [\psi] T_\varepsilon$.

\subsection{Approximation of the operator $e^{-i\tau\widehat{\mathcal{A}}_\varepsilon}$}
We start with the simpler operator~(\ref{Ahat_eps}). Let $\widehat{\mathcal{A}}^0$ be the effective operator~(\ref{hatA0}). Using relations of the form~(\ref{exp_scale_transform}) (for the operators $\widehat{\mathcal{A}}_\varepsilon$ and $\widehat{\mathcal{A}}^0$) and~(\ref{H0_resolv_scale_transform}), we obtain
\begin{equation}
\label{hatB_eps_exps_scale_transform}
(e^{-i\tau\widehat{\mathcal{A}}_{\varepsilon}} - e^{-i\tau\widehat{\mathcal{A}}^0}) (\mathcal{H}_0 + I)^{-s/2} =  T_{\varepsilon}^* \widehat{J}(\varepsilon; \tau) \mathcal{R} (\varepsilon)^{s/2} T_{\varepsilon}, \qquad \varepsilon > 0.
\end{equation}

Using Theorem~\ref{hatA_exp_general_thrm} and~(\ref{hatB_eps_exps_scale_transform}), we obtain the following result proved before in~\cite[Theorem~12.2]{BSu2008}

\begin{thrm}[\cite{BSu2008}]
	\label{hatA_eps_exp_general_thrm}  
	Let $\widehat{\mathcal{A}}_{\varepsilon}$ be the operator~\emph{(\ref{Ahat_eps})} and let  $\widehat{\mathcal{A}}^0$ be the effective operator~\emph{(\ref{hatA0})}. Then for $0 \le s \le 3$, and $\tau \in \mathbb{R}$, $\varepsilon > 0$ we have
	\begin{equation*}
	\| e^{-i\tau\widehat{\mathcal{A}}_{\varepsilon}} - e^{-i\tau\widehat{\mathcal{A}}^0} \|_{H^s (\mathbb{R}^d) \to L_2 (\mathbb{R}^d)} \le  \widehat{\mathfrak{C}}_1 (s) (1 +  |\tau|)^{s/3} \varepsilon^{s/3},  
	\end{equation*}
	where $\widehat{\mathfrak{C}}_1 (s) = 2^{1-s/3} \widehat{\mathcal{C}}_1^{s/3} $. The constant $\widehat{\mathcal{C}}_1$ depends only on $\alpha_0$, $\alpha_1$, $\|g\|_{L_\infty}$, $\|g^{-1}\|_{L_\infty}$, and $r_0$. 
\end{thrm}

This result can be improved under some additional assumptions. From Theorem~\ref{hatA_exp_enchcd_thrm_1} we deduce the following result.

\begin{thrm}
	\label{hatA_eps_exp_enchcd_thrm_1}
	Suppose that the assumptions of Theorem~\emph{\ref{hatA_eps_exp_general_thrm}} are satisfied. Let $\widehat{N}(\boldsymbol{\theta})$ be the operator defined by~\emph{(\ref{hatN(theta)})}. Suppose that $\widehat{N}(\boldsymbol{\theta})=0$ for any $\boldsymbol{\theta} \in \mathbb{S}^{d-1}$. Then for $0 \le s \le 2$, and $\tau \in \mathbb{R}$, $\varepsilon > 0$ we have
	\begin{equation}
	\label{hatA_eps_exp_enchcd_est_1}
	\| e^{-i\tau\widehat{\mathcal{A}}_{\varepsilon}} - e^{-i\tau\widehat{\mathcal{A}}^0} \|_{H^s (\mathbb{R}^d) \to L_2 (\mathbb{R}^d)} \le  \widehat{\mathfrak{C}}_2 (s) (1 +  |\tau|^{1/2})^{s/2} \varepsilon^{s/2},  
	\end{equation}
	where $\widehat{\mathfrak{C}}_2 (s) = 2^{1-s/2} \widehat{\mathcal{C}}_2^{s/2} $. The constant $\widehat{\mathcal{C}}_2$ depends only on $\alpha_0$, $\alpha_1$, $\|g\|_{L_\infty}$, $\|g^{-1}\|_{L_\infty}$, and $r_0$.
\end{thrm}

\begin{proof}
	Since $T_\varepsilon$ is unitary, it follows from Theorem~\ref{hatA_exp_enchcd_thrm_1} and~(\ref{hatB_eps_exps_scale_transform}) that
	\begin{equation}
	\label{hatA_eps_exp_enchcd_est_1_L2L2}
	\| ( e^{-i\tau\widehat{\mathcal{A}}_{\varepsilon}} - e^{-i\tau\widehat{\mathcal{A}}^0}) (\mathcal{H}_0 + I)^{-1} \|_{L_2(\mathbb{R}^d) \to L_2(\mathbb{R}^d)} \le \widehat{\mathcal{C}}_2 (1+ |\tau|^{1/2}) \varepsilon.
	\end{equation}
	Obviously,
	\begin{equation}
	\label{hatA_exp_trivial_est}
	\| e^{-i\tau\widehat{\mathcal{A}}_{\varepsilon}} - e^{-i\tau\widehat{\mathcal{A}}^0} \|_{L_2 (\mathbb{R}^d) \to L_2 (\mathbb{R}^d)} \le 2.
	\end{equation}
	Interpolating between~(\ref{hatA_eps_exp_enchcd_est_1_L2L2}) and~(\ref{hatA_exp_trivial_est}), for $ 0 \le s \le 2$ we obtain
	\begin{equation}
	\label{hatA_exp_enchcd_est_1_interp}
	\| (e^{-i\tau\widehat{\mathcal{A}}_{\varepsilon}} - e^{-i\tau\widehat{\mathcal{A}}^0}) (\mathcal{H}_0 + I)^{-s/2} \|_{L_2 (\mathbb{R}^d) \to L_2 (\mathbb{R}^d)}
	\le 2^{1-s/2} \widehat{\mathcal{C}}_2^{s/2} (1 +  |\tau|^{1/2})^{s/2} \varepsilon^{s/2}. 
	\end{equation}
	The operator $(\mathcal{H}_0 + I)^{s/2}$ is an isometric isomorphism of $H^s(\mathbb{R}^d; \mathbb{C}^n)$ onto 
	$L_2(\mathbb{R}^d; \mathbb{C}^n)$. Therefore, (\ref{hatA_exp_enchcd_est_1_interp}) is equivalent to~(\ref{hatA_eps_exp_enchcd_est_1}).
\end{proof}

Similarly, Theorem~\ref{hatA_exp_enchcd_thrm_2} implies the following result.

\begin{thrm}
	\label{hatA_eps_exp_enchcd_thrm_2}
	Suppose that the assumptions of Theorem~\emph{\ref{hatA_eps_exp_general_thrm}} are satisfied. Suppose that Condition~\emph{\ref{cond1}} \emph{(}or more restrictive Condition~\emph{\ref{cond2})} is satisfied. Then for $0 \le s \le 2$, and $\tau \in \mathbb{R}$, $\varepsilon > 0$ we have
	\begin{equation*}
	\| e^{-i\tau\widehat{\mathcal{A}}_{\varepsilon}} - e^{-i\tau\widehat{\mathcal{A}}^0} \|_{H^s (\mathbb{R}^d) \to L_2 (\mathbb{R}^d)} \le  \widehat{\mathfrak{C}}_3 (s) (1 +  |\tau|^{1/2})^{s/2} \varepsilon^{s/2},  
	\end{equation*}
	where $\widehat{\mathfrak{C}}_3 (s) = 2^{1-s/2} \widehat{\mathcal{C}}_3^{s/2} $. The constant $\widehat{\mathcal{C}}_3$ depends on $\alpha_0$, $\alpha_1$, $\|g\|_{L_\infty}$, $\|g^{-1}\|_{L_\infty}$, $r_0$, and also on $n$ and the number $\widehat{c}^{\circ}$.
\end{thrm}

Theorems~\ref{hatA_eps_exp_enchcd_thrm_1} and~\ref{hatA_eps_exp_enchcd_thrm_2} improve the results of Theorems~13.2 and~13.4 from~\cite{Su2017} with respect to dependence of the estimates on $\tau$.

Application of Theorems~\ref{hatA_exp_smooth_shrp_thrm_1} and \ref{hatA_exp_smooth_shrp_thrm_2} allows us to confirm that the results of Theorems~\ref{hatA_eps_exp_general_thrm}, \ref{hatA_eps_exp_enchcd_thrm_1}, and \ref{hatA_eps_exp_enchcd_thrm_2} are sharp with respect to the type of  the operator norm.

\begin{thrm}[\cite{Su2017}]
	\label{hatA_eps_exp_smooth_shrp_thrm_1}
	Suppose that $\widehat{N}_0 (\boldsymbol{\theta}_0) \ne 0$ for some $\boldsymbol{\theta}_0 \in \mathbb{S}^{d-1}$. Let $\tau \ne 0$ and $0 \le s < 3$. Then there does not exist a constant $\mathcal{C}(\tau) >0$ such that the estimate
	\begin{equation}
	\label{10.9a}
	\| e^{-i\tau\widehat{\mathcal{A}}_{\varepsilon}} - e^{-i\tau\widehat{\mathcal{A}}^0} \|_{H^s (\mathbb{R}^d) \to L_2 (\mathbb{R}^d)} \le \mathcal{C}(\tau) \varepsilon
	\end{equation}
	holds for all sufficiently small $\varepsilon > 0$.
\end{thrm}

\begin{thrm}
	Suppose that $\widehat{N}_0 (\boldsymbol{\theta}) = 0$ for any $\boldsymbol{\theta} \in \mathbb{S}^{d-1}$ and  $\widehat{\mathcal{N}}^{(q)} (\boldsymbol{\theta}_0) \ne 0$  for some $\boldsymbol{\theta}_0 \in \mathbb{S}^{d-1}$ and some $q \in \{1,\ldots,p(\boldsymbol{\theta}_0)\}$. 
	Let $\tau \ne 0$ and $0 \le s < 2$. Then there does not exist a constant $\mathcal{C}(\tau) >0$ such that estimate \eqref{10.9a}
	holds for all sufficiently small $\varepsilon > 0$.
\end{thrm}

Theorem~\ref{hatA_eps_exp_smooth_shrp_thrm_1} was proved in~\cite[Theorem~13.6]{Su2017}.

Finally, application of Theorems~\ref{hatA_exp_shrp_thrm_1} and \ref{hatA_exp_shrp_thrm_2} allows us to confirm that the results of Theorems~\ref{hatA_eps_exp_general_thrm}, \ref{hatA_eps_exp_enchcd_thrm_1}, and \ref{hatA_eps_exp_enchcd_thrm_2} are sharp with respect to dependence of the estimates on time.

\begin{thrm}
	Suppose that $\widehat{N}_0 (\boldsymbol{\theta}_0) \ne 0$ for some $\boldsymbol{\theta}_0 \in \mathbb{S}^{d-1}$. Let $s \ge 3$. Then there does not exist a positive function $\mathcal{C}(\tau)$ such that $\lim_{\tau \to \infty} \mathcal{C}(\tau)/ |\tau| = 0$ and estimate \eqref{10.9a}
	holds for all $\tau \in \mathbb{R}$ and all sufficiently small $\varepsilon > 0$.
\end{thrm}

\begin{thrm}
	Suppose that $\widehat{N}_0 (\boldsymbol{\theta}) = 0$ for any $\boldsymbol{\theta} \in \mathbb{S}^{d-1}$ and  $\widehat{\mathcal{N}}^{(q)} (\boldsymbol{\theta}_0) \ne 0$  for some $\boldsymbol{\theta}_0 \in \mathbb{S}^{d-1}$ and some $q \in \{1,\ldots,p(\boldsymbol{\theta}_0)\}$. Let $s \ge 2$. Then there does not exist a positive function $\mathcal{C}(\tau)$ such that $\lim_{\tau \to \infty} \mathcal{C}(\tau)/ |\tau|^{1/2} = 0$ and estimate \eqref{10.9a}
	holds for all $\tau \in \mathbb{R}$ and all sufficiently small $\varepsilon > 0$.
\end{thrm}

\subsection{Homogenization of the sandwiched operator $e^{-i\tau\mathcal{A}_\varepsilon}$}

Now, we proceed to the more general operator $\mathcal{A}_{\varepsilon}$ (see~(\ref{A_eps})).
Let $\mathcal{A}^0$ be defined by~(\ref{A0}). Using the relations of the form~(\ref{exp_scale_transform}) (for the operators $\mathcal{A}_\varepsilon$ and $\mathcal{A}^0$), and~(\ref{H0_resolv_scale_transform}), we obtain
\begin{equation}
\label{sndw_exps_scale_transform}
\bigl(f^\varepsilon e^{-i \tau \mathcal{A}_{\varepsilon}} (f^\varepsilon)^{-1} - f_0 e^{-i \tau \mathcal{A}^0} f_0^{-1} \bigr) (\mathcal{H}_0 + I)^{-s/2} =  T_{\varepsilon}^* J(\varepsilon; \tau) \mathcal{R} (\varepsilon)^{s/2} T_{\varepsilon}, \qquad \varepsilon > 0.
\end{equation}

From Theorem~\ref{sndw_A_exp_general_thrm} and  identity~(\ref{sndw_exps_scale_transform}) we deduce the following result proved before in~\cite[Theorem~12.4]{BSu2008}.

\begin{thrm}[\cite{BSu2008}]
	\label{sndw_A_eps_exp_general_thrm}
	Let $\mathcal{A}_{\varepsilon}$ and $\mathcal{A}^0$ be the operators defined by~\emph{(\ref{A_eps})} and~\emph{(\ref{A0})}. Then for $0 \le s \le 3$, and $\tau \in \mathbb{R}$, $\varepsilon > 0$ we have
	\begin{equation*}
	\| f^\varepsilon e^{-i \tau \mathcal{A}_{\varepsilon}} (f^\varepsilon)^{-1} - f_0 e^{-i \tau \mathcal{A}^0} f_0^{-1} \|_{H^s (\mathbb{R}^d) \to L_2 (\mathbb{R}^d)} \le  \mathfrak{C}_1 (s) (1 +  |\tau|)^{s/3} \varepsilon^{s/3},  
	\end{equation*}
	where $\mathfrak{C}_1 (s) = (2 \|f\|_{L_\infty} \|f^{-1}\|_{L_\infty})^{1-s/3} \mathcal{C}_1^{s/3}$. The constant $\mathcal{C}_1$ depends only on $\alpha_0$, $\alpha_1$, $\|g\|_{L_\infty}$, $\|g^{-1}\|_{L_\infty}$, $\|f\|_{L_{\infty}}$, $\|f^{-1}\|_{L_{\infty}}$, and $r_0$.
\end{thrm}

This result can be improved under some additional assumptions. Applying Theorem~\ref{sndw_A_exp_enchcd_thrm_1} and taking into account~(\ref{sndw_exps_scale_transform}) and the obvious estimate
\begin{equation*}
\| f^\varepsilon e^{-i \tau \mathcal{A}_{\varepsilon}} (f^\varepsilon)^{-1} - f_0 e^{-i \tau \mathcal{A}^0} f_0^{-1} \|_{L_2(\mathbb{R}^d) \to L_2(\mathbb{R}^d)} \le 2 \|f\|_{L_\infty} \|f^{-1}\|_{L_\infty},
\end{equation*}
we obtain the following result.

\begin{thrm}
	\label{sndw_A_eps_exp_enchcd_thrm_1}
	Suppose that the assumptions of Theorem~\emph{\ref{sndw_A_eps_exp_general_thrm}} are satisfied. Suppose that the operator $\widehat{N}_Q (\boldsymbol{\theta})$ defined by~\emph{(\ref{N_Q(theta)})} is equal to zero\emph{:} $\widehat{N}_Q (\boldsymbol{\theta}) = 0$ for any $\boldsymbol{\theta} \in \mathbb{S}^{d-1}$. Then for $0 \le s \le 2$, and $\tau \in \mathbb{R}$, $\varepsilon > 0$ we have
	\begin{equation*}
	\| f^\varepsilon e^{-i \tau \mathcal{A}_{\varepsilon}} (f^\varepsilon)^{-1} - f_0 e^{-i \tau \mathcal{A}^0} f_0^{-1} \|_{H^s (\mathbb{R}^d) \to L_2 (\mathbb{R}^d)} \le  \mathfrak{C}_2 (s) (1 +  |\tau|^{1/2})^{s/2} \varepsilon^{s/2},  
	\end{equation*}
	where $\mathfrak{C}_2 (s) = (2 \|f\|_{L_\infty} \|f^{-1}\|_{L_\infty})^{1-s/2} \mathcal{C}_2^{s/2}$. The constant $\mathcal{C}_2$ depends only on $\alpha_0$, $\alpha_1$, $\|g\|_{L_\infty}$, $\|g^{-1}\|_{L_\infty}$, $\|f\|_{L_{\infty}}$, $\|f^{-1}\|_{L_{\infty}}$, and $r_0$.
\end{thrm}

Similarly, application of Theorem~\ref{sndw_A_exp_enchcd_thrm_2} yields the following results.

\begin{thrm}
	\label{sndw_A_eps_exp_enchcd_thrm_2}
	Suppose that the assumptions of Theorem~\emph{\ref{sndw_A_eps_exp_general_thrm}} are satisfied. Suppose that Condition~\emph{\ref{sndw_cond1}} \emph{(}or more restrictive Condition~\emph{\ref{sndw_cond2})} is satisfied. Then for $0 \le s \le 2$, and $\tau \in \mathbb{R}$, $\varepsilon > 0$ we have
	\begin{equation*}
	\| f^\varepsilon e^{-i \tau \mathcal{A}_{\varepsilon}} (f^\varepsilon)^{-1} - f_0 e^{-i \tau \mathcal{A}^0} f_0^{-1} \|_{H^s (\mathbb{R}^d) \to L_2 (\mathbb{R}^d)} \le  \mathfrak{C}_3 (s) (1 +  |\tau|^{1/2})^{s/2} \varepsilon^{s/2}, 
	\end{equation*}
	where $\mathfrak{C}_3 (s) = (2 \|f\|_{L_\infty} \|f^{-1}\|_{L_\infty})^{1-s/2}  \mathcal{C}_3^{s/2}$. The constant $\mathcal{C}_3$ depends on $\alpha_0$, $\alpha_1$, $\|g\|_{L_\infty}$, $\|g^{-1}\|_{L_\infty}$, $\|f\|_{L_{\infty}}$, $\|f^{-1}\|_{L_{\infty}}$, $r_0$, and also on $n$ and the number $c^{\circ}$.
\end{thrm}

Theorems~\ref{sndw_A_eps_exp_enchcd_thrm_1} and~\ref{sndw_A_eps_exp_enchcd_thrm_2} improve the results of Theorems~13.8 and~13.10 from~\cite{Su2017} with respect to  dependence of the estimates on $\tau$.

Application of Theorems~\ref{sndw_A_exp_smooth_shrp_thrm_1} and \ref{sndw_A_exp_smooth_shrp_thrm_2} allows us to confirm that the results of Theorems~\ref{sndw_A_eps_exp_general_thrm}, \ref{sndw_A_eps_exp_enchcd_thrm_1}, and \ref{sndw_A_eps_exp_enchcd_thrm_2} are sharp with respect to the type of the operator norm.

\begin{thrm}[\cite{Su2017}]
	\label{sndw_A_eps_exp_smooth_shrp_thrm_1}
	Suppose that $\widehat{N}_{0,Q} (\boldsymbol{\theta}_0) \ne 0$ for some $\boldsymbol{\theta}_0 \in \mathbb{S}^{d-1}$. Let $\tau \ne 0$ and $0 \le s < 3$. Then there does not exist a constant $\mathcal{C}(\tau) >0$ such that the estimate
	\begin{equation}
	\label{10.10a}
	\| f^\varepsilon e^{-i \tau \mathcal{A}_{\varepsilon}} (f^\varepsilon)^{-1} - f_0 e^{-i \tau \mathcal{A}^0} f_0^{-1} \|_{H^s (\mathbb{R}^d) \to L_2 (\mathbb{R}^d)} \le \mathcal{C}(\tau) \varepsilon
	\end{equation}
	holds for all sufficiently small $\varepsilon > 0$.
\end{thrm}

\begin{thrm}
	Suppose that $\widehat{N}_{0,Q} (\boldsymbol{\theta}) = 0$ for any $\boldsymbol{\theta} \in \mathbb{S}^{d-1}$ and  $\widehat{\mathcal{N}}_Q^{(q)} (\boldsymbol{\theta}_0) \ne 0$ for some $\boldsymbol{\theta}_0 \in \mathbb{S}^{d-1}$ and some $q \in \{1,\ldots,p(\boldsymbol{\theta}_0)\}$. 
	Let $\tau \ne 0$ and $0 \le s < 2$. Then there does not exist a constant $\mathcal{C}(\tau) >0$ such that estimate \eqref{10.10a}
	holds for all sufficiently small $\varepsilon > 0$.
\end{thrm}

Theorem~\ref{sndw_A_eps_exp_smooth_shrp_thrm_1} was proved in~\cite[Theorem~13.12]{Su2017}.

Application of Theorem~\ref{sndw_A_exp_time_shrp_thrm_1} allows us to confirm that the result of Theorem~\ref{sndw_A_eps_exp_general_thrm} is sharp with respect to dependence of the estimates on time.

\begin{thrm}
	Suppose that $\widehat{N}_{0,Q} (\boldsymbol{\theta}_0) \ne 0$ for some $\boldsymbol{\theta}_0 \in \mathbb{S}^{d-1}$. Let $s \ge 3$. Then there does not exist a positive function $\mathcal{C}(\tau)$ such that $\lim_{\tau \to \infty} \mathcal{C}(\tau)/ |\tau| = 0$ and estimate \eqref{10.10a}
	holds for all $\tau \in \mathbb{R}$ and all sufficiently small $\varepsilon > 0$.
\end{thrm}

Similarly, application of Theorem~\ref{sndw_A_exp_time_shrp_thrm_2} allows us to confirm that the results of Theorems~\ref{sndw_A_eps_exp_enchcd_thrm_1} and \ref{sndw_A_eps_exp_enchcd_thrm_2} are sharp.

\begin{thrm}
	Suppose that $\widehat{N}_{0,Q} (\boldsymbol{\theta}) = 0$ for any $\boldsymbol{\theta} \in \mathbb{S}^{d-1}$ and  $\widehat{\mathcal{N}}_Q^{(q)} (\boldsymbol{\theta}_0) \ne 0$  for some $\boldsymbol{\theta}_0 \in \mathbb{S}^{d-1}$ and some $q \in \{1,\ldots,p(\boldsymbol{\theta}_0)\}$. Let $s \ge 2$. Then there does not exist a positive function $\mathcal{C}(\tau)$ such that $\lim_{\tau \to \infty} \mathcal{C}(\tau)/ |\tau|^{1/2} = 0$ and estimate \eqref{10.10a}
	holds for all $\tau \in \mathbb{R}$ and all sufficiently small $\varepsilon > 0$.
\end{thrm}

\section{Homogenization of the Cauchy problem for the Schr\"{o}dinger-type equations}
\label{main_results_Cauchy_section}
\subsection{The Cauchy problem for the equation with the operator $\widehat{\mathcal{A}}_\varepsilon$}
Let $\mathbf{u}_\varepsilon (\mathbf{x}, \tau)$ be the solution of the Cauchy problem
\begin{equation}
\label{Cauchy_hatA_eps}
\left\{
\begin{aligned}
&i\frac{\partial \mathbf{u}_\varepsilon (\mathbf{x}, \tau)}{\partial \tau} = b(\mathbf{D})^*g^{\varepsilon}(\mathbf{x}) b(\mathbf{D}) \mathbf{u}_\varepsilon (\mathbf{x}, \tau) + \mathbf{F} (\mathbf{x}, \tau), \qquad \mathbf{x} \in \mathbb{R}^d, \; \tau \in \mathbb{R}, \\
& \mathbf{u}_\varepsilon (\mathbf{x}, 0) = \boldsymbol{\phi} (\mathbf{x}), \qquad \mathbf{x} \in \mathbb{R}^d,
\end{aligned}
\right.
\end{equation}
where $\boldsymbol{\phi} \in L_2 (\mathbb{R}^d; \mathbb{C}^n)$, $ \mathbf{F} \in L_{1, \mathrm{loc}} (\mathbb{R}; L_2 (\mathbb{R}^d; \mathbb{C}^n))$. The solution can be represented as
\begin{equation*}
\mathbf{u}_\varepsilon (\cdot, \tau) = e^{-i\tau\widehat{\mathcal{A}}_{\varepsilon}} \boldsymbol{\phi} - i \int_{0}^{\tau} e^{-i(\tau-\tilde{\tau})\widehat{\mathcal{A}}_{\varepsilon}} \mathbf{F} (\cdot, \tilde{\tau}) \, d \tilde{\tau}.
\end{equation*}
Let $\mathbf{u}_0 (\mathbf{x}, \tau)$ be the solution of the homogenized problem
\begin{equation}
\label{Cauchy_hatA0}
\left\{
\begin{aligned}
&i\frac{\partial \mathbf{u}_0 (\mathbf{x}, \tau)}{\partial \tau} = b(\mathbf{D})^* g^{0} b(\mathbf{D}) \mathbf{u}_0 (\mathbf{x}, \tau) + \mathbf{F} (\mathbf{x}, \tau), \qquad \mathbf{x} \in \mathbb{R}^d, \; \tau \in \mathbb{R}, \\
& \mathbf{u}_0 (\mathbf{x}, 0) = \boldsymbol{\phi} (\mathbf{x}), \qquad \mathbf{x} \in \mathbb{R}^d.
\end{aligned}
\right.
\end{equation}
Then
\begin{equation*}
\mathbf{u}_0 (\cdot, \tau) = e^{-i\tau\widehat{\mathcal{A}}^0} \boldsymbol{\phi} - i\int_{0}^{\tau} e^{-i(\tau-\tilde{\tau})\widehat{\mathcal{A}}^0} \mathbf{F} (\cdot, \tilde{\tau}) \, d \tilde{\tau}.
\end{equation*}

The following result is deduced from Theorem~\ref{hatA_eps_exp_general_thrm} (it has been proved before in~\cite[Theorem~14.2]{BSu2008}).

\begin{thrm}[\cite{BSu2008}]
	\label{hatA_eps_Cauchy_general_thrm}
	Let $\mathbf{u}_\varepsilon$ be the solution of problem~\emph{(\ref{Cauchy_hatA_eps})} and let $\mathbf{u}_0$ be the solution of problem~\emph{(\ref{Cauchy_hatA0})}.
	\begin{enumerate}[label=\emph{\arabic*$^{\circ}.$}, ref=\arabic*$^{\circ}$, leftmargin=2\parindent]
	\item If $\boldsymbol{\phi} \in H^s (\mathbb{R}^d; \mathbb{C}^n)$, $\mathbf{F} \in L_{1,\mathrm{loc}}(\mathbb{R}; H^s (\mathbb{R}^d; \mathbb{C}^n))$, where $0 \le s \le 3$, then for $\tau \in \mathbb{R}$ and $\varepsilon > 0$ we have
	\begin{equation*}
	\| \mathbf{u}_\varepsilon (\cdot, \tau) - \mathbf{u}_0 (\cdot, \tau) \|_{L_2(\mathbb{R}^d)} \le 
	\widehat{\mathfrak{C}}_1 (s) \varepsilon^{s/3} (1 +  |\tau|)^{s/3} \left(\| \boldsymbol{\phi} \|_{H^s(\mathbb{R}^d)} + \|\mathbf{F} \|_{L_1((0,\tau);H^s(\mathbb{R}^d))} \right).
	\end{equation*}
	Under the additional assumption that $\mathbf{F} \in L_p (\mathbb{R}_\pm; H^s (\mathbb{R}^d, \mathbb{C}^n))$, where $p \in [1, \infty]$, for $\tau = \pm \varepsilon^{-\alpha}$, $0 < \varepsilon \le 1$, and $0 < \alpha < s(s + 3/p')^{-1}$ we have
	\begin{multline*}
	\| \mathbf{u}_\varepsilon (\cdot, \pm \varepsilon^{-\alpha}) - \mathbf{u}_0 (\cdot, \pm \varepsilon^{-\alpha}) \|_{L_2(\mathbb{R}^d)} \\ 
	\le 2^{s/3} \widehat{\mathfrak{C}}_1 (s) \varepsilon^{s(1-\alpha)/3} \left(\| \boldsymbol{\phi} \|_{H^s(\mathbb{R}^d)} + \varepsilon^{-\alpha/p'} \|\mathbf{F} \|_{L_p(\mathbb{R}_\pm;H^s(\mathbb{R}^d))} \right).
	\end{multline*}
	The constant $\widehat{\mathfrak{C}}_1 (s)$ is defined in Theorem~\emph{\ref{hatA_eps_exp_general_thrm}}. Here $p^{-1} + (p')^{-1} = 1$.
	
	\item If $\boldsymbol{\phi} \in L_2 (\mathbb{R}^d; \mathbb{C}^n)$ and $ \mathbf{F} \in L_{1, \mathrm{loc}} (\mathbb{R}; L_2 (\mathbb{R}^d; \mathbb{C}^n) )$, then
	\begin{equation*}
	\lim\limits_{\varepsilon \to 0} \| \mathbf{u}_\varepsilon (\cdot, \tau) - \mathbf{u}_0 (\cdot, \tau) \|_{L_2(\mathbb{R}^d)} = 0, \qquad \tau \in \mathbb{R}.
	\end{equation*}
	Under the additional assumption that $\mathbf{F} \in L_1 (\mathbb{R}_\pm; L_2 (\mathbb{R}^d, \mathbb{C}^n))$, we have
	\begin{equation*}
	\lim\limits_{\varepsilon \to 0} \| \mathbf{u}_\varepsilon (\cdot, \pm\varepsilon^{-\alpha}) - \mathbf{u}_0 (\cdot, \pm\varepsilon^{-\alpha}) \|_{L_2(\mathbb{R}^d)} = 0, \qquad 0 < \alpha < 1.
	\end{equation*}
\end{enumerate}
\end{thrm}

The result of Theorem~\ref{hatA_eps_Cauchy_general_thrm} can be refined under some additional assumptions. Applying Theorem~\ref{hatA_eps_exp_enchcd_thrm_1}, we obtain the following result.

\begin{thrm}
	\label{hatA_eps_Cauchy_enchcd_thrm_1}
	Suppose that the assumptions of Theorem~\emph{\ref{hatA_eps_Cauchy_general_thrm}} are satisfied. Let $\widehat{N} (\boldsymbol{\theta})$ be the operator defined by~\emph{(\ref{hatN(theta)})}. Suppose that $\widehat{N} (\boldsymbol{\theta}) = 0$ for any $\boldsymbol{\theta} \in \mathbb{S}^{d-1}$.
	\begin{enumerate}[label=\emph{\arabic*$^{\circ}.$}, ref=\arabic*$^{\circ}$, leftmargin=2\parindent]
	\item If $\boldsymbol{\phi} \in H^s (\mathbb{R}^d; \mathbb{C}^n)$, $\mathbf{F} \in L_{1,\mathrm{loc}}(\mathbb{R}; H^s (\mathbb{R}^d; \mathbb{C}^n))$, where $0 \le s \le 2$, then for $\tau \in \mathbb{R}$ and $\varepsilon > 0$ we have
	\begin{equation*}
	\| \mathbf{u}_\varepsilon (\cdot, \tau) - \mathbf{u}_0 (\cdot, \tau) \|_{L_2(\mathbb{R}^d)} \le
	\widehat{\mathfrak{C}}_2 (s) \varepsilon^{s/2} (1 +  |\tau|^{1/2})^{s/2}  \left(\| \boldsymbol{\phi} \|_{H^s(\mathbb{R}^d)} + \|\mathbf{F} \|_{L_1((0,s);H^s(\mathbb{R}^d))} \right).
	\end{equation*}
	Under the additional assumption that $\mathbf{F} \in L_p (\mathbb{R}_\pm; H^s (\mathbb{R}^d, \mathbb{C}^n))$, where $p \in [1, \infty]$, for $\tau = \pm \varepsilon^{-\alpha}$, $0 < \varepsilon \le 1$, and $0 < \alpha < 2s(s + 4/p')^{-1}$ we have
	\begin{multline*}
	\| \mathbf{u}_\varepsilon (\cdot, \pm \varepsilon^{-\alpha}) - \mathbf{u}_0 (\cdot, \pm \varepsilon^{-\alpha}) \|_{L_2(\mathbb{R}^d)} \\ 
	\le 2^{s/2} \widehat{\mathfrak{C}}_2 (s) \varepsilon^{s(1-\alpha/2)/2}  \left(\| \boldsymbol{\phi} \|_{H^s(\mathbb{R}^d)} + \varepsilon^{-\alpha/p'} \|\mathbf{F} \|_{L_p(\mathbb{R}_\pm;H^s(\mathbb{R}^d))} \right).
	\end{multline*}
	The constant $\widehat{\mathfrak{C}}_2 (s)$ is defined in Theorem~\emph{\ref{hatA_eps_exp_enchcd_thrm_1}}. Here $p^{-1} + (p')^{-1} = 1$.
	
	\item If $\boldsymbol{\phi} \in L_2 (\mathbb{R}^d; \mathbb{C}^n)$, and $\mathbf{F} \in L_1 (\mathbb{R}_\pm; L_2 (\mathbb{R}^d, \mathbb{C}^n))$, then
	\begin{equation*}
	\lim\limits_{\varepsilon \to 0} \| \mathbf{u}_\varepsilon (\cdot, \pm\varepsilon^{-\alpha}) - \mathbf{u}_0 (\cdot, \pm\varepsilon^{-\alpha}) \|_{L_2(\mathbb{R}^d)} = 0, \qquad 0 < \alpha < 2.
	\end{equation*}
\end{enumerate}
\end{thrm}

Similarly, Theorem~\ref{hatA_eps_exp_enchcd_thrm_2} implies the following result.

\begin{thrm}
	Suppose that the assumptions of Theorem~\emph{\ref{hatA_eps_Cauchy_general_thrm}} are satisfied. Suppose that Condition~\emph{\ref{cond1}} \emph{(}or more restrictive Condition~\emph{\ref{cond2})} is satisfied.
	\begin{enumerate}[label=\emph{\arabic*$^{\circ}.$}, ref=\arabic*$^{\circ}$, leftmargin=2\parindent]
	\item If $\boldsymbol{\phi} \in H^s (\mathbb{R}^d; \mathbb{C}^n)$, $\mathbf{F} \in L_{1,\mathrm{loc}}(\mathbb{R}; H^s (\mathbb{R}^d; \mathbb{C}^n))$, where $0 \le s \le 2$, then for $\tau \in \mathbb{R}$ and $\varepsilon > 0$ we have
	\begin{equation*}
	\| \mathbf{u}_\varepsilon (\cdot, \tau) - \mathbf{u}_0 (\cdot, \tau) \|_{L_2(\mathbb{R}^d)} \le \widehat{\mathfrak{C}}_3 (s) \varepsilon^{s/2} (1 +  |\tau|^{1/2})^{s/2}  \left(\| \boldsymbol{\phi} \|_{H^s(\mathbb{R}^d)} + \|\mathbf{F} \|_{L_1((0,\tau);H^s(\mathbb{R}^d))} \right).
	\end{equation*}
	Under the additional assumption that $\mathbf{F} \in L_p (\mathbb{R}_\pm; H^s (\mathbb{R}^d, \mathbb{C}^n))$, where $p \in [1, \infty]$, for $\tau = \pm \varepsilon^{-\alpha}$, $0 < \varepsilon \le 1$, and $0 < \alpha < 2s(s + 4/p')^{-1}$ we have
	\begin{multline*}
	\| \mathbf{u}_\varepsilon (\cdot, \pm \varepsilon^{-\alpha}) - \mathbf{u}_0 (\cdot, \pm \varepsilon^{-\alpha}) \|_{L_2(\mathbb{R}^d)} \\ 
	\le 2^{s/2} \widehat{\mathfrak{C}}_3 (s) \varepsilon^{s(1-\alpha/2)/2} \left(\| \boldsymbol{\phi} \|_{H^s(\mathbb{R}^d)} + \varepsilon^{-\alpha/p'} \|\mathbf{F} \|_{L_p(\mathbb{R}_\pm;H^s(\mathbb{R}^d))} \right).
	\end{multline*}
	The constant $\widehat{\mathfrak{C}}_3 (s)$ is defined in Theorem~\emph{\ref{hatA_eps_exp_enchcd_thrm_2}}. Here $p^{-1} + (p')^{-1} = 1$.
	\item If $\boldsymbol{\phi} \in L_2 (\mathbb{R}^d; \mathbb{C}^n)$, and $\mathbf{F} \in L_1 (\mathbb{R}_\pm; L_2 (\mathbb{R}^d, \mathbb{C}^n))$, then
	\begin{equation*}
	\lim\limits_{\varepsilon \to 0} \| \mathbf{u}_\varepsilon (\cdot, \pm\varepsilon^{-\alpha}) - \mathbf{u}_0 (\cdot, \pm\varepsilon^{-\alpha}) \|_{L_2(\mathbb{R}^d)} = 0, \qquad 0 < \alpha < 2.
	\end{equation*}
\end{enumerate}
\end{thrm}

\subsection{The Cauchy problem for the equation with the operator $\mathcal{A}_\varepsilon$}

Now, we consider more general Cauchy problem for the equation involving the operator $\mathcal{A}_\varepsilon$:
\begin{equation}
\label{Cauchy_A_eps}
\left\{
\begin{aligned}
&i\frac{\partial \mathbf{u}_\varepsilon (\mathbf{x}, \tau)}{\partial \tau} = (f^{\varepsilon}(\mathbf{x}))^* b(\mathbf{D})^* g^{\varepsilon}(\mathbf{x}) b(\mathbf{D}) f^{\varepsilon}(\mathbf{x}) \mathbf{u}_\varepsilon (\mathbf{x}, \tau) +  (f^\varepsilon (\mathbf{x}))^{-1} \mathbf{F} (\mathbf{x}, \tau), \qquad \mathbf{x} \in \mathbb{R}^d, \; \tau \in \mathbb{R},\\
& f^\varepsilon (\mathbf{x}) \mathbf{u}_\varepsilon (\mathbf{x}, 0) = \boldsymbol{\phi}(\mathbf{x}), \qquad \mathbf{x} \in \mathbb{R}^d,
\end{aligned}
\right.
\end{equation}
where $\boldsymbol{\phi} \in L_2 (\mathbb{R}^d; \mathbb{C}^n)$, $ \mathbf{F} \in L_{1, \mathrm{loc}} (\mathbb{R}; L_2 (\mathbb{R}^d; \mathbb{C}^n) )$. The solution can be represented as
\begin{equation*}
\mathbf{u}_\varepsilon (\cdot, \tau) = e^{-i\tau\mathcal{A}_{\varepsilon}} (f^\varepsilon)^{-1} \boldsymbol{\phi} - i \int_{0}^{\tau} e^{-i(\tau-\tilde{\tau})\mathcal{A}_{\varepsilon}} (f^\varepsilon)^{-1} \mathbf{F} (\cdot, \tilde{\tau}) \, d \tilde{\tau}.
\end{equation*}
Let $\mathbf{u}_0 (\mathbf{x}, \tau)$ be the solution of the homogenized problem
\begin{equation}
\label{Cauchy_A0}
\left\{
\begin{aligned}
&i\frac{\partial \mathbf{u}_0 (\mathbf{x}, \tau)}{\partial \tau} = f_0 b(\mathbf{D})^* g^0 b(\mathbf{D}) f_0 \mathbf{u}_0 (\mathbf{x}, \tau) +  f_0^{-1} \mathbf{F} (\mathbf{x}, \tau), \qquad \mathbf{x} \in \mathbb{R}^d, \; \tau \in \mathbb{R},\\
& f_0 \mathbf{u}_0 (\mathbf{x}, 0) = \boldsymbol{\phi}(\mathbf{x}), \qquad \mathbf{x} \in \mathbb{R}^d.
\end{aligned}
\right.
\end{equation}
Then
\begin{equation*}
\mathbf{u}_0 (\cdot, \tau) = e^{-i\tau\mathcal{A}^0} f_0^{-1} \boldsymbol{\phi} - i \int_{0}^{\tau} e^{-i(\tau-\tilde{\tau})\mathcal{A}^0} f_0^{-1} \mathbf{F} (\cdot, \tilde{\tau}) \, d \tilde{\tau}.
\end{equation*}

The following result is deduced from Theorem~\ref{sndw_A_eps_exp_general_thrm} (it has been proved before in~\cite[Theorem~14.5]{BSu2008}).

\begin{thrm}[\cite{BSu2008}]
	\label{A_eps_Cauchy_general_thrm}
	Let $\mathbf{u}_\varepsilon$ be the solution of problem~\emph{(\ref{Cauchy_A_eps})}, and let $\mathbf{u}_0$ be the solution of problem~\emph{(\ref{Cauchy_A0})}.
	\begin{enumerate}[label=\emph{\arabic*$^{\circ}.$}, ref=\arabic*$^{\circ}$, leftmargin=2\parindent]
	\item If $\boldsymbol{\phi} \in H^s (\mathbb{R}^d; \mathbb{C}^n)$, $\mathbf{F} \in L_{1,\mathrm{loc}}(\mathbb{R}; H^s (\mathbb{R}^d; \mathbb{C}^n))$, where $0 \le s \le 3$, then for $\tau \in \mathbb{R}$ and $\varepsilon > 0$ we have
	\begin{equation*}
	\|f^\varepsilon \mathbf{u}_\varepsilon (\cdot, \tau) - f_0 \mathbf{u}_0 (\cdot, \tau) \|_{L_2(\mathbb{R}^d)} \le \mathfrak{C}_1 (s) \varepsilon^{s/3} (1 +  |\tau|)^{s/3}   \left(\| \boldsymbol{\phi} \|_{H^s(\mathbb{R}^d)} + \|\mathbf{F} \|_{L_1((0,\tau);H^s(\mathbb{R}^d))} \right).
	\end{equation*}
	Under the additional assumption that $\mathbf{F} \in L_p (\mathbb{R}_\pm; H^s (\mathbb{R}^d, \mathbb{C}^n))$, where $p \in [1, \infty]$, for $\tau = \pm \varepsilon^{-\alpha}$, $0 < \varepsilon \le 1$, and $0 < \alpha < s(s + 3/p')^{-1}$ we have
	\begin{multline*}
	\| f^\varepsilon \mathbf{u}_\varepsilon (\cdot, \pm \varepsilon^{-\alpha}) - f_0 \mathbf{u}_0 (\cdot, \pm \varepsilon^{-\alpha}) \|_{L_2(\mathbb{R}^d)} \\ 
	\le 2^{s/3}\mathfrak{C}_1 (s) \varepsilon^{s(1-\alpha)/3} \left(\| \boldsymbol{\phi} \|_{H^s(\mathbb{R}^d)} + \varepsilon^{-\alpha/p'} \|\mathbf{F} \|_{L_p(\mathbb{R}_\pm;H^s(\mathbb{R}^d))} \right).
	\end{multline*}
	The constant $\mathfrak{C}_1 (s)$ is defined in Theorem~\emph{\ref{sndw_A_eps_exp_general_thrm}}. Here $p^{-1} + (p')^{-1} = 1$.
	
	\item If $\boldsymbol{\phi} \in L_2 (\mathbb{R}^d; \mathbb{C}^n)$ and $ \mathbf{F} \in L_{1, \mathrm{loc}} (\mathbb{R}; L_2 (\mathbb{R}^d; \mathbb{C}^n) )$, then
	\begin{equation*}
	\lim\limits_{\varepsilon \to 0} \|  f^\varepsilon \mathbf{u}_\varepsilon (\cdot, \tau) - f_0 \mathbf{u}_0 (\cdot, \tau) \|_{L_2(\mathbb{R}^d)} = 0, \qquad \tau \in \mathbb{R}.
	\end{equation*}
	Under the additional assumption that $\mathbf{F} \in L_1 (\mathbb{R}_\pm; L_2 (\mathbb{R}^d, \mathbb{C}^n))$  we have
	\begin{equation*}
	\lim\limits_{\varepsilon \to 0} \| f^\varepsilon \mathbf{u}_\varepsilon (\cdot, \pm\varepsilon^{-\alpha}) - f_0 \mathbf{u}_0 (\cdot, \pm\varepsilon^{-\alpha}) \|_{L_2(\mathbb{R}^d)} = 0, \qquad 0 < \alpha < 1.
	\end{equation*}
\end{enumerate}
\end{thrm}

The result of Theorem~\ref{A_eps_Cauchy_general_thrm} can be refined under some additional assumptions. Application of Theorem~\ref{sndw_A_eps_exp_enchcd_thrm_1} yields the following result.

\begin{thrm}
	\label{A_eps_Cauchy_enchcd_thrm_1}
	Suppose that the assumptions of Theorem~\emph{\ref{A_eps_Cauchy_general_thrm}} are satisfied. Let $\widehat{N}_Q (\boldsymbol{\theta})$ be the operator defined by~\emph{(\ref{N_Q(theta)})}. Suppose that $\widehat{N}_Q (\boldsymbol{\theta}) = 0$ for any $\boldsymbol{\theta} \in \mathbb{S}^{d-1}$.
	\begin{enumerate}[label=\emph{\arabic*$^{\circ}.$}, ref=\arabic*$^{\circ}$, leftmargin=2\parindent]
	\item	If $\boldsymbol{\phi} \in H^s (\mathbb{R}^d; \mathbb{C}^n)$, $\mathbf{F} \in L_{1,\mathrm{loc}}(\mathbb{R}; H^s (\mathbb{R}^d; \mathbb{C}^n))$, where $0 \le s \le 2$, then for $\tau \in \mathbb{R}$ and $\varepsilon > 0$ we have
	\begin{equation*}
	\| f^\varepsilon \mathbf{u}_\varepsilon (\cdot, \tau) - f_0 \mathbf{u}_0 (\cdot, \tau) \|_{L_2(\mathbb{R}^d)} \le  \mathfrak{C}_2 (s) \varepsilon^{s/2} (1 +  |\tau|^{1/2})^{s/2}
	\left(\| \boldsymbol{\phi} \|_{H^s(\mathbb{R}^d)} + \|\mathbf{F} \|_{L_1((0,s);H^s(\mathbb{R}^d))} \right).
	\end{equation*}
	Under the additional assumption that $\mathbf{F} \in L_p (\mathbb{R}_\pm; H^s (\mathbb{R}^d, \mathbb{C}^n))$, where $p \in [1, \infty]$, for $\tau = \pm \varepsilon^{-\alpha}$, $0 < \varepsilon \le 1$, and $0 < \alpha < 2s(s + 4/p')^{-1}$ we have
	\begin{multline*}
	\| f^\varepsilon \mathbf{u}_\varepsilon (\cdot, \pm \varepsilon^{-\alpha}) - f_0 \mathbf{u}_0 (\cdot, \pm \varepsilon^{-\alpha}) \|_{L_2(\mathbb{R}^d)} \\ 
	\le  2^{s/2}\mathfrak{C}_2 (s) \varepsilon^{s(1-\alpha/2)/2}  \left(\| \boldsymbol{\phi} \|_{H^s(\mathbb{R}^d)} + \varepsilon^{-\alpha/p'} \|\mathbf{F} \|_{L_p(\mathbb{R}_\pm;H^s(\mathbb{R}^d))} \right).
	\end{multline*}
	The constant $\mathfrak{C}_2 (s)$ is defined in Theorem~\emph{\ref{sndw_A_eps_exp_enchcd_thrm_1}}. Here $p^{-1} + (p')^{-1} = 1$.
	
	\item If $\boldsymbol{\phi} \in L_2 (\mathbb{R}^d; \mathbb{C}^n)$ and $\mathbf{F} \in L_1 (\mathbb{R}_\pm; L_2 (\mathbb{R}^d, \mathbb{C}^n))$, then
	\begin{equation*}
	\lim\limits_{\varepsilon \to 0} \| f^\varepsilon \mathbf{u}_\varepsilon (\cdot, \pm \varepsilon^{-\alpha}) - f_0 \mathbf{u}_0 (\cdot, \pm \varepsilon^{-\alpha}) \|_{L_2(\mathbb{R}^d)} = 0, \qquad 0 < \alpha < 2.
	\end{equation*}
\end{enumerate}
\end{thrm}

Similarly, applying Theorem~\ref{sndw_A_eps_exp_enchcd_thrm_2}, we arrive at the following statement.

\begin{thrm}
	Suppose that the assumptions of Theorem~\emph{\ref{A_eps_Cauchy_general_thrm}} are satisfied. Suppose that Condition~\emph{\ref{sndw_cond1}} \emph{(}or more restrictive Condition~\emph{\ref{sndw_cond2})} is satisfied.
	\begin{enumerate}[label=\emph{\arabic*$^{\circ}.$}, ref=\arabic*$^{\circ}$, leftmargin=2\parindent]
	\item	If $\boldsymbol{\phi} \in H^s (\mathbb{R}^d; \mathbb{C}^n)$, $\mathbf{F} \in L_{1,\mathrm{loc}}(\mathbb{R}; H^s (\mathbb{R}^d; \mathbb{C}^n))$, where $0 \le s \le 2$, then for $\tau \in \mathbb{R}$ and $\varepsilon > 0$ we have
	\begin{equation*}
	\| f^\varepsilon \mathbf{u}_\varepsilon (\cdot, \tau) - f_0 \mathbf{u}_0 (\cdot, \tau) \|_{L_2(\mathbb{R}^d)} \le 
	\mathfrak{C}_3 (s) \varepsilon^{s/2} (1 +  |\tau|^{1/2})^{s/2} \left(\| \boldsymbol{\phi} \|_{H^s(\mathbb{R}^d)} + \|\mathbf{F} \|_{L_1((0,\tau);H^s(\mathbb{R}^d))} \right).
	\end{equation*}
	Under the additional assumption that $\mathbf{F} \in L_p (\mathbb{R}_\pm; H^s (\mathbb{R}^d, \mathbb{C}^n))$, where $p \in [1, \infty]$, for $\tau = \pm \varepsilon^{-\alpha}$, $0 < \varepsilon \le 1$, and $0 < \alpha < 2s(s + 4/p')^{-1}$ we have
	\begin{multline*}
	\| f^\varepsilon \mathbf{u}_\varepsilon (\cdot, \pm \varepsilon^{-\alpha}) - f_0 \mathbf{u}_0 (\cdot, \pm \varepsilon^{-\alpha}) \|_{L_2(\mathbb{R}^d)} \\ 
	\le  2^{s/2} \mathfrak{C}_3 (s) \varepsilon^{s(1-\alpha/2)/2}  \left(\| \boldsymbol{\phi} \|_{H^s(\mathbb{R}^d)} + \varepsilon^{-\alpha/p'} \|\mathbf{F} \|_{L_p(\mathbb{R}_\pm;H^s(\mathbb{R}^d))} \right).
	\end{multline*}
	The constant $\mathfrak{C}_3 (s)$ is defined in Theorem~\emph{\ref{sndw_A_eps_exp_enchcd_thrm_2}}. Here $p^{-1} + (p')^{-1} = 1$.
	
	\item If $\boldsymbol{\phi} \in L_2 (\mathbb{R}^d; \mathbb{C}^n)$ and $\mathbf{F} \in L_1 (\mathbb{R}_\pm; L_2 (\mathbb{R}^d, \mathbb{C}^n))$, then
	\begin{equation*}
	\lim\limits_{\varepsilon \to 0} \| f^\varepsilon \mathbf{u}_\varepsilon (\cdot, \pm \varepsilon^{-\alpha}) - f_0 \mathbf{u}_0 (\cdot, \pm \varepsilon^{-\alpha}) \|_{L_2(\mathbb{R}^d)} = 0, \qquad 0 < \alpha < 2.
	\end{equation*}
\end{enumerate}
\end{thrm}

\section{Applications of the general results}
\label{appl_section}
\subsection{The Schr\"{o}dinger-type equation with the operator $\widehat{\mathcal{A}}_\varepsilon = - \operatorname{div} g^\varepsilon \nabla$}
Consider the scalar elliptic operator
\begin{equation}
\label{appl_Schr}
\widehat{\mathcal{A}} = - \operatorname{div} g (\mathbf{x}) \nabla = \mathbf{D}^* g (\mathbf{x}) \mathbf{D}
\end{equation}
acting in $L_2 (\mathbb{R}^d)$, $d \ge 1$, which is a particular case of the operator~(\ref{hatA}). In this case $n=1$, $m=d$, $b(\mathbf{D}) = \mathbf{D}$.

The effective matrix $g^0$ is defined in the standard way. Let $\psi_j \in \widetilde{H}^1(\Omega)$ be a (weak) $\Gamma$-periodic solution of the problem
\begin{equation}
\label{psi_equation}
\operatorname{div} g(\mathbf{x}) (\nabla \psi_j (\mathbf{x}) + \mathbf{e}_j) = 0, \qquad \int_{\Omega} \psi_j (\mathbf{x}) \, d\mathbf{x} = 0.
\end{equation}
Here $\mathbf{e}_1, \ldots, \mathbf{e}_d$ is the standard orthonormal basis in $\mathbb{R}^d$. The matrix $\widetilde{g}(\mathbf{x})$ is the ($d \times d$)-matrix with the columns $\widetilde{\mathbf{g}}_j (\mathbf{x}) \coloneqq g(\mathbf{x}) (\nabla \psi_j (\mathbf{x}) + \mathbf{e}_j)$, $j=1,\dots,d$.  Then $g^0 = | \Omega |^{-1} \int_{\Omega} \widetilde{g}(\mathbf{x}) \, d \mathbf{x}$.

If $g(\mathbf{x})$ is a symmetric matrix with real entries, then, by Proposition~\ref{N=0_proposit}(\ref{N=0_proposit_it1}),  $\widehat{N} (\boldsymbol{\theta}) = 0$ for any $\boldsymbol{\theta} \in \mathbb{S}^{d-1}$. If $g(\mathbf{x})$ is a Hermitian matrix with complex entries, then, in general, $\widehat{N} (\boldsymbol{\theta})$ is not zero. Since $n=1$, then $\widehat{N} (\boldsymbol{\theta}) = \widehat{N}_0 (\boldsymbol{\theta})$ is the operator of multiplication by $\widehat{\mu}(\boldsymbol{\theta})$, where $\widehat{\mu}(\boldsymbol{\theta})$ is the coefficient in the expansion for the first eigenvalue $\widehat{\lambda}(t, \boldsymbol{\theta}) = \widehat{\gamma} (\boldsymbol{\theta}) t^2 + \widehat{\mu} (\boldsymbol{\theta}) t^3 + \widehat{\nu} (\boldsymbol{\theta}) t^4 + \ldots$ of the operator $\widehat{\mathcal{A}} (\mathbf{k})$. A calculation~(see~\cite[Subsection~10.3]{BSu2005-2}) shows that
\begin{align*}
&\widehat{N} (\boldsymbol{\theta}) = \widehat{\mu} (\boldsymbol{\theta}) = -i \sum_{j,l,r=1}^{d} (a_{jlr} - a_{jlr}^*) \theta_j \theta_l \theta_r, \\
&a_{jlr} = |\Omega|^{-1} \int_{\Omega} \psi_j (\mathbf{x})^* \left\langle g(\mathbf{x}) (\nabla \psi_l (\mathbf{x}) + \mathbf{e}_l), \mathbf{e}_r \right\rangle \, d \mathbf{x}, \qquad j, l, r = 1, \ldots, d. 
\end{align*}
The following example is borrowed from~\cite[Subsection~10.4]{BSu2005-2}.

\begin{example}[\cite{BSu2005-2}]
	Let $d=2$, $\Gamma = (2 \pi \mathbb{Z})^2$, and let $g(\mathbf{x})$ be given by
	\begin{equation*}
	g(\mathbf{x}) = \begin{pmatrix}
	1 & i \beta'(x_1) \\
	- i \beta' (x_1) & 1
	\end{pmatrix},
	\end{equation*}
	where $\beta(x_1)$ is a smooth $(2 \pi)$-periodic real-valued function such that $1 - (\beta'(x_1))^2 > 0$ and $\int_{0}^{2 \pi} \beta(x_1)\, d x_1 = 0$. Then $\widehat{N} (\boldsymbol{\theta}) = - \alpha \pi^{-1} \theta_2^3$, where $\alpha = \int_{0}^{2 \pi} \beta (x_1) (\beta' (x_1))^2 dx_1$. It is easy to give a concrete example where $\alpha \ne 0${\rm :} if $\beta (x_1) = c (\sin x_1 + \cos 2x_1)$ with $0 < c < 1/3$, then $\alpha = - (3 \pi/2) c^3 \ne 0$. In this example $\widehat{N} (\boldsymbol{\theta}) = \widehat{\mu} (\boldsymbol{\theta}) \ne 0$ for all $\boldsymbol{\theta} \in \mathbb{S}^1$ except for the points $(\pm 1 ,0)$.
\end{example}

Next, let $\phi_{jl} (\mathbf{x})$ be a $\Gamma$-periodic solution of the problem
\begin{equation}
\label{phi_equation}
-\operatorname{div} g(\mathbf{x}) (\nabla \phi_{jl} (\mathbf{x}) - \psi_j(\mathbf{x}) \mathbf{e}_l) = g^0_{lj} - \widetilde{g}_{lj} (\mathbf{x}), \qquad \int_{\Omega} \phi_{jl} (\mathbf{x}) \, d\mathbf{x} = 0.
\end{equation}
The operator $\widehat{\mathcal{N}}^{(1,1)} (\boldsymbol{\theta})$ is the operator of multiplication by $\widehat{\nu}(\boldsymbol{\theta})$. A calculation~(see~\cite[Subsection~14.5]{VSu2012}) shows that
\begin{align}
\label{hat_nu(k)}
\widehat{\mathcal{N}}^{(1,1)} (\boldsymbol{\theta}) &= \widehat{\nu}(\boldsymbol{\theta}) = \sum_{p,q,l,r=1}^{d} (\alpha_{pqlr} - (\overline{\psi_p^* \psi_q}) g^0_{lr}) \theta_p \theta_q \theta_l \theta_r,
\\
\notag
\alpha_{pqlr} &= |\Omega|^{-1} \int_{\Omega} (\widetilde{g}_{lp} (\mathbf{x}) \phi_{qr}(\mathbf{x}) + \widetilde{g}_{rq} (\mathbf{x}) \phi_{pl}(\mathbf{x})) \, d \mathbf{x}
\\ 
\notag
&+
|\Omega|^{-1} \int_{\Omega} \left\langle g(\mathbf{x}) (\nabla \phi_{qr} (\mathbf{x}) - \psi_q(\mathbf{x}) \mathbf{e}_r), \nabla \phi_{pl} (\mathbf{x}) - \psi_p(\mathbf{x}) \mathbf{e}_l \right\rangle \, d \mathbf{x},
\\
\notag
&p,q,l,r = 1, \ldots, d.
\end{align}
\begin{lemma}
	Let $d=1$ and $\widehat{\mathcal A} = - \frac{d}{dx} g(x)\frac{d}{dx}$. If $g \ne \const$, then $\widehat{\nu}(-1) = \widehat{\nu}(1) \ne 0$.
\end{lemma}
\begin{proof}
	The problem~(\ref{psi_equation}) now takes the form $\frac{d}{dx} g(x) (\frac{d}{dx} \psi_1 (x) + 1) = 0$, $\overline{\psi_1}= 0$. Then $\frac{d}{dx} \psi_1 (x) = \underline{g} (g(x))^{-1} - 1$. Since $g(x) \ne \const$, then $\underline{g} (g(x))^{-1} - 1 \not\equiv 0$ and therefore $\psi_1 \not\equiv 0$. Next, $\widetilde{g}(x) = \underline{g} = g^0$ and equation~(\ref{phi_equation}) takes the form $\frac{d}{dx} g(x) (\frac{d}{dx} \phi_{11} (x) - \psi_1(x)) = 0$, $\overline{\phi_{11}} = 0$. Then $\frac{d}{dx} \phi_{11} (x) - \psi_1(x) = 0$. It is easy to check that $\alpha_{1111}$ in~(\ref{hat_nu(k)}) is equal to zero: $\alpha_{1111} = 0$. Since $\overline{\psi_1^2} g^0 \ne 0$, then $\widehat{\nu}(-1) = \widehat{\nu}(1) \ne 0$.
\end{proof}

Consider the Cauchy problem~\eqref{Cauchy_hatA_eps} with the operator $\widehat{\mathcal{A}}_\varepsilon = - \operatorname{div} g^\varepsilon (\mathbf{x}) \nabla$. We can apply Theorem~\ref{hatA_eps_Cauchy_general_thrm} in the general case and Theorem~\ref{hatA_eps_Cauchy_enchcd_thrm_1} in the \textquotedblleft real\textquotedblright \ case. These results are sharp with respect to smoothness of the initial data and with respect to dependence of the estimates on time.

\subsection{The nonstationary Schr\"{o}dinger equation with a singular potential}
(See~\cite[Chapter~6, Subsection~1.1]{BSu2003}.) In the space $L_2(\mathbb{R}^d)$, $d \ge 1$, we consider the operator $\mathcal{H} = \mathbf{D}^* \check{g} (\mathbf{x}) \mathbf{D} +  V (\mathbf{x})$, where a symmetric $(d \times d)$-matrix-valued function $\check{g} (\mathbf{x})$ with \emph{real entries} and a \emph{real-valued} potential $V (\mathbf{x})$ are $\Gamma$-periodic and satisfy
\begin{gather*}
\check{g} (\mathbf{x}) > 0, \quad \check{g}, \check{g}^{-1} \in L_\infty;\\
V \in L_q(\Omega), \qquad q > d/2 \; \text{ for } \; d \ge 2, \quad q =1 \; \text{ for } \; d = 1.
\end{gather*}
Adding an appropriate constant to $V (\mathbf{x})$, we may assume that \emph{the point $\lambda = 0$ is the bottom of the spectrum of $\mathcal{H}$}. Then the operator $\mathcal{H}$ can be written in the factorized form:
\begin{equation}
\label{appl_Schr_factorized}
\mathcal{H} = \omega^{-1} \mathbf{D}^* \omega^2 \check{g} \mathbf{D} \omega^{-1},
\end{equation}
where  $\omega(\mathbf{x})$ is a positive $\Gamma$-periodic solution of the equation $\mathbf{D}^* \check{g}(\mathbf{x}) \mathbf{D} \omega(\mathbf{x}) + V(\mathbf{x})\omega(\mathbf{x}) = 0$, $\int_{\Omega}\omega^2(\mathbf{x}) d \mathbf{x}= |\Omega|$. Therefore, the operator~\eqref{appl_Schr_factorized} is a particular case of the operator~(\ref{A}). In this case $n=1$, $m=d$, $b(\mathbf{D}) = \mathbf{D}$,
$g = \omega^2 \check{g}$, $f = \omega^{-1}$.

Let $g^0$ be the effective matrix for the operator~\eqref{appl_Schr} (with $g = \omega^2 \check{g}$). Now $Q(\mathbf{x}) = \omega^2(\mathbf{x})$, and, by the normalization condition for $\omega$, we have  $\overline{Q} = 1$ and $f_0 = (\overline{Q})^{-1/2} = 1$. Therefore, the operator~\eqref{A0} takes the form $\mathcal{H}^0 = \mathbf{D}^* g^0 \mathbf{D}$.

Now we consider the operator
\begin{equation}
\label{appl_Schr_factorized_eps}
\mathcal{H}_\varepsilon = (\omega^\varepsilon)^{-1} \mathbf{D}^* (\omega^\varepsilon)^2 \check{g}^\varepsilon \mathbf{D} (\omega^\varepsilon)^{-1}.
\end{equation}
In the initial form, the operator~\eqref{appl_Schr_factorized_eps} can be written as $
\mathcal{H}_\varepsilon = \mathbf{D}^* \check{g}^\varepsilon \mathbf{D} + \varepsilon^{-2} V^\varepsilon$. Note that this expression contains a large factor $\varepsilon^{-2}$ at the rapidly oscillating potential $V^\varepsilon$.

Let us consider the Cauchy problem~\eqref{Cauchy_A_eps} with the operator~\eqref{appl_Schr_factorized_eps}. By Proposition~\ref{hatN_Q=0_proposit}(\ref{hatN_Q=0_proposit_it1}), $\widehat{N}_Q(\boldsymbol{\theta}) = 0$ for any $\boldsymbol{\theta} \in \mathbb{S}^{d-1}$. We can apply Theorem~\ref{A_eps_Cauchy_enchcd_thrm_1}. This result is sharp with respect to smoothness of the initial data and with respect to dependence of the estimates on time. 

\begin{remark}
	It is also possible to consider the Cauchy problem for the magnetic Schr\"{o}dinger equation with a small magnetic potential, using an appropriate factorization for the magnetic Schr\"{o}dinger operator\emph{;} see~\emph{\cite[Subsection~15.4]{Su2017}}. In this case, we do not have improvement of the general results.
\end{remark}

\subsection{The nonstationary two-dimensional Pauli equation}
(See~\cite[Chapter~4, \S12, Subsection~12.3]{BSu2005-2}.) Let \emph{the magnetic potential} $\mathbf{A} = \{A_1, A_2\}$ be a $\Gamma$-periodic real vector-valued function in $\mathbb{R}^2$ such that $\mathbf{A} \in L_p(\Omega; \mathbb{C}^2)$, $p>2$. By the gauge transformation, we may assume that $\operatorname{div} \mathbf{A} = 0$, $\int_{\Omega} \mathbf{A}(\mathbf{x}) \, d\mathbf{x} =0$. Under these conditions there exists a (unique) $\Gamma$-periodic real-valued function $\varphi$ such that $\nabla \varphi = \{A_2, -A_1\}$, $\int_{\Omega} \varphi(\mathbf{x}) \, d \mathbf{x} = 0$.

In $L_2(\mathbb{R}^2; \mathbb{C}^2)$, we consider the Pauli operator
\begin{equation}
\label{appl_Pauli}
\mathcal{P} = \begin{pmatrix}
P_{-} & 0 \\
0 & P_{+} 
\end{pmatrix}, \qquad
\begin{aligned}
P_{+} &= \omega_{-} \partial_{+} \omega_{+}^2 \partial_{-} \omega_{-},\\
P_{-} &= \omega_{+} \partial_{-} \omega_{-}^2 \partial_{+} \omega_{+},
\end{aligned}
\end{equation}
where $\omega_{\pm} (\mathbf{x}) = e^{\pm \varphi(\mathbf{x})}$ and $\partial_{\pm} = D_1 \pm iD_2$. If the potential $\mathbf A$ is sufficiently smooth, then the blocks $P_{\pm}$ of the operator~\eqref{appl_Pauli} are of the form
\begin{equation*}
P_{\pm} = (\mathbf{D}-\mathbf{A})^2 \pm B, \qquad B = \partial_1 A_2 - \partial_2 A_1,
\end{equation*}
where the expression $B$ corresponds to the strength of the magnetic field.

The operator~\eqref{appl_Pauli} can be written as $\mathcal{P} = f_\times b_\times(\mathbf{D}) g_\times b_\times(\mathbf{D}) f_\times$, where
\begin{equation*}
b_\times(\mathbf{D}) = \begin{pmatrix}
0 & \partial_{-}\\
\partial_{+} & 0
\end{pmatrix}, \qquad f_\times (\mathbf{x}) = \begin{pmatrix}
\omega_{+}(\mathbf{x})  & 0\\
0 & \omega_{-}(\mathbf{x})
\end{pmatrix}, \qquad g_\times (\mathbf{x}) = \begin{pmatrix}
\omega^2_{+}(\mathbf{x}) & 0\\
0 & \omega^2_{-}(\mathbf{x})
\end{pmatrix}.
\end{equation*}
The operator $\mathcal{P}$  is of the form~\eqref{A} with
$m=n=d=2$, $b(\mathbf{D}) = b_\times(\mathbf{D})$, $f(\mathbf{x})=f_\times(\mathbf{x})$, $g(\mathbf{x})=g_\times(\mathbf{x})$.

Since $m = n$, the effective matrix is equal to $
g_{\times}^0 = \underline{g_{\times}} = \operatorname{diag}\{g_{+}^0, g_{-}^0\}$, $g_{\pm}^0 = \underline{\omega_{\pm}^2}$.
The matrix $Q_{\times} = f_{\times}^{-2} = g_{\times}^{-1}$ plays the role of  $Q$. Then $\overline{Q_\times} = \operatorname{diag}\{(g_{+}^0)^{-1}, (g_{-}^0)^{-1}\}$. The role of $f_0$ is played by $f_{\times, 0} = \operatorname{diag} \{(g_{+}^0)^{1/2}, (g_{-}^0)^{1/2}\}$. The operator~(\ref{A0}) now takes the form 
\begin{equation*}
\mathcal{P}_{\times}^0 = f_{\times,0} b_\times(\mathbf{D}) g^0_\times b_\times(\mathbf{D}) f_{\times,0} = \begin{pmatrix}
-\gamma \Delta & 0 \\
0 & -\gamma \Delta
\end{pmatrix}.
\end{equation*}
Here $\gamma = g_{+}^0 g_{-}^0 = |\Omega|^2 \| \omega_{+} \|_{L_2(\Omega)}^{-2} \| \omega_{-} \|_{L_2(\Omega)}^{-2}$.

Now, we describe the operator $\widehat{N}_{Q, \times} (\boldsymbol{\theta})$ that plays the role of $\widehat{N}_Q (\boldsymbol{\theta})$ for $\mathcal{P}$. Let $w_{\pm}(\mathbf{x})$ be the $\Gamma$-periodic solutions of the problems $
\partial_{\mp} w_{\pm}(\mathbf{x}) = g_{\pm}^0 \omega_{\mp}^2 (\mathbf{x}) - 1$, $\int_{\Omega} w_{\pm}(\mathbf{x}) \, d \mathbf{x} = 0$.
Then $\widehat{N}_{Q, \times} (\boldsymbol{\theta}) = \operatorname{diag} \{\widehat{N}_{Q,-}(\boldsymbol{\theta}), \widehat{N}_{Q,+}(\boldsymbol{\theta})\}$, $
\widehat{N}_{Q,\pm}(\boldsymbol{\theta}) = - 2 \gamma \left( \theta_1 \Re \overline{\omega_{\pm}^2 w_{\pm}} \pm \theta_2 \Im \overline{\omega_{\pm}^2 w_{\pm}} \right)$, $\boldsymbol{\theta} \in \mathbb{S}^1$. The following example is borrowed form~\cite[Example~16.2]{Su2017}. 

\begin{example}[\cite{Su2017}]
	Let $\Gamma = (2 \pi {\mathbb Z})^2$ and let $\omega_{-}^2 (\mathbf{x}) = 1 + \alpha  (\sin x_2 + 4 \sin 2x_2)$, where $\alpha>0$
	is sufficiently small. Then $\widehat{N}_{Q, \times} (\boldsymbol{\theta}) \ne 0$ for $\theta_1 \ne 0$.
\end{example}

Now, we consider the operator
\begin{equation}
\label{appl_Pauli_eps}
\mathcal{P}_\varepsilon = f^\varepsilon_\times b_\times(\mathbf{D}) g^\varepsilon_\times b_\times(\mathbf{D}) f^\varepsilon_\times = \begin{pmatrix}
P_{-,\varepsilon} & 0 \\
0 & P_{+,\varepsilon}
\end{pmatrix},
\end{equation}
where $P_{+,\varepsilon} = \omega_{-}^\varepsilon \partial_{+} (\omega_{+}^\varepsilon)^2 \partial_{-} \omega_{-}^\varepsilon$, $P_{-,\varepsilon} = \omega_{+}^\varepsilon \partial_{-} (\omega_{-}^\varepsilon)^2 \partial_{+} \omega_{+}^\varepsilon$. If the potential $\mathbf A$ is sufficiently smooth, then the blocks of the operator~\eqref{appl_Pauli_eps} are of the form
$P_{\pm,\varepsilon} = (\mathbf{D}-\varepsilon^{-1}\mathbf{A}^\varepsilon)^2 \pm \varepsilon^{-2}B^\varepsilon$.

Let us consider the Cauchy problem~\eqref{Cauchy_A_eps} with the operator~\eqref{appl_Pauli_eps}. We can apply Theorem~\ref{A_eps_Cauchy_general_thrm}. In general, this result is sharp with respect to smoothness of the initial data and with respect to dependence of the estimates on time.

\end{document}